\documentclass{article}
\usepackage[a4paper]{geometry}

\usepackage{amsmath}
\usepackage{amssymb}
\usepackage{amsthm}
\usepackage{ dsfont } 
\usepackage{tikz}
\usetikzlibrary{decorations.markings}
\usetikzlibrary{decorations.pathreplacing}
\usetikzlibrary{arrows}
\usetikzlibrary{shapes.misc}
\usepackage{hyperref}
\usepackage[normalem]{ulem}

\newenvironment{rules}
{\begin{list}{\(\cdot\)}{
\itemsep=0em
\leftmargin=5mm
\labelwidth=3mm
\labelsep=3mm
}}{\end{list}}

\newtheorem{lemma}{Lemma}[section]
\newtheorem{corollary}[lemma]{Corollary}
\newtheorem{prop}[lemma]{Proposition}
\newtheorem{thm}[lemma]{Theorem}
\newtheorem{thm*}{Theorem}

\theoremstyle{definition}

\newtheorem{Def}[lemma]{Definition}
\newtheorem*{notation}{Notation}
\newtheorem*{convention}{Convention}
\newtheorem{rmk}[lemma]{Remark}

\newtheorem{example}[lemma]{Example}
\newtheorem{standingassumption}[lemma]{Standing Assumptions}

\theoremstyle{remark}

\newcommand{\cuco}[1]{{\mathcal #1}}

\def\RR{\mathbb{R}}
\def\ZZ{\mathbb{Z}}
\def\NN{\mathbb{N}}
\def\H{\mathcal{H}}
\def\F{\mathcal{F}}

\def\1{\mathds{1}}

\def\S{\mathfrak{S}}

\def\MCG{G}
\def\G{G}

\def\GG{\mathcal{G}}
\def\X{\mathcal{X}}

\def\preg{\mathbb{P}}

\let\temp\epsilon
\let\epsilon\varepsilon
\let\varepsilon\temp

\title{Hyperbolic HHS I: Factor Systems and quasiconvex subgroups}
\author{Davide Spriano, ETH Z\"urich, Switzerland.   \\
{\normalsize Address: R\"amistrasse 101, Z\"urich.}\\{\normalsize Email \href{mailto:davide.spriano@math.ethz.ch}{davide.spriano@math.ethz.ch}}}
\date{}
\begin{document}

\maketitle

\abstract{
In this paper we provide a procedure to obtain a non-trivial HHS structure on a hyperbolic space. In particular, we prove that given a finite collection \(\F\) of quasiconvex subgroups of a hyperbolic group \(G\), there is an HHG structure on \(G\) that is compatible with \(\F\).
We will use this to provide explicit descriptions of the Gromov Boundary of hyperbolic HHS and HHG, and we recover results from Hamenst\"adt, Manning, Tran for the case when \(G\) is hyperbolic relative to \(\F\).
Further applications in the construction of new HHG will be presented in a subsequent paper.}

\tableofcontents
\section{Introduction}

The study of hierarchically hyperbolic spaces and groups (respectively HHSs and HHGs) was introduced by Behrstock, Hagen and Sisto in \cite{HHSI}, and led to a number of remarkable results explored, for instance, in \cite{HHSII,HHSAsdim2015,HHSFlats,HHSBoundaries,HagenSusse}.
HHSs and HHGs form a very large class of spaces which contains several examples of interest including mapping class groups, right-angled Artin groups, proper cocompact CAT(0) cube complexes, most 3-manifold groups, Teichm\"uller space (in any of the standard metrics), etc. 
Important properties of HHGs include that, for example, any HHG admits an acylindrical action on a hyperbolic space that depends on the HHG strucutre (\cite{HHSI}). This recovers the result that Mapping Class Groups of non-sporadic surfaces, and  non-cyclic, non directly indecomposable RAAGs are acylindrically hyperbolic (\cite{MasurMinskyI, BowditchTightGeodesics,SistoContracting, KimKoberdaRAAGS, OsinAcyl}). 
Another remarkable result that was unknown before is that, under very mild conditions, every top-dimensional quasi-flat in an HHS lies within  finite distance of a union of standard orthants. This proves open conjectures of Farb in the case of Mapping Class Group, and Brock in the case of Teichm\"uller space.

Work of Bowditch and Osin (\cite{BowditchRelHypGroups, OsinRelHypGroups}) shows that given a hyperbolic group \(G\) and a collection \(\F\) of proper  subgroups, \(G\) is hyperbolic relative to \(\F\) if and only if \(\F\) is an almost malnormal collection of quasiconvex subgroups.

The extra structure on \(G\) given by the family \(\F\) provides useful information on \(G\).
An example of this can be found in the Appendix of \cite{AgolVirtualHaken}, where Dehn fillings in hyperbolic relatively hyperbolic groups are used in a step of the proof of the virtual Haken conjecture.
Moreover, independent works of Manning and Tran (\cite{ManningBowditch_Boundary,TranRelationsBetween}) provide an explicit decomposition of the boundary of \(G\) in terms of \(\F\). We will elaborate on this result later. 

The natural question that arises is: What happens if the malnormality condition is weakened (or removed)? The theory of hierarchically hyperbolic spaces provides some answers in these cases.
We recall that an HHS structure on a quasi-geodesic space \(\X\) consists of a collection of hyperbolic spaces \(\{CU \mid  U \in \S\}\) indexed by a set \(\S\), and projections that relate them, satisfying some axioms and relations (see Section \ref{sec:hierarchically hyperbolic spaces} for more details). An HHG is a group that acts on an HHS in a compatible way. The typical example is given by a group admitting an HHS structure on its Cayley graph (with some addition hypotheses of "equivariance of such a structure").

It is showed in \cite{HHSII} that a space \(G\) which is hyperbolic relative to a (uniform) family \(\F\) of HHSs is itself an HHS. In particular, the result of Bowditch implies that given an almost malnormal family \(\F\) of quasiconvex subgroups of \(G\), the family \(\F\) provides an HHS structure on the Cayley graph of \(G\), where each element \(CU\) is either a coset of an element of \(\F\), or the cone-off of \(G\) with respect to \(\F\). 

In this paper, we show that the malnormality condition can be completely dropped if we are looking for an HHS structure.

\begin{thm*}[\ref{cor: HHG structure on G}]\label{Theorem A introduction}
Let \(G\) be a hyperbolic group and let \(\F= \{F_1, \dots, F_N\}\) be a finite family of infinite quasiconvex subgroups. 
Let \(\sim\) be the equivalence relation between subset of \(G\) given by having finite Hausdorff distance in \(\mathrm{Cay}(G)\).
Then there exists a finite family of quasiconvex subgroups \(\F'\) that contains \(\F\) such that if \(\F'_{\mathrm{cos}}\) is the set of cosets of \(\F'\), then \((G, \F'_{\mathrm{cos}} /_\sim )\) is a hierarchically hyperbolic structure on \(G\), and each  \(CU\) is a cone-off of a coset of an element \(F \in (\F' \cup G)\).  
\end{thm*}
The family \(\F'\) of the Theorem is obtained from the family \(\F\) by considering an appropriate set of intersections of conjugates of the elements of \(\F\).

The key ingredient in the above result is the concept of \emph{weak factor system}. A weak factor system consists of a  family \(\F\) of  uniformly quasiconvex subspaces such that there exist constants \(c,q,\xi, B, D\) such that the following hold:
\begin{enumerate}
	\item Every chain of proper coarse inclusions \(H_n \precnsim \cdots \precnsim H_1\) of elements of \(\F\) has length at most \(c\).
	\item Given \(V, W \in \F\), then either \(\mathrm{diam}_V(p_V(W)) < \xi\), or there exists \(U \in \ F\) such that
		\(d_{\mathrm{Haus}}(U, p_V(W)) \leq B\).
	\item For each \(V\in \F\) and every \(v \in V\) there is an arbitrarily long \(q\)--quasi-geodesic segment with endpoints on \(V\) whose midpoint lies at distance at most \(D\) from \(v\).
\end{enumerate}

The first two condition can be reinterpreted as: the closure process given by coarse projection terminates, and the third is always true in an infinite quasiconvex subgroup of a hyperbolic group. 
In the case of a finite family \(\F\) of quasiconvex subgroup, the first two conditions are basically a consequence of the fact that \(\F\) has \emph{finite height} (see \cite{GMRS1998widths}). Having finite height should be thought as "almost malnormality after finitely many intersections". As mentioned before, almost malnormality in hyperbolic groups implies relative hyperbolicity. In \cite{GradedRelHyp}, Dahmani and Mj explored this analogy and introduced the notion of \emph{graded relative hyperbolicity}, which is the analogous of relative hyperbolicity for a family of subgroups of finite height. 
The definition of (weak) factor system is closely related (and it is basically a special case) to graded relative hyperbolicity. 
A natural question to ask is if there is a relative analogous of Theorem \ref{Theorem A introduction} when \((G, \F)\) has graded relative hyperbolicity.

The first three sections of this paper will investigate the geometry of cone-offs of hyperbolic spaces. If \(\widehat{\X}\) is the cone-off of \(\X\) with respect to a family \(\F\), given a path \(\widehat{\gamma}\) of \(\widehat{\X}\), we can obtain a path \(\gamma \) of \(\X\) substituting each connected component of \(\gamma - X\) with a geodesic of \(X\) connecting the endpoints of such connected component.
Such a path \(\gamma\) is called a \emph{de-electrification} of \(\widehat{\gamma}\). 

We show (\ref{Bound Hausdorff distance}) that if the family \(\F\) satisfies some mild hypotheses, which are implied by the fact that \(\F\) consists of uniformly-quasiconvex subsets of \(\X\), then for every pair of points \(x,y\) of \(X\), there is a \(\tau_1\)--quasi-geodesic \(\widehat{\gamma}\) of \(\widehat{\X}\) connecting them, such that its de-electrification is a \(\tau_2\)--quasi-geodesic of \(X\). 
An analogous result was previously established by Hamenst\"adt in the case where \(X\) is hyperbolic relative to \(\F\) (\cite{Hamenstadt_Hyp_rel_hyp_graphs}).

Using this result (although we will need it in full generality only in a subsequent paper), and other considerations, we will pin down a handful of properties that the family \(\F\) needs to satisfy in order for \((\X,\F)\) to be an HHS. In this case we say that \(\F\) is a factor system.
In Section \ref{sec: Obtaining a factor system}, we will show that such properties can be weakened to a weak-factor system, adding the extra cost of an equivalence relation, as can be seen in Theorem \ref{Theorem A introduction}.

As an application, in the last section we provide an explicit description of the Gromov boundary for a hyperbolic HHS. It was proved in \cite{HHSBoundaries} that for a hyperbolic HHS \((\X, \S)\) the following holds:
\[\partial \X = \bigcup_{U \in \S} \partial CU,\]
where \(\partial Z\) denotes the Gromov boundary of \(Z\), for a hyperbolic space \(Z\). 
Combining this with our construction, we get a very explicit description of the Gromov boundary in a variety of cases. Indeed, given a (weak) factor system \(\F\) on a hyperbolic space \(\X\), we can decompose the boundary of \(\X\) as the union of boundaries of the various cone-offs with respect to elements of \(\F\).

\begin{thm*}[\ref{cor:bdry_dec_for_gps}]
Let \(G\) be a hyperbolic group and let \(\F\) be a finite family of quasiconvex subgroups of \(G\). Then there is a family \(\F'\) such that
\[\partial G = \bigcup_{U \in \S} \partial CU,\]
where each \(CU\) is as in Theorem \ref{Theorem A introduction}. 
\end{thm*}

This recovers two results of Hamenst\"adt (\cite{Hamenstadt_Hyp_rel_hyp_graphs}), namely the case when \(\X\) is hyperbolic relative to \(\F\), and the case of the disk graph of a handlebody. Indeed, the proof of the latter implies that the family of electrified disk graphs of a certain class of subsurfaces of the boundary handlebody is a factor system for the disk graph of the handlebody.

In the case where \(\X\) is hyperbolic relative to \(\F\), we also recover a result proved independently by Manning and Tran (respectively, \cite[Theorem 1.3]{ManningBowditch_Boundary}, \cite[Section 6]{TranRelationsBetween}). Namely, in this case the Bowditch boundary of \((\X,\F)\) can be expressed as the quotient 
\[\partial (\X, \F) = \partial \X /_\sim,\]
where \(x \sim y\) if there is \( U \in \F\) such that \(x, y \in \partial CU\). 

In a follow-up paper, we will use Theorem \ref{Theorem A introduction} to prove a combination Theorem for HHG. 
Indeed, the flexibility coming from the mild hypotheses on the family \(\F\) allows to turn a large class of graphs of groups into graphs of HHG. 
Roughly speaking, given a graph \(\GG\) of HHG, we can add a new vertex \(v\) with vertex group a hyperbolic group \(G\). 
Then, the HHG structure on \(G\) will be the one induced choosing \(\F\) to be the set of images of the edge groups adjacent to \(G\). 
As an application we get that if \(H\) is a quasiconvex subgroup of a hyperbolic group \(G\), and if, moreover, \(H\) is hyperbolically embedded in the Mapping Class Group of a surface \(\Sigma\), then \(G \ast_{H} \mathrm{MCG}(\Sigma)\) is an HHG.

\subsection*{Outline}
Subsections \ref{subsec: Approximating metric spaces}, \ref{subsec:Basics on hyp sp} and \ref{subsec: Behrstock's inequalities}, contain some background on hyperbolic spaces and approximating graphs, subsection \ref{subsec:coning-off} contains the bulk of the theory on the cone-off procedure that is going to be used in Sections \ref{sec:hierarchically hyperbolic spaces} and \ref{sec: Obtaining a factor system}. 
Section \ref{sec:hierarchically hyperbolic spaces} contains the definition of factor system (\ref{Factor System}) and the proof of the fact that a factor system on a hyperbolic space determines a non-trivial HHS structure (Theorem \ref{thm:main result for factor system}).
Section \ref{sec: Obtaining a factor system} provides some coarser conditions (\ref{def: Weak factor system}) that are sufficient to obtain an HHS structure (Theorem \ref{thm: Main result for weak factor systems}). 
In Section \ref{section:Groups with quasiconvex}, we prove that a finite family of infinite quasiconvex subgroups induces a weak factor system for a hyperbolic group, providing a variety of new HHG structures on hyperbolic groups (Corollary \ref{cor: HHG structure on G}).
In Section \ref{sec:boundaries}, we use the HHS structure to give an explicit description of the Gromov boundary of \(X\), where \(X\) is a hyperbolic space equipped with a factor system (Theorem \ref{thm:main_result_bdry}). We use this result to give an explicit description of the Bowditch boundary of a hyperbolic relative hyperbolic space (Theorem \ref{thm:Bowditch Boundary}).

\subsection*{Acknowledgments}
The author would like to thank Alessandro Sisto for suggesting the topic and for very helpful comments and suggestions, and Dominik Gruber for very helpful conversations on the geometry of Cayley graphs.

\section{Introductory tools}\label{sec:introductory tools}

\subsection{Approximating metric spaces and standing conventions}\label{subsec: Approximating metric spaces}
It is a well known fact that every quasi-geodesic metric space (Definition \ref{def:quasi-geodesic-metric space}) is quasi-isometric to a geodesic one. Moreover, working in the setting of geodesic metric spaces allows to simplify several arguments.
Throughout the paper, we will assume that that all geodesic metric spaces are geodesic and that all quasi-geodesics are continuous and rectifiable. The goal of this subsections is to justify such assumption and set up some notations. Those are standard facts which are included for the sake of completeness. The expert reader can confidently skip this part. 

We recall the definition of quasi-geodesic metric space.

\begin{Def}[Quasi-geodesic metric space]\label{def:quasi-geodesic-metric space}
A metric space \(X\) is \emph{\((C, \epsilon)\)--quasi-geodesic} if any two points can be joined by a \((C, \epsilon)\)--quasi-geodesic. A metric space \(X\) is a \emph{quasi-geodesic metric space} if there exists \((C, \epsilon)\) such that \(X\) is a \((C, \epsilon)\)--quasi-geodesic metric space.
\end{Def}

\begin{Def}[Quasiconvex subset]\label{def: quasiconvex subset}
Let \(X\) be a metric space, \(Y\) be a subset of \(X\) and let \(K \colon \RR^2 \to \RR\) be a non-decreasing function. We say that \(Y\) is \emph{$K$--quasiconvex} if for every pair of points \(x,y \in Y\) and \((C,\epsilon)\)--quasi-geodesic \(\gamma\) between them, we have that \(\gamma\) is contained in the \(K(C,\epsilon)\)--neighborhood of \(Y\). The map \(K\) is called the \emph{quasiconvexity gauge} of \(Y\).
\end{Def}

\begin{lemma}[Approximation Lemma]\label{lem: approximation Lemma}
Let \(X\) be a quasi-geodesic metric space and \(\F\) a family of \(K\)--quasiconvex subsets of \(X\).
Then there exists a connected graph \(\Omega(X)\) and a quasi-isometry \(\omega \colon X \to \Omega(X)\) such that for each \(W \in \F\), the image \(\omega(W)\) is connected and quasi-isometric to \(W\), where \(W\) is equipped with the restriction of the metric on \(X\), and \(\omega(W)\) with the graph metric. 
Moreover, suppose that a group \(G\) acts coboundedly on \(X\) and suppose that the family \(\F\) is \(G\)--invariant and has finitely many \(G\)--orbit. Then the map \(\omega\) can be chosen to be \(G\)--equivariant. 
\end{lemma}

The proof is different for the non-equivariant and the equivariant case. We start by setting some notation.

\begin{notation}
Let \(X\) be a metric space. For a subset \(Y\) of \(X\), we define \(p_Y \colon  X \rightarrow 2^Y\) to be the shortest distance projection from \(X\) to \(Y\). 
\end{notation}

Note that for a general subspace \(Y\), the projection \(p_Y (x)\) of a point may be empty.
In order to avoid this, we will abuse notation and denote in the same way the \emph{\(\epsilon\)--projection} \(p_Y\) which is defined as \(p_Y(x)=\{y \in Y \mid  d(y,x) - d( Y,x) < \epsilon\}\) for some very small \(\epsilon\).
From now on, with an abuse, we will implicitly assume that all projections are, in fact, \(\epsilon\)--projections.

\begin{Def}[Maximal nets and approximation graphs]\label{defn: approximation graph}
	Let \(X\) be a metric space, and let \(\zeta >0\). A \emph{\(\zeta\)--net} for \(X\) is a collection \(N\) of points of \(x\) such that for each pair of 
	elements \(x,y\) of \(N\), we have \(d(x,y) \geq k\). A \(\zeta\)--net is \emph{maximal} if it is not possible to add more points to it. 
	
	Given a metric space \(X\), a \(\zeta\)--net \(N\) and \(\lambda>0\), the \emph{\((\zeta, \lambda)\)--approximation graph} on \(N\) is the graph \(\Omega (X)\)
	whose vertex set is the set \(N\) with the condition that two vertices are connected
	 if and only if their distance as points of \(X\) is at most \(\lambda\).
	
	Let \(\omega \colon X \rightarrow 2^{\Omega(X)}\) be the map that associates to each point of \(X\) its closest point projection on \(N\), which in general is a set of vertices of \(\Omega(X)\). For each subspace \(Y \subseteq X\), we denote by \(\Omega(Y)\) the subgraph of \(\Omega(X)\) induced by \(\omega(Y)\). 
\end{Def}

\begin{convention}
When dealing with a graph, unless differently specified, we will consider the standard graph metric. This also applies to subgraphs of a given graph. For this reason, when considering subgraphs we will try to ensure that they are connected.
\end{convention}

\begin{proof}[Proof of Lemma \ref{lem: approximation Lemma}, non equivariant case.]

 \label{some inequality}
	Let  \(X\) be a \((C, \epsilon)\)--quasi-geodesic metric space and 
	let \(\F\) be a family of \(K\)--quasiconvex subsets of \(X\).
	Let \(\zeta = \max\{C + \epsilon, K(C,\epsilon)+1\}\), \(N\) be a \(\zeta\)--maximal net for \(X\) and 
	 \(\Omega(X)\) be a \((\zeta,5\zeta )\)--approximation graph for \(X\) obtained
	from \(N\).
	We claim that for each \(W \in \F \cup \{X\}\) and \(x,y \in W\), we have:
	\[\frac{1}{5\zeta}d_X(x,y) -5\zeta \leq d_{\Omega(W)}(\omega(x), \omega(y)) \leq C d_X(x,y) + \epsilon +1,\] where \(d_{\Omega(W)}\) denotes the graph metric on \(\Omega(W)\).
	
	Let \(\eta\colon [a,b] \rightarrow X\) be a \((C,\epsilon)\)--quasi-geodesic in \(X\) between \(x\) and \(y\). 
	We observe that \(\eta([a,b]) \subseteq N_{2\zeta} (p_N(W))\), where \(p_{N(W)}\) represents the closest point projection of \(W\) on the net \(N\). 
	Indeed, by quasiconvexity, each point of \(\eta([a,b])\) lies at distance at most \(K(C,\epsilon)\)
	from a point of \(W\), and every point of \(W\) lies at distance at most \(\zeta\) from \(N\), since \(N\) is a maximal \(\zeta\)--net.
	
	Let \(a=t_0 < \cdots < t_s=b\) be a minimal (i.e. $s\leq b-a +1$) partition of \([a,b]\) satisfying \(|t_i - t_{i+1}|\leq 1\). This implies \(d(\eta(t_i), \eta(t_{i+1}))\leq C +\epsilon \leq \zeta\),  for all \(i < s\). Choose representatives 
	\(v_i \in \omega(\eta(t_i))\), setting \(v_0 = x'\), \(v_s = y'\). We claim that  \(v_0 \cdots v_s\) is a path in \(\Omega(W)\) joining \(x'\) and \(y'\).
	Indeed, observe that the \(X\)--distance between two consecutive \(v_i\) and \(v_{i+1}\) is at most \(5\zeta\), meaning that the pair \(v_i, v_{i+1}\) forms an edge of \(\Omega (W)\). In particular \(d_{\Omega(X)}(x',y') \leq s\).
	 Since \(s \leq b-a +1\) and since \(\eta \colon [a,b] \rightarrow X\) is a \((C,\epsilon)\)--quasi-geodesic, we get 
	 \(d_{\Omega(W)} (\omega(x), \omega(y))\leq Cd_X(x,y) + \epsilon + 1\). 
	 \sloppy
	 On the other hand, by triangular inequality we have \(d_{X}(x, y) \leq 5\zeta d_{\Omega(W)}(\chi(x), \chi(y)) + 5\zeta\). 
	
	This shows that the map \(\omega \colon X \to \Omega(X)\) is a quasi-isometry, and it restricts to a uniform quasi-isometry for each \(W \in \F\). Moreover, since being quasiconvex is a quasi-isometric invariant property, the result follows. 
\end{proof}
\label{cor:omega(W)q.i.emb}
\begin{proof}[of Lemma \ref{lem: approximation Lemma}, equivariant case.]
By assumption, the group \(G\) acts coboundedly on \(X\). This allows us to apply a standard Milnor-Schwarz argument and obtain that there exists a (possibly infinite) generating set for \(G\) and a Cayley graph \(\Omega(X)\) of \(G\) which is quasi-isometric to \(X\).

More precisely, let \(x\) be a point of \(X\) and let \(N = G\cdot x\). Since the action of \(G\) is cobounded, there exists \(D\) such that \(X\) is contained in the \(D\)--neighbourhood of \(N\).
Let \(\zeta =  \max\{C + \epsilon, K(C,\epsilon)+1, D\}\) and let \(\Omega(X)\) be the graph with vertex set \(N\) an with and between two points \(gx\) and \(hx\) if and only if \(d_{X}(gx, hx) \leq 5\zeta\). Following the same reasoning as before, we obtain that \(\Omega(X)\) is quasi-isometric to \(X\). Moreover, since the vertex set is an orbit of \(G\) in \(X\) and since \(G\) acts by isometries on \(X\), we obtain that \(G\) acts by isometries on \(\Omega(X)\) as well. Thus there is an equivariant quasi-isometry \(\Omega(X) \to X\) given by inclusion on the level of vertices. Moreover, it is not hard to see that \(\Omega(X)\) is equivariantly quasi-isometric to  a Cayley graph for \(G\) with generating set \(\{g \in G \mid d_{X}(x, gx) \leq 5\zeta\}\), and the isometry is obtained by choosing a representative of the stabilizer of \(x\). 
\end{proof}

We recall the definition of length of a curve (see, for instance, \cite[Chapter I.1, Definition 1.18]{BridsonHaefliger}
\begin{Def}
Let \(X\) be a metric space. The \emph{length} of a curve \(\gamma \colon [a,b]\to X\) is:
\[L(\gamma) = \sup_{a=t_0 \leq t_1 \leq \cdots \leq t_n = b} \sum_{i=0}^{n-1}d_X(\gamma(t_i), \gamma (t_{i+1})),\]
where the supremum is taken over all finite partitions (no bound on \(n\)) with \(a=t_0 \leq t_1 \leq \cdots \leq t_n = b\).
If \(L(\gamma) < \infty\) we say that \(\gamma\) is \emph{rectifiable}.
\end{Def}

\begin{lemma}[Taming quasi-geodesic, {\cite[Chapter III.H, Lemma 1.11]{BridsonHaefliger}}]\label{lem: Taming quasi-geodesics}
Let \(X\) be a geodesic metric space. Given any \((C, \epsilon)\)--quasi-geodesic \(\gamma \colon [a,b] \to X\), one can find a continuous \((C, \epsilon')\)--quasi-geodesic \(\gamma'\colon [a,b] \to X\) such that: \begin{enumerate}
    \item \(\gamma(a)= \gamma'(a)\) and \(\gamma(b) = \gamma'(b)\);
    \item \(\epsilon'= 2(C+ \epsilon)\);
    \item \(L(\gamma'_{|[t, t']}) \leq k_1 d(\gamma'(t), \gamma'(t')) + k_2\), where \(k_1= C (C+\epsilon)\) and \(k_2= (C\epsilon'+3)(C + \epsilon)\);
    \item The Hausdorff distance between the images of \(\gamma\) and \(\gamma'\) is less than \(C+\epsilon\).
\end{enumerate}
\end{lemma}

\begin{standingassumption} From now on we will assume the followings.
\begin{enumerate}
    \item We use the approximation Lemma \ref{lem: approximation Lemma} to assume that metric spaces are geodesic. 
    \item We assume that all quasi-geodesics are as in Lemma \ref{lem: Taming quasi-geodesics}. In particular, we can define the length of quasi-geodesics and we will freely identify (quasi-)geodesics with their images.  Given two points \(x,y\) in \(X\), we denote by \([x,y]\) a geodesic segment connecting them.  
\end{enumerate}
\end{standingassumption}

\subsection{Basics on hyperbolic spaces}\label{subsec:Basics on hyp sp}

When considering Gromov-hyperbolic spaces, quasiconvex subsets can be easily described in terms of geodesics. Indeed, suppose that \(X\) is a geodesic \(\delta\)--hyperbolic space and \(Y\subseteq X\) is a subset with the property that there exists \(k\) such that all geodesic with endpoints on \(Y\) are contained in  the \(k\)--neighbourhood of \(Y\). Then there exist a function \(K\), determined by \(\delta\) and \(k\), such that \(Y\) is \(K\)--quasiconvex in the sense of Definition \ref{def: quasiconvex subset}.
We will slightly abuse notation and denote \(k = K(1,0)\) simply with \(K\).
The above is a consequence of the, so called, Morse Lemma for hyperbolic spaces, that we recall. 

\begin{thm}[Morse Lemma, {\cite[Chapter III.H, Theorem 1.7]{BridsonHaefliger}}]\label{lem:Morse Lemma}
Let \(X\) be a geodesic Gromov hyperbolic space. Then there exists a function \(Q\colon  \RR^2\rightarrow \RR\) such that for each geodesic \(\alpha\) and \((C,\epsilon)\)--quasi-geodesic \(\gamma\) with the same endpoints of \(\alpha\), we have:
\[d_{\mathrm{Haus}}(\gamma, \alpha) \leq Q(C, \epsilon).\] 
\end{thm}

\begin{Def}[Projection geodesics]
	Let \(X\) be a geodesic metric space and \(Y \subseteq X\) be a subspace of \(X\). 
	We say that a  geodesic \(\gamma\) with endpoints \(x,y\) is a 
	\emph{projection geodesic} for \(Y\) if \(y \in p_Y(x)\). 
\end{Def}
The following lemma is well-known. We recall it as it will be used extensively.

\begin{lemma}[Quadrilateral argument]\label{quadrilateral}
	Let \(X\) be a  geodesic \(\delta\)--hyperbolic space, and let \(H\) be a \(K\)--quasiconvex subspace. 
	Let \(\gamma, \gamma'\) be projection geodesics for \(H\) 
	and let \(a,b,a',b'\) be the endpoints. Up to changing names, we may assume \(b,b' \in H\). 
	Consider the geodesic quadrilateral \(\gamma [b, b'] \gamma' [a',a]\), and let \(I\) be the set of points in the image of \([b,b']\) consisting of
	 points at distance at least \(4\delta +K\) from \(\{b \cup b'\}\).
	Then for each \(s\in I\), we have that \(d(s, [a,a']) < 2\delta\).
	\begin{proof}
		We may assume that \(I\) is non empty, otherwise the lemma is trivially true. 
		Let \(s \in I\) and consider the diagonal \([b,a']\). By hyperbolicity, \(d(s, [b,a'] \cup [a,a']) < \delta\). 
		If \(d(s, [a,a']) < \delta\), then we would get a contradiction. In fact, let \(m\) be a point in the image of \([a,a']\)  witnessing the distance. 
		Then, by triangular inequality, \(d(m, a') \geq d(s, a') - d(a', m) > 3\delta + K\). But since \(H\) is \(K\)--quasiconvex, there is a point
		\(t\in H\) at distance less than \(K\) from \(s\). Then \(d(m,t) < d(m,a')\) contradicting \(\gamma'\) being a projection geodesic. 
		Hence there is a point \(q \in [b,a']\) with \(d(q, s) < \delta\). The same argument on the triangle \(\gamma [b,a'] [a'a]\) shows that 
		\(d(q, [a',a])< \delta\), and hence \(d(s,[a',a]) < 2\delta\). 
	\end{proof}
\end{lemma}

\begin{lemma}[Closest point projections are quasi-Lipschitz]\label{projection q.Lips} 
Let \(X\) be a \(\delta\)--hyperbolic space and \(H\) a \(K\)--quasiconvex subspace. Then there exists \(\rho = \rho (K,\delta)\) such that the map \(p_H\) is \((1, \rho)\)--quasi-Lipschitz.
	\begin{proof}
		Consider \(x,y\) in \(X\), and their projections \(p_H (x)\), \(p_H(y)\) on \(H\). 
		Pick points \(p \in p_H (x)\), \(q \in p_H(y)\) realizing \(d(p_H (x), p_H (y))\). 
		Consider the sub-interval \(I\) of a geodesic \([p,q]\) consisting of points that have distance at least \(4\delta + K\) from both \(p\) and \(q\). 			A hyperbolic quadrilateral argument (Lemma \ref{quadrilateral}) shows that for each \(s \in I\), \(d(s, [x,y]) < 2\delta\). In particular, if we 
		consider the endpoints of \(I\) and choose points \(u,v \in [x,y]\) realizing the minimal distance, we get \(|L(I) - d(u,v)| \leq 4\delta\). 
		In particular \( d(p_H(x), p_H(y)) \leq d(x,y) + 12\delta + 2K\). 
	\end{proof}
\end{lemma}

\begin{corollary}\label{nbhd inclusion of proj}
	Let \(X\) be \(\delta\)--hyperbolic and \(H \subset X\) be \(K\)--quasiconvex. Then for each \(R\) there exists \(S= S(R, K, \delta)\) such that for each 	
	quasiconvex \(Y \subseteq X\), we have \(p_H(N_R(Y)) \subseteq N_S (p_H (Y))\). 
	\begin{proof}
		This is a easy consequence of the fact that the map \(p_H\) is quasi-Lipschitz.
		In fact, let \(x \in N_R(Y)\), and let \(y \in Y\) such that \(d(x,y) < R\). 
		Then, by Lemma \ref{projection q.Lips}, \(d(p_H(x), p_H(y)) < R + 12\delta + 2K\).
		Since \(\mathrm{diam}(p_H (x))\) is uniformly bounded, we get the claim.
	\end{proof}
\end{corollary}

\subsection{Behrstock's inequalities}\label{subsec: Behrstock's inequalities}

\begin{prop}\label{first B. inequality}
Let \(X\) be a \(\delta\)--hyperbolic space and \(V,W\) be \(K\)--quasiconvex subspaces of \(X\).
Then there exists \(\kappa_1= \kappa_1 (K, \delta)\) such that for each \(x\) in \(X\)
\[\mathrm{min} \{d(p_V (x), p_V (W)), d(p_W(x), p_W(V))\} \leq \kappa_1.\]
	\begin{proof}
		Suppose that \(d(p_W(x), p_W(V)) > 8\delta + 2K\). We claim that this will imply a uniform bound on \(d(p_V(x), p_V(W))\).
		Choose points \(a\in p_V (x)\),  \(b \in p_W(a) \subseteq p_W(V)\) and \(c\in p_W (x)\). By assumption, \(d(b,c) > 8\delta +2K\).
		Consider a geodesic quadrilateral between \(x,a,b,c\). Observe that \([x,a], [x,c]\) and \([a,b]\) are projection geodesics.
		Since \(L([b,c]) > 8\delta + 2K\), a quadrilateral argument (Lemma \ref{quadrilateral}) gives that there exists a (non-empty) sub-interval 
		\(I\) of \([b,c]\)
		with \(d(I, [x,a]) < 2\delta\). Quasiconvexity of \(W\) gives \( [x,a] \cap N_{K+2\delta +1} (W) \neq \emptyset\).
		Since \([x,a]\) is a projection geodesic for \(V\), we have that \(p_V (x) \cap (N_{K + 2\delta + 1} (W))\neq \emptyset\). 
		Lemma \ref{nbhd inclusion of proj} gives a uniform \(\kappa'= \kappa'(K + 2\delta +1,K,\delta)= \kappa' (K, \delta)\) such that
		\(p_V (N_{K+ 2\delta + 1} (W)) \subseteq N_\kappa (p_V (W))\). In particular, \(p_V (x) \cap N_\kappa (p_V (W))\neq \emptyset\). 
		Since \(\kappa_1 =\mathrm{max}\{\kappa', 8\delta + 2K\}\) depends only on \(K\) and \(\delta\), we get the claim.
	\end{proof}
\end{prop}

\begin{prop}\label{second B. inequality}
Let \(X\) be a \(\delta\)--hyperbolic space and \(V \subseteq W\) be \(K\)--quasiconvex subspaces of \(X\). Then 
\[\mathrm{diam} (p_V(x) \cup p_V(p_W (x))) \leq 12\delta + 4K.\]
	\begin{proof}
		Let \(a,b \in p_V(x) \cup p_V(p_W (x))\) be witnessing the diameter. We can assume that \(a \in p_V(x)\) and \(b \in p_V (p_W(x))\), 
		otherwise we would have a uniform bound since projections are quasi-Lipschitz. We claim that \(d(a,b) \leq 12\delta +4K\). 
		Suppose that it is not the case.
		Let \(c\in p_W(x)\) be such that \(b \in p_V(c)\). Consider the quadrilateral \(a,b,c,x\). By assumption, there is a point 
		\(s \in [a,b]\) such that \(d(s,a)> 4\delta +K\) and \(d(s,b)> 8\delta + 3K\). 
		Since \([x,a]\) and \([c,b]\) are projection geodesics onto \(V\), we have that \(d(s, [x,c]) < 2\delta\). Let \(m\) be a point on \([x,c]\)
		witnessing the distance. Since \(V \subseteq W\) and \(c \in p_W (m)\), we have that \(d(m,c) < 2\delta + K\).
		By triangular inequality, \(d(c,b) < 4\delta + 2K\).
		Again by triangular inequality, we get that \(d(b,s) < 8\delta +3K\), obtaining a contradiction. 
	\end{proof}
\end{prop}

\subsection{Coning-off}\label{subsec:coning-off}

\begin{Def}[Coning-off]
	Let \(\Gamma\) be a graph, and \(H\) be a connected subgraph of \(\Gamma\). 
	The \emph{cone-off} of \(\Gamma\) with respect to \(H\), denoted by \(\widehat{\Gamma}\),
	is the graph obtained from \(\Gamma\) adding 
	an edge connecting each pair of vertices in \((H \times H) -\Delta_{H\times H}\), where \(\Delta_{H\times H}\) denotes the diagonal.
	We call the edges added in such a way \emph{\(H\)--components}.
	Similarly, the cone-off with respect to a family of connected subgraphs \(\H=\{H_i\}\) is obtained adding the 
	\(H_i\)--components for each \(H_i \in \H\). 
	An edge is an \(\H\)--component if it is a \(H_i\)--component for some \(H_i \in \H\).
\end{Def}

Given a graph \(\Gamma\) and its cone-off \(\widehat{\Gamma}\) with respect to some family \(\H\), we have that \(V(\Gamma) = V(\widehat{\Gamma})\).
In particular, if we regard \(\Gamma\) and \(\widehat{\Gamma}\) as metric spaces, this implies that there is a distance-non-increasing bijection \(i \colon \Gamma \rightarrow \widehat{\Gamma}\).
To simplify notation, we will identify points of \(\Gamma\) and \(\widehat{\Gamma}\) and simply write \(d(x,y)_\Gamma \geq d_{\widehat{\Gamma}}(x,y)\).

The goal of this section is to understand the relation between \(\Gamma\) and \(\widehat{\Gamma}\) through the bijection \(i\). The first observation is that \(i\) sends paths of \(\Gamma\) to paths of \(\widehat{\Gamma}\), but the converse does not hold. In order to solve this, we introduce a standard way to obtain a path in \(\Gamma\) from a path in \(\widehat{\Gamma}\).

\begin{notation}
Recall that for a geodesic metric space \(X\) and points \(x,y \in X\) we denote with \([x,y]\) a geodesic in \(X\) connecting \(x\) and \(y\). When working with a graph \(\Gamma\) and its cone-off \(\widehat{\Gamma}\), we denote with \([x,y]\) a geodesic in the space \(\Gamma\) (and not in \(\widehat{\Gamma}\)).
\end{notation}

\begin{Def}[Pieces]
	Let \(\eta= \eta_1 * \cdots * \eta_n\) be a concatenation of geodesic segments. 
	Then we call each of the non-trivial \(\eta_i\) a \emph{piece} of \(\eta\).  
\end{Def}

Note that, in general,  a path \(\eta\) can be expressed in different ways as the concatenations of its geodesic sub segments. For this reason, the definition of piece depends on the chosen subdivision of the path \(\eta\).

\begin{Def}[De-electrifications]
	Let \(\widehat{\Gamma}\) be the cone-off of a graph \(\Gamma\) with respect to a family of subgraphs \(\H\). Let 
	\(\gamma = u_1 *e_1 * \cdots * e_n * u_{n+1}\) be a path of \(\widehat{\Gamma}\), where the \(e_i\) are \(\H\)--components
	and the \(u_i\) are (possibly trivial) segments of \(\Gamma\). 
	The \emph{total de-electrification} (or simply de-electrification) \(\widetilde{\gamma}\)
	of \(\gamma\) is the concatenation \(u_1 * \eta_1 * \cdots * \eta_n *u_{n+1}\) where each \(\eta_i\) is a geodesic segment of \(\Gamma\) connecting 
	the endpoints of \(e_i\). If \(e_i\) was an \(H\)--component, we say that \(\eta_i\) is an \emph{\(H\)--piece}.
	A piece of \(\widetilde{\gamma}\) is an \emph{\(\H\)--piece} if it is an \(H\)--piece for some \(H \in \H\). 
\end{Def}
Even though the definition of de-electrification makes formal sense for any family of subgraphs \(\H\), in practice we will be interested in the case of  the elements of \(\H\) being quasiconvex. Indeed, this will guarantee that for each \(H\)--component \(e\), the corresponding geodesic segment \(\eta\) will be coarsely contained in \(H\). Without this property (or maybe some different property of the same flavor), there is almost no relation between the original path and the de-electrification of it. However, it is possible to establish at least some mild result about the combinatorial properties of de-electrifications.
\begin{lemma}[Pigeonhole for cone offs]\label{Pigeonhole for cone-offs}
Let \(\Gamma\) be a graph and \(\widehat{\Gamma}\) be the cone off with respect to a family of graphs \(\H\). 
Then for each \(\theta\) there exists a \(T= T(\theta)\) such that if \(d(x,y)_\Gamma \geq T\), then for each 
\(\widehat{\Gamma}\)--path \(\gamma\) connecting \(x\) and \(y\), either \(L_{\widehat{\Gamma}}(\gamma)\geq \theta\) or 
\(\widetilde{\gamma}\) has a \(\H\)--piece that has \(\Gamma\)--length greater or equal \(\theta\).
	\begin{proof}
		We can assume that \(\theta>1\).
		We claim that \(T={2\theta^2}\) does the job. 
		Consider \(x,y\) with \(d(x,y)_\Gamma = T_0 \geq T\), and fix a \(\widehat{\Gamma}\)--path \(\gamma\) between them.
		Assume that \(L_{\widehat{\Gamma}}(\gamma) = \epsilon < \theta\), and let \(P\) be the number of \(\H\)--components of \(\gamma\).
		Consider a de-electrification \(\widetilde{\gamma}\) of \(\gamma\) and let
		\(A = \sum_P L(u_i)_\Gamma\), where the \(u_i\) are the \(\H\)--pieces of \(\widetilde{\gamma}\) and \(B= L(\widetilde{\gamma})-A\).
		Then we have that \(A+B = T_0\) and \(B + P= \epsilon\). Hence \(A = P + T_0 - \epsilon\). Since \(A\) is the sum of \(P\) terms, 
		we have that there is at lest one that has value greater or equal \(1 + \frac{1}{P}(T_0-\epsilon)\). 
		Since \(P \leq \epsilon < \theta\), we get 
		\begin{align*}
		1 + \frac{1}{P}(T_0-\epsilon) \geq 1 + \frac{1}{P}(2\theta^2-\epsilon) \geq 1 + \frac{2\theta^2-\theta}{\theta}\geq \theta.		
		\end{align*}
	\end{proof}
\end{lemma}
It is easily seen that the de-electrification of a geodesics of \(\widehat{\Gamma}\) consists of a concatenation of geodesic segments in \(\Gamma\). We will now record a basic fact about concatenations of geodesic segments. 
\begin{lemma}\label{hyperbolic polygon}
		Let \(X\) be a \(\delta\)--hyperbolic space, \(\eta\) be a geodesic segment of \(X\) and \(\sigma\) the concatenation of \(n\) geodesic segments.
		If \(\sigma \cap N_{(n+1)\delta}(\eta) = \emptyset\), then \(\mathrm{diam}(p_\eta (\sigma)) < 8\delta\).
	\begin{proof}	
		Let \(a,b \in p_\eta(\sigma)\) be two points realizing the diameter, and let \(a',b'\) be preimages in \(\sigma\). 
		Then \(\sigma_{|_{[a',b']}} * [a',b']\) is an \(m\)--gon, with \(m \leq n+1\). In particular, we get that \([a',b']\) is contained
		in the \((n-1)\delta\)--neighborhood of \(\sigma\), and hence \(N_{2\delta} \cap [a',b'] = \emptyset\). A quadrilateral argument gives the claim.
	\end{proof}
\end{lemma}

\begin{prop}\label{19 pieces}
	Let \(\Gamma\) be a \(\delta\)--hyperbolic graph, \(\H\) a family of connected subgraphs and \(\widehat{\Gamma}\) the cone-off of \(\Gamma\) 
	with respect to \(\H\).
	Then there exist \(D', p\) depending only on \(\delta\) such that the following holds.
	For each pair of points \(x,y\) of \(\Gamma\), for each geodesic \(\gamma\) of \(\widehat{\Gamma}\) connecting \(x\) and \(y\) and for each geodesic \([x,y]\)
	of \(\Gamma\) connecting \(x\) and \(y\), every connected component of \(\widetilde{\gamma} - N^\Gamma_{D'}([x,y])\) has at most \(p\) pieces. 
\begin{proof}
	To simplify notation, throughout this proof we will drop the superscript \(\Gamma\) and assume that \(N_D(Y)\) denotes the 
	\(D\)--neighborhood of \(Y\) with respect to the metric of \(\Gamma\). 
	Let \(\xi = 8\delta +1\), \(D'= \delta(\xi +1)\), and suppose that \(\widetilde{\gamma}\) leaves the \(D'\)--neighborhood of \([x,y]\).
	Let \(a,b\)	be the endpoints of one of the connected components of \(\widetilde{\gamma} - N_{D'}([x,y])\). To simplify notation
	let \(\sigma = \widetilde{\gamma}_{|_{[a,b]}}\). Let \(P\) be the number of pieces of \(\sigma\), and let \(q\) be the maximal integer 
	such that \(P= q\xi+ r\), for some \(r< \xi\). 
	We subdivide the concatenation \(\sigma\) into sub-concatenations \(\sigma_i\) 
	such that:
	\begin{rules}
		\item each \(\sigma_i\) is the concatenation of consecutive pieces of \(\sigma\);
		\item each \(\sigma_i\) contains at most \(\xi\) pieces of \(\sigma\);
		\item the subdivision is chosen in such a way that the number of \(\sigma_i\) is minimal (in fact, it is at most \(q+1\)).
		\end{rules}
	To simplify notation, let \(q\leq Q\leq q+1\) be the number of the sub-concatenations \(\sigma_i\).
	Our goal is to give a uniform bound on \(Q\) and hence on the number of pieces.
	Let \(a', b'\) be closest point projections in \(\Gamma\) of \(a\) and \(b\) on \([x,y]\).
	We want to argue that \(L( [a,a'] * [a',b'] *[b', b])_{\Gamma} < L( \gamma_{|_{[a,b]}})_{\widehat{\Gamma}}\). 
	In particular, this will imply that \([a,a']*[a',b']*[b',b]\) is a short cut for \(\gamma\) in \(\widehat{\Gamma}\). 
	Note that \(L([a,a'])_\Gamma \leq D'\) and  \(L([b,b'])_\Gamma \leq D'\). Moreover, 
	\(L([a',b'])_\Gamma \leq \sum_{i=1}^{Q} \mathrm{diam}_\Gamma (p_{[x,y]}(\sigma_i))\). 
	Note, however, that Lemma \ref{hyperbolic polygon} guarantees that for each \(i\leq Q\), 
	\(\mathrm{diam}_\Gamma (p_{[x,y]}(\sigma_i)) < 8\delta < \xi - 1\).
	Hence we have that \(L([a,a'] * [a',b'] *[b', b])_\Gamma < 2\delta (\xi +1) + Q(\xi-1)\). 
	On the other hand, since \(\sigma\) has \(P\) pieces, we have that \(L(\sigma)_{\widehat{\Gamma}} \geq P = Q\xi  + r \geq  Q\xi \). 
	But it is clear that for a large enough  \(Q_0\), the following holds: \(2\delta (\xi +1) + Q_0(\xi-1) < Q_0\xi \). 
	That is, for large enough values of \(Q\), we get a short-cut.
	Hence we get the desired bound on \(Q\), and thus a bound on the maximum number of pieces that \(\sigma\) can have. 
\end{proof}
\end{prop}
\begin{rmk}
	In the above proof, the quantities \(D', p\) depends on \(\xi\), which, ultimately, only needs to satisfy \(\xi -1 > 8\delta\).
	In particular, the bounds on \(p\) and \(Q\) of the above proof are surely not optimal, and it is possible to minimize and explicitly compute them
	via varying \(\xi\). 
\end{rmk}

The proof of Proposition \ref{19 pieces}, gives us two important corollaries:
\begin{corollary}\label{endpoints sigma}
	In the hypotheses and notations of Proposition \ref{19 pieces}, we have the following: There exist  \(D=D(\delta)\) such that 
	for each connected component \(\sigma\) of \(\gamma - N_D([x,y])\), the distance in \(\Gamma\) between the endpoints of \(\sigma\) is uniformly
	bounded by \(2D + 8\delta\).
	\begin{proof}
		Set \(D=\mathrm{max}\{\delta(p+1),D'\} \), where \(p, D'\) are as in Proposition \ref{19 pieces}. Then we have that each connected component 
		\(\sigma\) of 	\(\gamma - N_D([x,y])\) has at most \(p\) pieces. But then Lemma \ref{hyperbolic polygon} gives that the projection 
		of \(\sigma\) on \([x,y]\) has diameter at most \(8\delta\). Hence, by triangular inequality, the distance between the endpoints of \(\sigma\) 
		is at most \(2D + 8\delta\).
	\end{proof}
\end{corollary}
\begin{corollary}\label{cor 19 pieces}
In the hypotheses and notations of the proof of Proposition \ref{19 pieces}, we have that there exists \(D= D(\delta)\) such that \([x,y] \subseteq N_{10\delta + D} (\widetilde{\gamma})\).
	\begin{proof}
	
	Let \(z\) be a point of \([x,y]\). 
	It is a well known fact of \(\delta\)--hyperbolic spaces that, if \(\eta\) is a geodesic and \(\gamma\) a path with the same endpoints of \(\eta\), then 
	the projection of \(\gamma\) onto \(\eta\) is \(2\delta\)--dense.
	Thus, there is a point \(z' \in [x,y]\), with \(d(z,z')_\Gamma <2\delta\) that is contained \(p_{[x,y]} (\gamma)\).
	If there is a preimage of \(z'\) in \(N_D([x,y]) \cap \gamma\), then we are done.
	So suppose this is not the case, and let \(a\) be a preimage of \(z'\) in \(\gamma\). 
	Let \(\sigma\) be the connected component of \(\gamma -N_D([x,y])\) that contains \(a\). 
	Let \(D\) as in Corollary \ref{endpoints sigma}, that is \(D= \mathrm{max}\{\delta (p+1), D'\}\), 
	where \(p, D'\) are as in Proposition \ref{19 pieces}.
	Since \(\sigma\) has at most \(p\) components,  Lemma \ref{hyperbolic polygon} gives that that \(\mathrm{diam}(p_{[x,y]} (\sigma)) < 8\delta\).
	Let \(\sigma^\pm\) be the endpoints of \(\sigma\). Since both \(p_{[x,y]}(\sigma^\pm), z'\) are contained in the projection
	\(p_{[x,y]}(\sigma)\), it follows that \(d(z', p_{[x,y]} (\sigma^\pm)) < 8\delta\).
	By triangular inequality, we get that \(d(z, \sigma^\pm) < 10\delta +D\), which proves the claim. 
	\end{proof}
\end{corollary}

Now we would like to prove the other inclusion, namely that there is a constant \(D'\) such that \(\widetilde{\gamma} \subseteq N_{D'} ([x,y])\). Without further assumptions on the family \(\H\), this is easily seen to be hopeless.
This motivates to consider a family \(\H\) with more structure, in particular, we will require the elements of \(\H\) to satisfy a slightly more general version of uniform-quasiconvexity.

\begin{Def}[Cone-off quasiconvexity]
Let \(\Gamma\) be a connected graph, \(\H\) a family of subgraphs of \(\Gamma\) and \(\widehat{\Gamma}\) the cone-off of \(\Gamma\) with respect to \(\H\).
We say that a subset \(S\subseteq\Gamma\) is \emph{\(K\)--cone-off quasiconvex} with respect to \((\Gamma, \H)\) if the following holds.
For each two points \(s, t\) of \(S\), each geodesic \([s,t]\) of \(\Gamma\) between them and each point \(z \in [s,t]\), we have:
\[d_{\widehat{\Gamma}}(z, S) \leq K.\]
Similarly, we say that the family \(\H\) is \emph{\(K\)--cone-off quasiconvex} (\(K\)--COQC) if for each \(H \) in \( \H\), \(H\) is \(K\)--COQC with respect to \((\Gamma, \H)\). 
\end{Def}
We emphasize that a \(K\)--quasiconvex subset is also \(K\)--COQC, regardless of the coning-off family. 

We will now introduce the definition of interruption.
The idea is the following: consider a path \(\gamma =u_1 e u_2\) of \(\widehat{\Gamma}\) and assume that \(e\) is an \(H\)--component (for simplicity assume the only one) of \(\gamma\), for some \(K\)--quasiconvex \(H\). If \(x,y\) are the endpoints of the edge \(e\), then \(\widetilde{\gamma}= u_1 [x,y] u_2\) is a de-electrification for \(\gamma\),  for some geodesic \([x,y]\) between \(x, y\) in \(\Gamma\).
It may happen that we are interested in some point \(z \in [x,y]\). In general, such a \(z\) will be an element of \(\widetilde{\gamma}\), but not of \(\gamma\). In order to solve this, we will modify \(\gamma\) into a path whose \(\widehat{\Gamma}\)--length is comparable to the one of \(\gamma\), 
but contains \(z\). A pictorial illustration of this can be found in Figure \ref{fig: motivation interruption}.

\begin{figure}[h]
\resizebox{12cm}{6cm}{
\definecolor{qqqqff}{rgb}{0,0,1}
\begin{tikzpicture}[line cap=round,line join=round,>=triangle 45,x=1cm,y=1cm]
\draw [shift={(-0.9318,-3.4966)}] plot[domain=0.98082:2.3610,variable=\t]({1*6.3383*cos(\t r)},{1*6.3383*sin(\t r)});
\draw [shift={(0.8006,-4.9417)}] plot[domain=0.9450:2.5650,variable=\t]({1*4.729725131958314*cos(\t r)},{1*4.7297*sin(\t r)});
\draw [shift={(-4.1352,-0.5876)}] plot[domain=2.268:5.212,variable=\t]({1*2.02389*cos(\t r)},{1*2.02389*sin(\t r)});
\draw [shift={(-0.608,-0.9212)}] plot[domain=-0.04471:0.6987,variable=\t]({1*4.1839*cos(\t r)},{1*4.1839*sin(\t r)});
\draw (-5.5418,-3.8875)-- (-4.2308,-2.3508);\draw (3.2779,-0.6155)-- (6.8497,-1.2635);
\draw [dash pattern=on 1pt off 1pt] (-4.2308,-2.3508)-- (3.2779,-0.6155);
\draw [shift={(-0.4764,-1.4831)},line width=1.7pt]  plot[domain=0.2271:3.3687,variable=\t]({1*3.8533*cos(\t r)},{1*3.8533*sin(\t r)});
\draw [color=qqqqff] (-1.1183,-0.6188)-- (-0.6760,-1.5292);
\draw [shift={(-2.37019584204191,-2.0318223715423795)},line width=1.6pt,color=qqqqff]  plot[domain=0.8457846446972239:3.311380732370041,variable=\t]({1*1.8877573328876038*cos(\t r)+0*1.8877573328876038*sin(\t r)},{0*1.8877573328876038*cos(\t r)+1*1.8877573328876038*sin(\t r)});
\draw [shift={(1.0803501837152183,-1.3234824910444012)},line width=1.6pt,color=qqqqff]  plot[domain=0.31165455260748604:2.8314566805375345,variable=\t]({1*2.3088407094160615*cos(\t r)+0*2.3088407094160615*sin(\t r)},{0*2.3088407094160615*cos(\t r)+1*2.3088407094160615*sin(\t r)});
\begin{scriptsize}
	\draw [fill=black] (-4.2308,-2.350) circle (2.5pt);
	\draw[color=black] (-3.889,-2.1387) node {$x$};
	\draw [fill=black] (3.277,-0.615) circle (2.5pt);
	\draw[color=black] (3.440,-0.906) node {$y$};
	\draw[color=black] (1.2997086534525255,-1.473870176540018) node {$[x,y]$};
	\draw [fill=black] (-0.6760,-1.5292) circle (2.5pt);
	\draw[color=black] (-0.4193,-1.9441) node {$z$};
	\draw[color=black] (-1.0031991000557792,2.4832389210376227) node {$e$};
	\draw [fill=black] (-1.1183,-0.6188) circle (2.5pt);
	\draw[color=black] (-1.1491,-0.9549) node {$z'$};
	\draw[color=qqqqff] (-2.9493183283726565,-0.046716075774311694) node {$e_1$};
	\draw[color=qqqqff] (1.3807969546323955,1.283132063575551) node {$e_2$};
	\draw[color=black] (-5,0) node {\large{$H$}};
	\draw[color=black] (-4.7, -3.4) node {$u_1$};
	\draw[color=black] (5.5, -0.7) node {$u_2$};
	\draw [line width=0.5 pt, black, decorate,decoration={brace,amplitude=8pt,  mirror}] (-0.6760,-1.5292) -- (-1.1183,-0.6188) node[black,midway,xshift=0.7cm, yshift=0.1cm] {$\leq K$ };
\end{scriptsize}
\end{tikzpicture}}
\caption{The point \(z\) is an element of the de-electrification, but not of the original path. To solve this, we substitute \(e\) with the concatenation \(e_1 [z',z][z,z'] e_2\) (blue in the picture). Since the space \(H\) is \(K\)--quasiconvex, the geodesic \([z,z']\) has length smaller or equal \(K\). 
Thus we have that the total length of \(e_1 [z',z][z,z'] e_2\) is bounded above by \(2+2K\).}\label{fig: motivation interruption}
\end{figure}
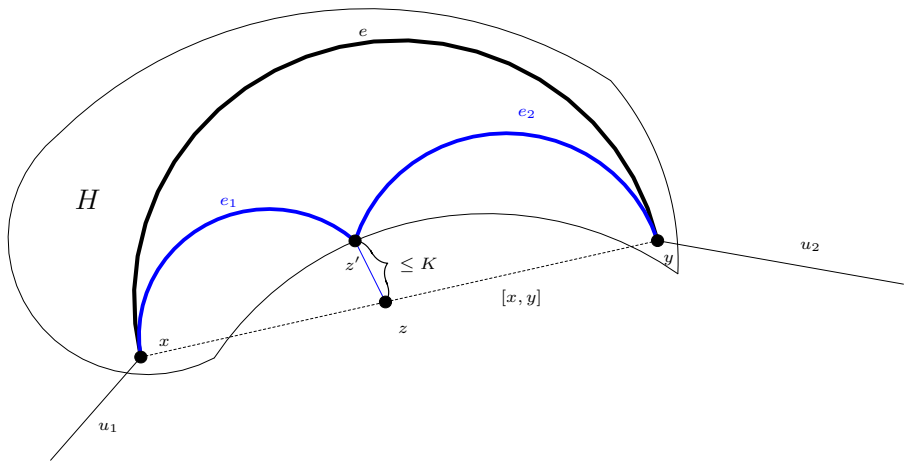

\begin{Def}[Interruption]\label{Def interruption}
Let \(\Gamma\) be a graph, \(\widehat{\Gamma}\) be the cone-off of \(\Gamma\) with respect to a \(K\)--COQC family \(\H\) 
	and let \(\gamma\) be a path of \(\widehat{\Gamma}\). Suppose that \(\gamma = u_1 e u_2\), where \(e\) is a \(\H\)--component for some \(H\in\H\), and
	let	\(u_1 \eta u_2\) be a de-electrification for \(u_1 e u_2\). Let \(z\) be a point in \(\eta\). We define 
	the \emph{interruption of \(\gamma\) in \(z\)} to be the path \(u_1 e_1 [z',z]_\Gamma [z,z']_\Gamma e_2 u_2\) of \(\widehat{\Gamma}\), 
	where  \(z'\) 
	realizes the shortest distance projection of \(z\) on \(H\) and \(e_1, e_2\) are defined as the edges \((e^{-}, z)\), \((z, e^{+})\) respectively.
	If \(\widetilde{\gamma}\) is a de-electrification of \(\gamma\) and \(S\) is a set of points belonging to \(\H\)--pieces of \(\widetilde{\gamma}\), 
	we similarly define the \emph{interruption of \(\gamma\) in \(S\).}
\end{Def}
\begin{rmk}\label{length interruption}
Let \(S\) be as in Definition \ref{Def interruption}.
If \(\gamma^S\) is the interruption of \(\gamma\) with respect to \(S\), we have that \(L\left( \gamma^S\right)_{\widehat{\Gamma}} \leq L\left( \gamma \right)_{\widehat{\Gamma}} + |S|(2K+1)\). This is easily seen because, in the notation of Definition \ref{Def interruption}, we substitute \(\H\) components \(e\) with concatenations of the form \(e_1 [z',z]_\Gamma [z,z']_\Gamma e_2\). By definition of \(K\)--COQC, \(L_{\widehat{\Gamma}}([z,z']) \leq K\).
\end{rmk}
It is quite remarkable that the estimate of Remark \ref{length interruption} is the only step of the proof or Proposition \ref{Bound Hausdorff distance} that uses quasiconvexity. 

We recall that given a graph \(\Gamma\), a point \(p \in \Gamma\) and a number \(R \in \RR\), the \emph{ball of center \(p\) and radius \(R\)},
is the set \(\{x \in \Gamma \mid  d(x,p) < R\}\), and the \emph{sphere of center \(p\) and radius \(R\)} is the set \(\{ x\in \Gamma \mid 
d(x,p) = \lceil R \rceil\}\).

\begin{prop}\label{Bound Hausdorff distance}
Let \(\Gamma\) be a \(\delta\)--hyperbolic graph, \(\H\) a family of \(K\)--COQC subgraphs and \(\widehat{\Gamma}\) the cone-off of \(\Gamma\) with respect to \(\H\). 
Then there exist \(\tau_1 = \tau_1 (\delta, K)\) and \(\tau_2= \tau_2 (\delta, K)\) such that for each pair of points \(x,y \in \Gamma\) there exists a \(\tau_1\)--quasi-geodesic \(\gamma'\) of \(\widehat{\Gamma}\) with the property that for each de-electrification \(\widetilde{\gamma}'\) of \(\gamma'\), 
\(\widetilde{\gamma}'\) is a \(\tau_2\)--quasi-geodesic of \(\Gamma\).
	\begin{proof}
	Consider a geodesic \(\gamma\) of \(\widehat{\Gamma}\) between \(x\) and \(y\) and fix a de-electrification \(\widetilde{\gamma}\). 
	We will stick to the convention that \([x,y]\) denotes a (fixed once and for all) geodesic segment of \(\Gamma\) between \(x\) and \(y\).	
	
	Our goal is to modify \(\gamma\) into a \(\tau_1\)--quasi-geodesic \(\gamma'\) with the same endpoints of \(\gamma\) such that each de-electrification
	of \(\gamma'\) is a \(\tau_2\)--quasi-geodesic, where \(\tau_1\) and \(\tau_2\) do not depend on \(x,y\).
	Since \(\gamma\) is a geodesic of \(\widehat{\Gamma}\), if \(\widetilde{\gamma}\) was already uniformly a quasi-geodesic of \(\Gamma\), 
	then this will conclude the proof. However, it is easily seen that the de-electrification \(\widetilde{\gamma}\)
	can, in general, be arbitrarily far away  in \(\Gamma\) from \([x,y]\). Since \(\Gamma\) is \(\delta\)--hyperbolic, 
	this clearly constitutes an obstruction for \(\widetilde{\gamma}\) to be uniformly a quasi-geodesic. 
	So, our first step will be to modify \(\gamma\) to some quasi-geodesic \(\gamma''\) such that 
	each de-electrification	of \(\gamma''\) is contained in a uniform \(\Gamma\)--neighborhood of \([x,y]\). 
	After this step, it is indeed possible to show that each de-electrification of \(\gamma''\) is a \(\tau_3\)--quasi-geodesic of \(\Gamma\), 
	but \(\tau_3\) will be depending on the distance between the points \(x\) and \(y\). 
	This is because the de-electrifications may have backtracking, which, in general, can only be estimated in terms of \(d_\Gamma(x,y)\).
	So, the second step will be to further modify \(\gamma''\) into a new path \(\gamma'\) such that the de-electrifications will (coarsely) not
	backtrack.
	The path \(\gamma'\) will be the desired one. The last two steps will consist in showing that \(\gamma'\) is still uniformly a quasi-geodesic
	of \(\widehat{\Gamma}\), and that the de-electrifications are uniformly quasi-geodesics of \(\Gamma\).

	\textbf{Step 1:} Obtain that \(\widetilde{\gamma}\subseteq N_{\lceil D+4\delta \rceil} ([x,y])\). 
	
	\begin{figure}[h]
	\begin{center}
	\begin{tikzpicture}[line cap=round,line join=round,>=triangle 45,x=1cm,y=1cm]
\draw [line width=0.6pt] (-4,0)-- (4,0);
\draw [dash pattern=on 1pt off 1pt] (-4,0.75)-- (4,0.75);
\draw [dash pattern=on 1pt off 1pt] (-4,-0.75)-- (4,-0.75);
\draw [shift={(-4,0)},dash pattern=on 1pt off 1pt]  plot[domain=1.5707963267948966:4.71238898038469,variable=\t]({1*0.75*cos(\t r)+0*0.75*sin(\t r)},{0*0.75*cos(\t r)+1*0.75*sin(\t r)});
\draw [shift={(4,0)},dash pattern=on 1pt off 1pt]  plot[domain=-1.5707963267948966:1.5707963267948966,variable=\t]({1*0.75*cos(\t r)+0*0.75*sin(\t r)},{0*0.75*cos(\t r)+1*0.75*sin(\t r)});
\draw [shift={(-4.03591659633896,0.5837078223964038)}] plot[domain=-1.5093420015989958:0.13113327527046542,variable=\t]({1*0.5848117849524986*cos(\t r)+0*0.5848117849524986*sin(\t r)},{0*0.5848117849524986*cos(\t r)+1*0.5848117849524986*sin(\t r)});\draw [shift={(-4.9046601182181435,0.04080474251354631)}] plot[domain=-0.41548670247160935:0.4040582554332149,variable=\t]({1*1.5753961553289493*cos(\t r)+0*1.5753961553289493*sin(\t r)},{0*1.5753961553289493*cos(\t r)+1*1.5753961553289493*sin(\t r)});\draw [shift={(-3.081109930045289,-0.7419667384190645)}] plot[domain=0.6706887823089361:2.77466878604076,variable=\t]({1*0.40944323351594125*cos(\t r)+0*0.40944323351594125*sin(\t r)},{0*0.40944323351594125*cos(\t r)+1*0.40944323351594125*sin(\t r)});\draw [shift={(-5.451088380682982,0.75)}] plot[domain=-0.43106196853185796:0.08451272350666827,variable=\t]({1*2.9616587633625366*cos(\t r)+0*2.9616587633625366*sin(\t r)},{0*2.9616587633625366*cos(\t r)+1*2.9616587633625366*sin(\t r)});\draw [shift={(-2.9020295021109126,1.763829025861302)}] plot[domain=5.196881493181275:6.184433712648291,variable=\t]({1*0.8631700303623696*cos(\t r)+0*0.8631700303623696*sin(\t r)},{0*0.8631700303623696*cos(\t r)+1*0.8631700303623696*sin(\t r)});\draw [shift={(-1.6554111337374706,2.028614259305032)}] plot[domain=3.8758280955114084:5.343490130420632,variable=\t]({1*0.5222027526542912*cos(\t r)+0*0.5222027526542912*sin(\t r)},{0*0.5222027526542912*cos(\t r)+1*0.5222027526542912*sin(\t r)});\draw [shift={(-0.7801668992174987,0.6829426640814864)}] plot[domain=2.121254709370696:3.1758288033992677,variable=\t]({1*1.084210804573463*cos(\t r)+0*1.084210804573463*sin(\t r)},{0*1.084210804573463*cos(\t r)+1*1.084210804573463*sin(\t r)});\draw [shift={(-1.550802200054347,1.21629563481753)}] plot[domain=4.2106440333824535:6.1111284336389975,variable=\t]({1*0.6506625625770598*cos(\t r)+0*0.6506625625770598*sin(\t r)},{0*0.6506625625770598*cos(\t r)+1*0.6506625625770598*sin(\t r)});\draw [shift={(0.24608559532811014,2.1487176500158838)}] plot[domain=3.876112651042882:4.637990207750469,variable=\t]({1*1.5574054974519438*cos(\t r)+0*1.5574054974519438*sin(\t r)},{0*1.5574054974519438*cos(\t r)+1*1.5574054974519438*sin(\t r)});\draw [shift={(0.1748978409819358,0.8666535124377105)}] plot[domain=-1.7337984267030917:0.16300209990819509,variable=\t]({1*0.27467401861417723*cos(\t r)+0*0.27467401861417723*sin(\t r)},{0*0.27467401861417723*cos(\t r)+1*0.27467401861417723*sin(\t r)});\draw [shift={(-0.08118772024898299,0.1883182240518979)}] plot[domain=0.06244295831963052:0.9407651500586107,variable=\t]({1*0.8946801371781796*cos(\t r)+0*0.8946801371781796*sin(\t r)},{0*0.8946801371781796*cos(\t r)+1*0.8946801371781796*sin(\t r)});\draw [shift={(1.2076338886148839,0.646419882264407)}] plot[domain=3.934992011532371:6.01194085142506,variable=\t]({1*0.5644000250045035*cos(\t r)+0*0.5644000250045035*sin(\t r)},{0*0.5644000250045035*cos(\t r)+1*0.5644000250045035*sin(\t r)});\draw [shift={(0.40528490350060825,0.850530965930843)}] plot[domain=3.4915013208930623:6.025103951368576,variable=\t]({1*1.3922219153727449*cos(\t r)+0*1.3922219153727449*sin(\t r)},{0*1.3922219153727449*cos(\t r)+1*1.3922219153727449*sin(\t r)});\draw [shift={(0.0015602150315912648,0.05956934630324587)}] plot[domain=2.8076357169485853:4.4075176969258,variable=\t]({1*0.9570061699709894*cos(\t r)+0*0.9570061699709894*sin(\t r)},{0*0.9570061699709894*cos(\t r)+1*0.9570061699709894*sin(\t r)});\draw [shift={(0.6661536929676898,-2.4316157437844175)}] plot[domain=1.3218970728182158:2.1134824690673906,variable=\t]({1*1.8431220906181545*cos(\t r)+0*1.8431220906181545*sin(\t r)},{0*1.8431220906181545*cos(\t r)+1*1.8431220906181545*sin(\t r)});\draw [shift={(2.356862217442911,-1.8787341143670584)}] plot[domain=1.1673852125235118:2.357504434439636,variable=\t]({1*1.7466414788207454*cos(\t r)+0*1.7466414788207454*sin(\t r)},{0*1.7466414788207454*cos(\t r)+1*1.7466414788207454*sin(\t r)});\draw [shift={(4.35059486468616,0.8703469449138664)}] plot[domain=3.2628955012143463:3.8595915302354475,variable=\t]({1*1.7368656076078013*cos(\t r)+0*1.7368656076078013*sin(\t r)},{0*1.7368656076078013*cos(\t r)+1*1.7368656076078013*sin(\t r)});\draw [shift={(4.652374886131706,3.116167613487951)}] plot[domain=4.02266400148017:4.506017815718893,variable=\t]({1*3.1837232271990206*cos(\t r)+0*3.1837232271990206*sin(\t r)},{0*3.1837232271990206*cos(\t r)+1*3.1837232271990206*sin(\t r)});
\draw [very thin, ->] (-4,0) -- (-4.592648595692134,-0.45963859936273926);
\draw[line width=0.8pt] (-2.49, 0.75 ) ..  controls (-2, 0.40) .. (-1.86, 0.75 );
\begin{scriptsize}
	\draw [fill=black] (-4,0) circle (1pt);
	\draw[color=black] (-3.8,0.17) node {$x$};
	\draw [fill=black] (4,0) circle (1pt);
	\draw[color=black] (4.1280161810558145,0.17014531407230715) node {$y$};
	\draw[color=black] (0.06810802167821486,0.15) node {$[x,y]$};
	\draw[color=black] (-4.1,-0.29) node {$\Delta$};
	\draw[color=black] (-2.7, 0.3) node {$\widetilde{\gamma}$};
	\draw[fill=black] (-2.49, 0.75 ) circle (1pt);
	\draw[color=black] (-2.7,0.95) node{$s_1^-$};
	\draw[fill=black] (-1.86, 0.75 ) circle (1pt);
	\draw[color=black] (-2,0.95) node{$s_1^+$};
	\draw[color=black] (-2, 0.35) node{$\eta_1$};
\end{scriptsize}
\end{tikzpicture}
\end{center}
\caption{The de-electrification \(\widetilde{\gamma}\) may exit the \(\Delta\)--neighborhood of \([x,y]\). 
In this case we will "cut-off" the parts that exits the \(\Delta\)--neighborhood with more than two pieces, and replace them with a geodesic segment \(\Gamma\). In the picture, there is only one such connected component, which is replaced by the segment \(\eta_1\).}
	\end{figure}
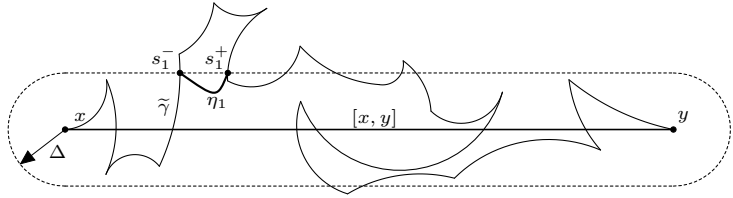	
	
	If \(\widetilde{\gamma} \subseteq N_D([x,y])\), this step is trivial.
	So suppose \(\widetilde{\gamma}\) exits the \(D\)--neighborhood of \([x,y]\), and let \(\Sigma=\{\sigma_\kappa\}\) be the set of connected components
	of \(\widetilde{\gamma} - N_D([x,y])\) that have at least two pieces. We can restrict to such components because if
	a connected component \(\sigma\) consists of only one piece,
	then it is easy to see that \(\sigma \subseteq N_{D + 4\delta} ([x,y])\). In order to simplify notation, 
	let \(\Delta = \lceil D + 4\delta \rceil \in \ZZ\). 
	Suppose, then,  that \(\Sigma\) is non empty.
	Let \(S\) be the set of endpoints of elements of \(\Sigma\) that are not contained in \(\gamma\).
	That is, the set of endpoints that are contained in \(\H\)--pieces of \(\widetilde{\gamma}\).
	We note that \(|S| \leq 2P \leq 2L_{\widehat{\Gamma}}(\gamma)\), where \(P\) is the number of \(\H\)--components of \(\gamma\). 
	Let \(\gamma^S\) be the interruption of \(\gamma\) with respect to \(S\). Note that 
	\(L(\gamma^S)_{\widehat{\Gamma}} \leq L_{\widehat{\Gamma}}(\gamma) + |S|(2K+1)\). 
	For each \(\sigma_\kappa \in \Sigma\), let \(s_\kappa^{\pm}\) be the endpoints of \(\sigma_\kappa\), and let \(\eta_\kappa\) be a geodesic 
	in \(\Gamma\) connecting \(s_\kappa^{+}\) and \(s_\kappa^{-}\).
	Applying Corollary \ref{endpoints sigma} to the endpoints \(s_\kappa^\pm\),  we get:
	\[L\left( \eta_\kappa\right)_\Gamma  \leq 2\Delta - 2.\]
	Let \(\gamma''\) be the path in \(\widehat{\Gamma}\) obtained from \(\gamma^S\) substituting each \(\gamma^S_{|_{[s_\kappa^{+}, s_\kappa^{-}]}}\) with 
	\(\eta_\kappa\), and let \(\widetilde{\gamma}''\) be the path in \(\Gamma\) obtained from \(\widetilde{\gamma}\) substituting all the
	\(\sigma_\kappa\) with \(\eta_\kappa\).
	It is clear from the construction that \(\widetilde{\gamma}''\subseteq N_{\Delta}([x,y])\), and that it is a de-electrification of \(\gamma''\).

	We want now to show that \(\gamma''\) is a quasi-geodesic of \(\widehat{\Gamma}\).
	Consider a subpath \(\chi\) of \(\gamma''\). We want produce a uniform estimate of the length of \(\chi\). Since \(\chi\) is arbitrary, this
	will guarantee that \(\gamma''\) is a quasi-geodesic. First, suppose that the endpoints of \(\chi\) are contained in \(\gamma^S\), that is, they are
	not part of the geodesic segments \(\eta_\kappa\).
	Note that this implies that \(s_\kappa^+\) is a point of \(\chi\)
	if and only if \(s_\kappa^-\) is.
	Finally \(S_\chi = S \cap \chi\), and let \(\chi^S\) be the restriction of \(\gamma^S\) between the endpoints of \(\chi\). 
	\begin{align*}	
	 L (\chi)_{\widehat{\Gamma}} &= L (\chi^S)_{\widehat{\Gamma}} + 
	\sum_\kappa \left(L\left(\eta_\kappa\right)_{\widehat{\Gamma}} - L\left({\chi^S}_{|_{[s_\kappa^{+}, s_\kappa^{-}]}}\right)_{\widehat{\Gamma}}\right) \leq \\
	&\leq L (\chi^S)_{\widehat{\Gamma}} + 
	\sum_\kappa \left(L\left(\eta_\kappa\right)_{\widehat{\Gamma}} \right) \leq \\
	&\leq  d_{\widehat{\Gamma}}(\chi^-, \chi^+) + |S_\chi|(2K+1) + 
	\sum_\kappa \left(L\left(\eta_\kappa\right)_{\Gamma}\right)  \leq \\
	&\leq d_{\widehat{\Gamma}}(\chi^-, \chi^+) +|S_\chi|(2K+1) + |S_\chi| (2\Delta -2) = \\
	&= d_{\widehat{\Gamma}}(\chi^-, \chi^+) + |S_\chi|(2\Delta +2K-1 ).	
	\end{align*}
	Since \(|S_\chi| \leq 2 L(\chi)_{\widehat{\Gamma}}\), we have the desired estimate. 
	Since, by Corollary \ref{endpoints sigma}, we can uniformly bound the length of the segments \(\eta_\kappa\), up to increasing the additive constant, 
	we get the estimate for general endpoints of \(\chi\). 
	
	\textbf{Step 2:} Removing backtracking.
	
\begin{figure}[h]
\begin{center}
\begin{tikzpicture}[line cap=round,line join=round,>=triangle 45,x=1cm,y=1cm]
\draw [line width=0.6pt] (-4,0)-- (4,0);
\draw [shift={(-4.03591659633896,0.5837078223964038)},dotted]  plot[domain=-1.5093420015989958:0.13113327527046542,variable=\t]({1*0.5848117849524986*cos(\t r)+0*0.5848117849524986*sin(\t r)},{0*0.5848117849524986*cos(\t r)+1*0.5848117849524986*sin(\t r)});
\draw [shift={(-4.9046601182181435,0.04080474251354631)},dotted]  plot[domain=-0.41548670247160935:0.4040582554332149,variable=\t]({1*1.5753961553289493*cos(\t r)+0*1.5753961553289493*sin(\t r)},{0*1.5753961553289493*cos(\t r)+1*1.5753961553289493*sin(\t r)});
\draw [shift={(-3.081109930045289,-0.7419667384190645)},dotted]  plot[domain=0.6706887823089361:2.77466878604076,variable=\t]({1*0.40944323351594125*cos(\t r)+0*0.40944323351594125*sin(\t r)},{0*0.40944323351594125*cos(\t r)+1*0.40944323351594125*sin(\t r)});
\draw [shift={(-5.451088380682982,0.75)},dotted]  plot[domain=-0.43106196853185796:0.08451272350666827,variable=\t]({1*2.9616587633625366*cos(\t r)+0*2.9616587633625366*sin(\t r)},{0*2.9616587633625366*cos(\t r)+1*2.9616587633625366*sin(\t r)});
\draw [shift={(-2.9020295021109126,1.763829025861302)},dotted]  plot[domain=5.196881493181275:6.184433712648291,variable=\t]({1*0.8631700303623696*cos(\t r)+0*0.8631700303623696*sin(\t r)},{0*0.8631700303623696*cos(\t r)+1*0.8631700303623696*sin(\t r)});
\draw [shift={(-1.6554111337374706,2.028614259305032)},dotted]  plot[domain=3.8758280955114084:5.343490130420632,variable=\t]({1*0.5222027526542912*cos(\t r)+0*0.5222027526542912*sin(\t r)},{0*0.5222027526542912*cos(\t r)+1*0.5222027526542912*sin(\t r)});
\draw [shift={(-0.7801668992174987,0.6829426640814864)},dotted]  plot[domain=2.121254709370696:3.1758288033992677,variable=\t]({1*1.084210804573463*cos(\t r)+0*1.084210804573463*sin(\t r)},{0*1.084210804573463*cos(\t r)+1*1.084210804573463*sin(\t r)});\draw [shift={(-1.550802200054347,1.21629563481753)},dotted]  plot[domain=4.2106440333824535:6.1111284336389975,variable=\t]({1*0.6506625625770598*cos(\t r)+0*0.6506625625770598*sin(\t r)},{0*0.6506625625770598*cos(\t r)+1*0.6506625625770598*sin(\t r)});\draw [shift={(0.24608559532811014,2.1487176500158838)},dotted]  plot[domain=3.876112651042882:4.637990207750469,variable=\t]({1*1.5574054974519438*cos(\t r)+0*1.5574054974519438*sin(\t r)},{0*1.5574054974519438*cos(\t r)+1*1.5574054974519438*sin(\t r)});\draw [shift={(0.1748978409819358,0.8666535124377105)},dotted]  plot[domain=-1.7337984267030917:0.16300209990819509,variable=\t]({1*0.27467401861417723*cos(\t r)+0*0.27467401861417723*sin(\t r)},{0*0.27467401861417723*cos(\t r)+1*0.27467401861417723*sin(\t r)});\draw [shift={(-0.08118772024898299,0.1883182240518979)},dotted]  plot[domain=0.06244295831963052:0.9407651500586107,variable=\t]({1*0.8946801371781796*cos(\t r)+0*0.8946801371781796*sin(\t r)},{0*0.8946801371781796*cos(\t r)+1*0.8946801371781796*sin(\t r)});\draw [shift={(1.2076339971618337,0.6464227590939446)},dotted]  plot[domain=3.934995449716556:6.011936893961497,variable=\t]({1*0.5644021515821692*cos(\t r)+0*0.5644021515821692*sin(\t r)},{0*0.5644021515821692*cos(\t r)+1*0.5644021515821692*sin(\t r)});\draw [shift={(0.405285597219999,0.8505326245508408)}] plot[domain=3.491502269231681:6.025103065816627,variable=\t]({1*1.392223135651016*cos(\t r)+0*1.392223135651016*sin(\t r)},{0*1.392223135651016*cos(\t r)+1*1.392223135651016*sin(\t r)});\draw [shift={(0.0015602150315912648,0.05956934630324587)}] plot[domain=2.8076357169485853:4.4075176969258,variable=\t]({1*0.9570061699709894*cos(\t r)+0*0.9570061699709894*sin(\t r)},{0*0.9570061699709894*cos(\t r)+1*0.9570061699709894*sin(\t r)});\draw [shift={(0.6661536929676898,-2.4316157437844175)}] plot[domain=1.3218970728182158:2.1134824690673906,variable=\t]({1*1.8431220906181545*cos(\t r)+0*1.8431220906181545*sin(\t r)},{0*1.8431220906181545*cos(\t r)+1*1.8431220906181545*sin(\t r)});\draw [shift={(2.356862217442911,-1.8787341143670584)}] plot[domain=1.1673852125235118:2.357504434439636,variable=\t]({1*1.7466414788207454*cos(\t r)+0*1.7466414788207454*sin(\t r)},{0*1.7466414788207454*cos(\t r)+1*1.7466414788207454*sin(\t r)});\draw [shift={(4.35059486468616,0.8703469449138664)}] plot[domain=3.2628955012143463:3.8595915302354475,variable=\t]({1*1.7368656076078013*cos(\t r)+0*1.7368656076078013*sin(\t r)},{0*1.7368656076078013*cos(\t r)+1*1.7368656076078013*sin(\t r)});\draw [shift={(4.652374886131706,3.116167613487951)}] plot[domain=4.02266400148017:4.506017815718893,variable=\t]({1*3.1837232271990206*cos(\t r)+0*3.1837232271990206*sin(\t r)},{0*3.1837232271990206*cos(\t r)+1*3.1837232271990206*sin(\t r)});\draw [dash pattern=on 1pt off 1pt] (1.7514,0.4952) circle (1.0530068721170978cm);
\draw [very thin, ->] (1.7514,0.4952) -- (1.1207950174225378,1.3385035210849723);
\draw[line width=0.8pt] (1.1175, -0.3456 ) -- (2.8, 0.52 );
	\begin{scriptsize}
	\draw [fill=black] (-4,0) circle (1pt);
	\draw[color=black] (-3.859385910866348,0.19390866540199736) node {$x$};
	\draw [fill=black] (4,0) circle (1pt);
	\draw[color=black] (4.142016419907008,0.19390866540199736) node {$y$};
	\draw[color=black] (0.07485842453051775,0.1584650227407644) node {$[x,y]$};
	\draw [fill=black] (1.7514,0.4952) circle (1pt);
	\draw[color=black] (1.9710933069064847,0.6369541986674093) node {$t_i$};
	\draw [fill=black] (2.8040146689978287,0.523937281298438) circle (1pt);
	\draw[color=black] (2.9812371227516254,0.5926496453408681) node {$b$};
	\draw[color=black] (1.7229878082778534,1.000251535945047) node {$3\Delta$};
	\draw [fill=black] (1.1175718892228943,-0.3456837010636835) circle (1pt);
	\draw[color=black] (1.1736113470287415,-0.19597140387156514) node {$a$};
	\end{scriptsize}
\end{tikzpicture}
\end{center}
\caption{In order to remove backtracking, we connect with a geodesic of \(\Gamma\) the first and the last point that intersects the 
sphere of radius \(3\Delta\) around \(t_i\).}
\end{figure}
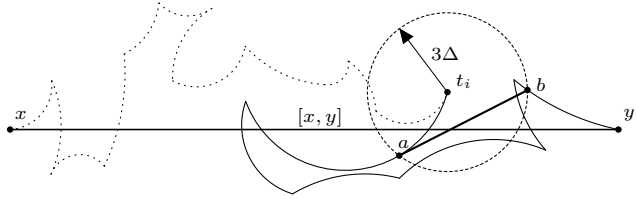

	The general strategy will be similar to the one of Step 1. First we remark that there is a natural order on a geodesic induced by the 
	choice of an order on the endpoints. 
	From now on, we will always consider the order on the geodesics induced by the choice \(x\leq y\). 
	
	We want now to further modify \(\gamma''\) and \(\widetilde{\gamma}''\). We will do this inductively, obtaining a sequence \(\{t_i\}_{i=1}^m\)
	of points
	in \(\widetilde{\gamma}''\), a sequence \(\{\gamma'_{t_i}\}_{i=1}^m\) of paths in \(\widehat{\Gamma}\), and a sequence 
	\(\{\widetilde{\gamma}'_{t_i}\}_{i=1}^m\) 
	of de-electrifications  in \(\Gamma\).
	We start the process setting \(t_0=x\), \(\gamma'_{t_{-1}}= \gamma''\) and
	\(\widetilde{\gamma}'_{-1} = \widetilde{\gamma}''\).

	Suppose that \(t_i\), \(\gamma'_{t_{i-1}}\) and \(\widetilde{\gamma}'_{t_{i-1}}\) are defined, and that \(t_i \in \gamma'_{t_{i-1}}\).
	Let \(B\) be the ball of center \(t_i\) and radius \(3\Delta\) in \(\Gamma\).
	For the rest of Step 2, we will consider only connected components of  \(\widetilde{\gamma}'_{t_{i-1}} -B\) that:
	\begin{rules}
	\item have at least two pieces or contain \(y\), 
	\item come after the point \(t_i\) in \(\widetilde{\gamma}'_{t_{i-1}}\).
	\end{rules}
	Let \(\mathcal{T}_i\) be the set of points obtained intersecting each connected component as above with the sphere of same radius and center as \( B\).
	If the set \(\mathcal{T}_i\) is empty or contains only one element, then no action is performed. 
	Otherwise, let \(a\) and \(b\) be the first and the last points of \(\mathcal{T}_i\).
	Then we obtain \(\widetilde{\gamma}'_{t_i}\) from  \(\widetilde{\gamma}'_{t_{i-1}}\) by removing \(\widetilde{\gamma}''_{|_{[a,b]}}\)
	and adding a geodesic \([a,b]\) of \(\Gamma\) connecting 
	\(a\) and \(b\). We modify \(\gamma'_{t_{i-1}}\) accordingly (namely, interrupting it on \(\{a,b\}\) and then substituting 
	\(\gamma''_{|_{[a,b]}}\) with \([a,b]\)). Note that \(\widetilde{\gamma}'_{t_i}\) is a de-electrification of 
	\(\gamma'_{t_i}\), and that \(L([a,b])_{\widehat{\Gamma}} - L \left({\gamma''_{t_i}}_{|_{[a,b]}}\right)_{\widehat{\Gamma}} \leq 2\Delta -2 \).
	If \(y \in B\), then we stop.
	Otherwise, we repeat the whole procedure with \(t_{i+1}=b\). Note that  if \(y\not\in B\), then \(\mathcal{T}_i\) is not empty and 
	\(\widetilde{\gamma}'_{|_{[b,y]}} \cap B = \emptyset\). This implies that \(y\) and \(b\) lie on the same connected component of 
	\(N_{\Delta}([x,y]) - B\). This, and the fact that the radius of \(B\) is \(3\Delta\), implies the following two:
	\begin{align}
	&d_\Gamma (p_{[x,y]}(a), p_{[x,y]}(b)) \geq \Delta,\\
	& d_\Gamma (p_{[x,y]}(a), y) \geq d_\Gamma (p_{[x,y]}(b), y) + \Delta,
	\end{align}
	 where \(p_{[x,y]}(a), p_{[x,y]}(b)\) denote the shortest distance projections of \(a\) and \(b\) on \([x,y]\).
	 In particular, there is \(m < \infty\) such that \(y\) is contained in the ball of radius \(3\Delta\) around \(t_m\), namely, the process stops
	 after finitely many steps.
	 We set \(\gamma'= \gamma'_{t_m}\) and \(\widetilde{\gamma}'=\widetilde{\gamma}'_{t_m}\). 
	 We remark that the sequence \(\{t_i\}_{i+1}^m\), may contain points that are not elements of \(\gamma'\). 
	 In fact, in the cases where the set \(\mathcal{T}_i\) is empty, no action is performed, and in particular no interruption on \(\gamma''\).

	\textbf{Step 3:} \(\widetilde{\gamma}'\) is a quasi-geodesic in \(\Gamma\).
	
	We claim that the sequence \(t_0, \dots, t_m\) provides a quasi-isometric embedding of \(\{0, \dots, m\}\) into \(\Gamma\).
	For each \(i<m\), since \(t_{i+1}\) is on the sphere of radius \(3\Delta\) around \(t_i\), clearly \(d_\Gamma (t_i, t_{i+1})= 3\Delta\).
	Hence, we get the bound \(d_\Gamma (t_i, t_j) \leq 3\Delta|i-j|\).
	We want to provide a lower bound. Given \(t_i\), \(t_j\), and applying inequality (2) to the sequence 
	\(p_{[x,y]}(t_i), p_{[x,y]}(t_{i+1}), \dots , p_{[x,y]}(t_j)\), we get that 
	\(d_\Gamma (p_{[x,y]}(t_i), y) \geq d_\Gamma (p_{[x,y]}(t_j), y) + |i-j|\Delta\). Since all the points considered are on a geodesic, we have that
	\(d_\Gamma (p_{[x,y]}(t_i), p_{[x,y]}(t_j)) \geq |i-j|\Delta\). The fact that the distances \(d_\Gamma(t_i, p_{[x,y]} (t_i))\) are uniformly bounded, 
	provides the desired bound. 
	Now we claim that for each \(i\), the path \(\widetilde{\gamma}'_{[t_i, t_{i+1}]}\) is uniformly a quasi-geodesic. 
	Note that it consists of a concatenation of pieces contained in a ball of radius \(\Delta\). Hence each piece has length at most \(2\Delta\).
	Moreover, it is clear by construction that \(\widetilde{\gamma}'_{[t_i, t_{i+1}]} = \widetilde{\gamma}''_{[t_i, a]} * [a, t_{i+1}]\),
	where \(a\) is as in Step 2.
	Let  \(q_1, q_2\) be any two points on \(\widetilde{\gamma}''_{[t_i, a]}\). Since: \(\gamma''\) is a quasi-geodesic 
	of \(\widehat{\Gamma}\), \(\widetilde{\gamma}''\) is a de-electrification for \(\gamma''\), 
	and \(\widetilde{\gamma}'_{[q_1, q_2]} = \widetilde{\gamma}''_{[q_1, q_2]}\), arguing as in Step 1, we get that the number of pieces between 
	them can be bounded by a linear function of the distance between \(q_1\) and \(q_2\). In particular, the bound on the length of each 
	piece gives an upper bound on \(L \left( \widetilde{\gamma}'_{[q_1, q_2]}\right) \).
	Now consider the case of general endpoints: since the length of \([a, t_{i+1}]\) is uniformly bounded by \(2\Delta\), up to increasing the 
	additive constant by \(4\Delta\), we get the desired bound.
	Hence, we get the claim.
	But a concatenation obtained joining a set of quasi-isometrically embedded points with uniform quasi-geodesic is a quasi-geodesic. 
	
	We remark that our construction of \(\gamma'\) depends on the choice of de-electrification \(\widetilde{\gamma}\) of \(\gamma\), and 
	in order to perform the construction, we used the auxiliary sequence \(\{\widetilde{\gamma}'_{t_i}\}\) of paths in \(\Gamma\). 
	This gave us a particular choice for the de-electrification of \(\gamma'\).
	However, we remark that after \(\gamma'\) is constructed, the argument of Step 3 applies to any de-electrification of \(\gamma'\), and not
	only to the de-electrification obtained during the construction.

	\textbf{Step 4:} \(\gamma'\) is a quasi-geodesic in \(\widehat{\Gamma}\).
	
		Consider the set \( X = \{t_i\}_{i=1}^m \cap \gamma'\). As remarked in Step 2, it may happen that \(X \subsetneq \{t_i\}\).
	Note that it is not possible that two elements \(t_i, t_j\) of \(X\) belong to the same \(\H\)--piece of \(\widetilde{\gamma}'\).
	Indeed, suppose that this holds for  \(t_i\) and \(t_j\), with \(j>i\).
	Then, since \(\widetilde{\gamma}''_{t_{i-1}}\) restricted to \(t_i, t_{i+1}\) is a geodesic segment, it follows that the set \(\mathcal{T}_i\) contains
	only one element.  Hence no action is performed. Since the point \(t_{i+1}\) is on an \(\H\)--piece and no interruption is performed, it follows that
	\(t_{i+1} \not\in\gamma'\). Repeating the above reasoning for \(t_{i+1}, \dots, t_{j-1}\) gives the claim.
	Since the number \(P\) of \(\H\)--components of \(\gamma\) is greater of equal to the number of \(\H\)--components of \(\gamma'\), 
	we have that \(|X| \leq P \leq L_{\widehat{\Gamma}}(\gamma) = d_{\widehat{\Gamma}}(x,y)\).
	Consider a restriction \(\chi\) of \(\gamma'\) with endpoints \(q_1, q_2\) and suppose that \(q_1, q_2 \in \gamma''\). 
	If \(X_\chi = X \cap \chi\), we similarly have that \(|X_\chi| \leq d_{\widehat{\Gamma}}(q_1, q_2)\). For each element of \(X_\chi\), let \(a_i\) and
	\(b_i\) be the endpoints of the corresponding \(\H\)--component. Finally, we observe that by Step 2, \(\gamma''\) is uniformly a quasi-geodesic. 
	In particular, there exists \(C\) such that \(L \left(\gamma''_{|_{[q_1, q_2]}}\right)_{\widehat{\Gamma}} \leq Cd_{\widehat{\Gamma}}(q_1, q_2) + C\).
	Then we have:
	 \begin{align*}	
	 L (\chi)_{\widehat{\Gamma}} &= L (\gamma''_{|_{[q_1, q_2]}})_{\widehat{\Gamma}} + 
	\sum_{X_\chi} \left(L\left([a_i, b_i])\right)_{\widehat{\Gamma}} - L\left({\gamma''}_{|_{[a_i, b_i]}}\right)_{\widehat{\Gamma}}\right) \leq \\
	&\leq L (\gamma'_{|_{[q_1, q_2]}})_{\widehat{\Gamma}} +  |X_\chi|(2K+1)+
	\sum_{X_\chi} \left(L\left([a_i, b_i]\right)_{\Gamma} - L\left({\gamma''}_{|_{[a_i, b_i]}}\right)_{\widehat{\Gamma}}\right) \leq \\
	&\leq Cd_{\widehat{\Gamma}}(q_1, q_2) + C + |X_\chi|(2K+1)+
	\sum_{X_\chi} \left(2\Delta d_\Gamma(q_1, q_2) + 2\Delta \right) =\\
	&= Cd_{\widehat{\Gamma}}(q_1, q_2) + C  + |X_\chi| \left(12\Delta^2 + 2\Delta +2K +1 \right) \leq \\
	&\leq  Cd_{\widehat{\Gamma}}(q_1, q_2) + C  + d_{\widehat{\Gamma}}(q_1, q_2) \left(12\Delta^2 + 2\Delta +2K+1 \right)=\\
	&= d_{\widehat{\Gamma}}(q_1, q_2)(C+12\Delta^2 + 2\Delta +2K+1) + C .	
	\end{align*}
	Since the \(\Gamma\)--geodesic segments \([a_i, b_i]\) have uniformly bounded length in \(\widehat{\Gamma}\), 
	we get that, up to uniformly increase the additive constant, we can bound the length of each subsegment of \(\gamma'\).
	Hence, \(\gamma'\) is uniformly a quasi geodesic of \(\widehat{\Gamma}\).
	\end{proof}
\end{prop}	

Suppose that \(\widehat{\Gamma}\) is the cone-off of \(\Gamma\) with respect to a general family of uniformly quasi-isometrically embedded subgraphs \(\H\). Sometimes it may be useful to consider a slightly different definition of de-electrification.
\begin{Def}
Let \(\widehat{\Gamma}\) be the cone-off of a graph \(\Gamma\) with respect to a family of uniformly quasi-isometrically embedded subgraphs \(\H\). Let 
	\(\gamma = u_1 *e_1 * \cdots * e_n * u_{n+1}\) be a path of \(\widehat{\Gamma}\), where each \(e_i\) is an \(H_i\)--component for some
	\(H_i \in \H\), and the \(u_i\) are (possibly trivial) segments of \(\Gamma\). 
	The \emph{embedded-de-electrification} \(\widetilde{\gamma}^e\)
	of \(\gamma\) is the concatenation \(u_1 * \eta_1 * \cdots * \eta_n *u_{n+1}\) where each \(\eta_i\) is a geodesic segment of \(H_i\) connecting 
	the endpoints of \(e_i\). We define \(H\)--pieces and \(\H\)--pieces as in the case of de-electrifications. 
\end{Def}
Note that if \(\Gamma\) is \(\delta\)--hyperbolic and the elements of \(\H\) are uniformly quasi-isometrically embedded, the two definitions are coarsely the same. Indeed, given a path \(\gamma\) of \(\widehat{\Gamma}\) it is easy to see that \(\widetilde{\gamma}\) is a quasi-geodesic of \(\Gamma\) if and only if \(\widetilde{\gamma}^e\) is. 

As a corollary, we get the following version of Proposition \ref{Bound Hausdorff distance}.
\begin{corollary}\label{cor: Bound Hausdorff distance}
Let \(\Gamma\) be a \(\delta\)--hyperbolic graph, \(\H\) a family of uniformly quasi-isometrically embedded subgraphs and \(\widehat{\Gamma}\) the cone-off of \(\Gamma\) with respect to \(\H\). 
Then there exist \(\tau_1 = \tau_1 (\delta, K)\) and \(\tau_2= \tau_2 (\delta, K)\) such that for each pair of points \(x,y \in \Gamma\) there exists a \(\tau_1\)--quasi-geodesic \(\gamma'\) of \(\widehat{\Gamma}\) with the property that for each  embedded-de-electrification \(\widetilde{\gamma}'\) of \(\gamma'\), \(\widetilde{\gamma}'\) is a \(\tau_2\)--quasi-geodesic of \(\Gamma\).
\end{corollary}

\section{Hierarchically hyperbolic spaces}\label{sec:hierarchically hyperbolic spaces}

\subsection{Factor Systems}

\begin{Def}[Factor System]\label{Factor System}
	Let \(\Gamma\) be a connected graph. We say that a family \(\H\) of connected subgraphs of \(\Gamma\) is a 
	\emph{factor system} for \(\Gamma\) if there are  \(K,c, \xi, B\) such that the following are satisfied. 
	\begin{enumerate}
	\item Each \(W \in \H\) is \(K\)--quasi-isometrically embedded in \(\Gamma\), with respect to the path metric on \(W\).
	\item Given \(W_1, W_2 \in \H\), either \(\mathrm{diam} (p_{W_1} (W_2)) \leq \xi\), 
	or there is \(U \in \H\) with such that \(U \subseteq W_1\) and \(d_{\mathrm{Haus}}(p_{W_1} (W_2), U ) \leq B\).
	\item Given \(W_1, W_2 \in \H\), if \(d_{\mathrm{Haus}}(p_{W_1}(W_2), W_1) \leq B\), then \(W_1 \subseteq W_2 \).
	\item Every ascending chain of inclusions \(W_1 \subsetneq W_2 \subsetneq \cdots \subsetneq W_n\) has length bounded above by \(c\).
	\item Given \(W_1, W_2 \in \H\), if \(d_{\mathrm{Haus}}(W_1, W_2) < \infty\), then \(W_1 = W_2\).
	\end{enumerate}
\end{Def}
We remark that if \(\Gamma\) is hyperbolic, then being uniformly quasi-isometrically embedded implies being uniformly quasiconvex.
Note that the inclusion induces a partial order on \(\H \cup \{\Gamma\}\) with unique maximal element \(\Gamma\).

\begin{Def}
	Let \(\Gamma\) be a connected graph, let \(\H\) be a factor system for \(\Gamma\) and let \(W \in \H\). We define
	\(\H_W = \{U \in \H \mid  U \subsetneq W\}\). We denote with \(\widehat{W}\) the cone off of \(W\) with respect to 
	\(\H_W\). Note that \(\widehat{W} \subseteq \widehat{\Gamma}\), where \(\widehat{\Gamma}\) is the cone-off of \(\Gamma\) with respect to \(\H\). 
\end{Def}

\begin{rmk}\label{sub factor system}
	If \(\Gamma\) is a graph, \(\H\) is a factor system for \(\Gamma\) and \(W \in \H\), then \(\H_W\) is a factor system for \(W\) 
	with constants \(K', c-1, \xi'\), where \(K' = K'(K), \xi' = \xi'(K, \xi)\).  
\end{rmk}

In the setting of factor systems, we want to introduce a finer version of de-electrification.
\begin{Def}[Partial de-electrification]
	Let \(\H\) be a factor system for a connected graph \(\Gamma\). Let 
	\(\gamma = u_1 *e_1 * \cdots * e_n * u_{n+1}\) be a path of \(\widehat{\Gamma}\), where the \(e_i\) are \(W_i\)--components, for \(W_i \in \H\)
	and the \(u_i\) are (possibly trivial) segments of \(\Gamma\). 
	The \(C\)--\emph{partial de-electrification} 
	of \(\gamma\) is the concatenation \(u_1 * \eta_1 * \cdots * \eta_n *u_{n+1}\) where each \(\eta_i\) is a  \(C\)--quasi-geodesic segment of 
	\(\widehat{W_i}\)
	connecting 	the endpoints of \(e_i\). We denote it by \(\widetilde{\gamma}^{(1)}_C\).
	We denote the partial \(C\)--de-electrification of \(\widetilde{\gamma}^{(1)}_C\) by \(\widetilde{\gamma}^{(2)}_C\), and so on.
	For simplicity, we denote \(\widetilde{\gamma}^{(0)}=\gamma \).
\end{Def}

Note that a partially de-electrified path still contains \(\H\) components. We gained that if \(e\) is a \(W\)--component for a partially
de-electrified geodesic, then necessarily there exists \(V\in\H\) with \(W \subsetneq V\).

\begin{lemma}[Partial pigeonhole for cone-offs]\label{partial pigeonhole}
Let \(\Gamma\) be a \(\delta\)--hyperbolic graph and \(\H\) be a factor system for \(\Gamma\). 
Then for each \(\theta\) there exists a \(T=T(K, c)\) and \(\tau= \tau (\delta, \H)\) such that for any pair of points \(x,y\) satisfying \(d(x,y)_\Gamma \geq T\), there is a path \(\gamma\) of \(\widehat{\Gamma}\), a number \(n\)
and \(W \in \H\cup \{\Gamma\}\) such that:
\begin{rules}
\item  the de-electrification \(\widetilde{\gamma}\) is uniformly a quasi-geodesic of \(\Gamma\),
\item the path \(\widetilde{\gamma}^{(n)}_\tau\cap W\) is uniformly a quasi-geodesic of \(\widehat{W}\), 
\item \(L(\widetilde{\gamma}^{(n)}_\tau\cap W)_{\widehat{W}} \geq \theta\).
\end{rules}
\begin{proof}
	We induct on the complexity of \(\H\). 
	So suppose now that the claim holds for \(c \leq m-1\).  
	Let \(\gamma\) be the path provided by Proposition \ref{Bound Hausdorff distance}. Lemma \ref{Pigeonhole for cone-offs} 
	guarantees that for each \(T\), there exists \(T' = T'(T)\) such that if \(d(x,y) \geq T'\), 
	then either there is a \(U\)--piece \(\sigma\) of \(\widetilde{\gamma}\)
	with \(L(\sigma)_U \geq T \) for some \(U\), or \(L_{\widehat{\Gamma}}(\gamma) \geq T\).
	In the second case, we are done choosing \(n=0\), because \(\gamma^{(0)}=\gamma\) is uniformly a quasi-geodesic and, by the choice of \(\gamma\), 
	every de-electrification is uniformly a quasi-geodesic.
	So suppose that there is \(U\in \H\) with \(L(\widetilde{\gamma} \cap U)_{U} \geq T \). Let \(a,b\) be the endpoints \(\widetilde{\gamma} \cap U\).
	By Remark \ref{sub factor system} \(U\) is a hyperbolic graph with a factor system \(\H_U\) of complexity strictly less than the complexity of \(\H\).
	Since \(U\) is uniformly quasi-isometrically embedded in \(\Gamma\), we can control the \(U\) distance between \(a, b\). In particular, we get a path
	\(\eta\) of \(\widehat{U}\) as in the statement of the Lemma.
	We claim that the path obtained substituting  \(\eta\) to \(\gamma_{|_{[a,b]}}\) satisfies the original requirements.
	This is because \(\eta\) is uniformly a quasi-geodesic of \(\Gamma\), and thus there is a uniform \(\tau\) such that \(\eta\) is 
	part of a \(\tau\)--partial de-electrification of \(\gamma\). 
	Since the original complexity was finite, and at each step all the constants can be chosen uniformly, this proves the claim.
	\end{proof}

\end{lemma}

\subsection{Hierarchical structure}
In this section, we will prove that given a \(\delta\)--hyperbolic graph \(\Gamma\) and a factor system \(\H\), the pair \((\Gamma, \H)\) provides a hierarchically hyperbolic structure for \(\Gamma\). We emphasize that the non trivial part of the claim is the one concerning the indexing set. In fact, every hyperbolic space \(X\) admits \((X, \{X\})\) as a hierarchically  hyperbolic structure.

First, we recall the definition of hierarchically hyperbolic space.

\begin{Def}[Hierarchically hyperbolic space, \cite{HHSII}]\label{defn:space_with_distance_formula}\  \\
The $q$--quasigeodesic space  $(\cuco X,d_{\cuco X})$ is a \emph{hierarchically hyperbolic space} if there exists $\delta\geq0$, an index set $\mathfrak S$ and a set $\{CW \mid W\in\mathfrak S\}$ of $\delta$--hyperbolic spaces $(C U,d_U)$,  such that the following conditions are satisfied:\begin{enumerate}
\item\textbf{(Projections.)}\label{item:dfs_curve_complexes} There is
a set $\{\pi_W\colon \cuco X\rightarrow2^{C W}\mid W\in\mathfrak S\}$
of \emph{projections} sending points in $\cuco X$ to sets of diameter
bounded by some $\xi\geq0$ in the various $C W\in\mathfrak S$.
Moreover, there exists $K$ so that each $\pi_W$ is $(K,K)$--coarsely
Lipschitz.

 \item \textbf{(Nesting.)} \label{item:dfs_nesting} $\mathfrak S$ is
 equipped with a partial order $\sqsubseteq$, and either $\mathfrak
 S=\emptyset$ or $\mathfrak S$ contains a unique $\sqsubseteq$--maximal
 element; when $V\sqsubseteq W$, we say $V$ is \emph{nested} in $W$.  We
 require that $W\sqsubseteq W$ for all $W\in\mathfrak S$.  For each
 $W\in\mathfrak S$, we denote by $\mathfrak S_W$ the set of
 $V\in\mathfrak S$ such that $V\sqsubseteq W$.  Moreover, for all $V,W\in\mathfrak S$
 with $V$ properly nested in $W$ there is a specified subset
 $\rho^V_W\subset C W$ with $\mathrm{diam}_{C W}(\rho^V_W)\leq\xi$.
 There is also a \emph{projection} $\rho^W_V\colon C
 W\rightarrow 2^{C V}$.  (The similarity in notation is
 justified by viewing $\rho^V_W$ as a coarsely constant map $C
 V\rightarrow 2^{C W}$.)
 
 \item \textbf{(Orthogonality.)} 
 \label{item:dfs_orthogonal} $\mathfrak S$ has a symmetric and
 anti-reflexive relation called \emph{orthogonality}: we write $V\perp
 W$ when $V,W$ are orthogonal.  Also, whenever $V\sqsubseteq W$ and $W\perp
 U$, we require that $V\perp U$.  Finally, we require that for each
 $T\in\mathfrak S$ and each $U\in\mathfrak S_T$ for which
 $\{V\in\mathfrak S_T\mid V\perp U\}\neq\emptyset$, there exists $W\in
 \mathfrak S_T-\{T\}$, so that whenever $V\perp U$ and $V\sqsubseteq T$, we
 have $V\sqsubseteq W$.  Finally, if $V\perp W$, then $V,W$ are not
 $\sqsubseteq$--comparable.
 
 \item \textbf{(Transversality and consistency.)}
 \label{item:dfs_transversal} If $V,W\in\mathfrak S$ are not
 orthogonal and neither is nested in the other, then we say $V,W$ are
 \emph{transverse}, denoted $V\pitchfork W$.  There exists
 $\kappa_0\geq 0$ such that if $V\pitchfork W$, then there are
  sets $\rho^V_W\subseteq C W$ and
 $\rho^W_V\subseteq C V$ each of diameter at most $\xi$ and 
 satisfying: $$\min\left\{d_{
 W}(\pi_W(x),\rho^V_W),d_{
 V}(\pi_V(x),\rho^W_V)\right\}\leq\kappa_0$$ for all $x\in\cuco X$.
 
 For $V,W\in\mathfrak S$ satisfying $V\sqsubseteq W$ and for all
 $x\in\cuco X$, we have: $$\min\left\{d_{
 W}(\pi_W(x),\rho^V_W),\mathrm{diam}_{C
 V}(\pi_V(x)\cup\rho^W_V(\pi_W(x)))\right\}\leq\kappa_0.$$ 
 
 The preceding two inequalities are the \emph{consistency inequalities} for points in $\cuco X$.
 
 Finally, if $U\sqsubseteq V$, then $d_W(\rho^U_W,\rho^V_W)\leq\kappa_0$ whenever $W\in\mathfrak S$ satisfies either $V\sqsubseteq W$ and \(V \neq W\) or $V\pitchfork W$ and $W\not\perp U$.
 
 \item \textbf{(Finite complexity.)} \label{item:dfs_complexity} There exists $n\geq0$, the \emph{complexity} of $\cuco X$ (with respect to $\mathfrak S$), so that any set of pairwise--$\sqsubseteq$--comparable elements has cardinality at most $n$.
  
 \item \textbf{(Large links.)} \label{item:dfs_large_link_lemma} There
exist $\lambda\geq1$ and $E\geq\max\{\xi,\kappa_0\}$ such that the following holds.
Let $W\in\mathfrak S$ and let $x,x'\in\cuco X$.  Let
$N=\lambda d_{_W}(\pi_W(x),\pi_W(x'))+\lambda$.  Then there exists $\{T_i\}_{i=1,\dots,\lfloor
N\rfloor}\subseteq\mathfrak S_W-\{W\}$ such that for all $T\in\mathfrak
S_W-\{W\}$, either $T\in\mathfrak S_{T_i}$ for some $i$, or $d_{
T}(\pi_T(x),\pi_T(x'))<E$.  Also, $d_{
W}(\pi_W(x),\rho^{T_i}_W)\leq N$ for each $i$.

 \item \textbf{(Bounded geodesic image.)} \label{item:dfs:bounded_geodesic_image} For all $W\in\mathfrak S$, all $V\in\mathfrak S_W-\{W\}$, and all geodesics $\gamma$ of $C W$, either $\mathrm{diam}_{C V}(\rho^W_V(\gamma))\leq E$ or $\gamma\cap N_E(\rho^V_W)\neq\emptyset$. 
 
 \item \textbf{(Partial Realization.)} \label{item:dfs_partial_realization} There exists a constant $\alpha$ with the following property. Let $\{V_j\}$ be a family of pairwise orthogonal elements of $\mathfrak S$, and let $p_j\in \pi_{V_j}(\cuco X)\subseteq C V_j$. Then there exists $x\in \cuco X$ so that:
 \begin{itemize}
 \item $d_{V_j}(x,p_j)\leq \alpha$ for all $j$,
 \item for each $j$ and 
 each $V\in\mathfrak S$ with $V_j\sqsubseteq V$, we have 
 $d_{V}(x,\rho^{V_j}_V)\leq\alpha$, and
 \item if $W\pitchfork V_j$ for some $j$, then $d_W(x,\rho^{V_j}_W)\leq\alpha$.
 \end{itemize}

\item\textbf{(Uniqueness.)} For each $\kappa\geq 0$, there exists
$\theta_u=\theta_u(\kappa)$ such that if $x,y\in\cuco X$ and
$d(x,y)\geq\theta_u$, then there exists $V\in\mathfrak S$ such
that $d_V(x,y)\geq \kappa$.\label{item:dfs_uniqueness}
\end{enumerate}
We often refer to $\mathfrak S$, together with the nesting
and orthogonality relations, the projections, and the hierarchy paths,
as a \emph{hierarchically hyperbolic structure} for the space $\cuco
X$.  
\end{Def}

\begin{convention}
As before, we will assume that a hyperbolic graph \(\Gamma\) and a factor system \(\H\) for \(\Gamma\) are fixed.
\end{convention}

\begin{Def}\label{Def rho}
	Let \(\widehat{\Gamma}\) be the cone-off of \(\Gamma\) with respect to \(\H\), and let \(i_\Gamma\colon \Gamma \rightarrow \widehat{\Gamma}\) 
		be the bijection on the vertex set.
		As remarked before, \(i_\Gamma\) is distance-non-increasing, and \(i_\Gamma^{-1}\) is distance-non-decreasing.
	\begin{enumerate} 
	\item For each \(W \in \H\), let \(\widehat{W}\) be the cone-off of \(W\) with respect to the family \(\H_W = \{H \in \H \mid  H \subsetneq W\}\). 
		Since \(\widehat{W} \subseteq \widehat{\Gamma}\), the maps \(i_\Gamma\) and \(i_\Gamma^{-1}\) restrict to maps \(i_W\), \(i_W^{-1}\). As before, 
		the maps \(i_W\) are distance-non-increasing and \(i_W^{-1}\) distance-non-decreasing.
		Note that \(\widehat{W} = i_W (W)\).
	\item We denote with \(\pi_W\colon \Gamma \rightarrow 2^{\widehat{W}}\) the map \(i_W \circ p_W\), where \(p_W\) denotes the shortest distance 
		projection in \(\Gamma\).
	\item For \(V, W \in\H\) such that \( V \not\subseteq W\), 
		we define \(\rho^W_V\) to be  \(p_V(W)\). Condition 2 of the definition of factor system yields that the sets \(\rho^W_V\) have uniformly bounded diameter.
	\item For \(V, W \in \H\) such that \(V \subseteq W\), we define a map \(\rho^W_V\colon \widehat{W} \rightarrow 2^{\widehat{V}}\) as 
	\(\rho^W_V =  \pi_V \circ i_W^{-1} \).	
	\end{enumerate}
\end{Def}

The following result is proved in \cite[Proposition 2.6]{KapovichRafi2014}. It is an application of the  Bowditch criterion for hyperbolicity (\cite{bowditch2006intersection}).

\begin{prop}[Kapovich-Rafi, Bowditch]\label{Kapovich-Rafi}
	Let \(\Gamma\) be a connected graph with simplicial metric \(d_\Gamma\) such that \((X,d_X)\) is \(\delta\)--hyperbolic. 
	Let \(K> 0\) and \(\H\) be a family of \(K\)--quasiconvex subgraphs of \(\Gamma\).
	Let \(\widehat{\Gamma}\) be the cone-off of \(\Gamma\) with respect to the family \(\H\).
	Then \(\widehat{\Gamma}\) is \(\delta'\)--hyperbolic (with respect to the path metric) 
	for some constant \(\delta' >0\) depending only on \(K\) and \(\delta\).  Moreover there exists \(H=H(K,\delta )>0\) such that whenever
	\(x,y\in V(\Gamma)\), \([x,y]_\Gamma\) is a \(d_\Gamma\)--geodesic from \(x\) to \(y\) in \(\Gamma\) and
	\( [x,y]_{\widehat{\Gamma}}\) is a \(d_{\widehat{\Gamma}}\)--geodesic from \(x\) to \(y\) in \(\widehat{\Gamma}\) 
	then \([x,y]_\Gamma\) and \([x,y]_{\widehat{\Gamma}}\) are \(H\)--Hausdorff close in \((\widehat{\Gamma},d_{\widehat{\Gamma}})\).
\end{prop}

\begin{corollary}\label{delta' hyperbolicity}
	There is a uniform \(\delta'\) such that the spaces \(\widehat{W}\), for \(W\in\H\cup\{ X\}\) are \(\delta'\)--hyperbolic.
\end{corollary}

\begin{prop}[Bounded projections]\label{bounded projections}
	There exists \(\Theta\) such that for each pair \(F,W\in\H\), one of the following holds:
		\begin{rules}
			\item \(\mathrm{diam}_{\widehat{F}} (\rho^\Gamma_F ( \widehat{W}) )  \leq \Theta\);
			\item \(F \subseteq W\).
		\end{rules}
	\begin{proof}
		Let \(\Theta = 2B + \xi + 2\) and suppose that the first does not hold.
		Unraveling the definitions, we get: 
		\[\rho^\Gamma_F (\widehat{W}) =  \pi_F \circ i_\Gamma^{-1} ( i_W (W)) = \pi_F(W) = i_F (p_F(W)).\]
		Since \(i_F\) is distance-non-increasing, we get that \(\mathrm{diam}_F (p_F(W)) > \Theta > \xi \). 
		By definition of factor system, this implies that there is \(U \in \H\) such that \(U \subseteq F\) and  \(d_{\mathrm{Haus}}(p_F(W), U) \leq B\). 
		Since \(\Theta = 2B +\xi +2\), we have that \(\mathrm{diam}_{\widehat{F}}(U) \geq 2\), thus \(U\) is not coned-off in \(\widehat{F}\). 
		Since each element of \(\H_F\) is coned-off in \(\widehat{F}\), we have that \(U \not \in \H_F\) and hence \(U=F\). 
		Thus, \(d_{\mathrm{Haus}}(F, p_F(W))\leq B\). The definition of factor system then implies that \(F \subseteq W\). 
	\end{proof}
\end{prop}

\begin{lemma}[Bounded geodesic image]\label{bounded geodesic image}
There exists \(B \) such that for all \(W \in \H\) and \(V \in \H_W\) and all geodesics \(\gamma\) of \(\widehat{W}\), either \(\mathrm{diam}_{\widehat{V}} (\rho^W_V (\gamma)) < B\), or 
\(\gamma \cap N_B (\rho^V_W) \neq \emptyset\).
\begin{proof}
	Suppose that \(\mathrm{diam} (\rho^W_V (\gamma))> 8\delta + K\), where  \(K\) is the uniform quasiconvexity constant of the elements of \(\H\).
	We claim that this implies that \(\gamma \cap N_{2 \delta + K + H} (\rho^V_W) \neq \emptyset\), where \(H\)
	is the constant of Proposition \ref{Kapovich-Rafi}. 
	Fix \(a,b \in \widehat{V}\) witnessing \(\mathrm{diam} (\rho^W_V (\gamma))\) and let
	\(x,y\in \gamma\) be in the preimage under \(\rho^W_V\) of \(a,b\). 
	Note that \(d_\Gamma (a,b) \geq d_{\widehat{V}}(a,b) = 8\delta + K\). Then a quadrilateral argument gives that \(d_\Gamma (V, [x,y]) < 2\delta +K\), 
	where \([x,y]\) is a geodesic in \(\Gamma\). But since, by Proposition \ref{Kapovich-Rafi}, \(d_{\mathrm{Haus}}([x,y], \gamma_{|_{[x,y]}}) \leq H\), 
	we get the claim.
\end{proof}
\end{lemma}

\begin{prop}[Large link Lemma]\label{LLL}
There exist \(\lambda \geq 1\) and \(E \) such that the following holds.
Let \(W \in \H \cup \{\Gamma\}\) and \(x,y \in \Gamma\). Let \(N = \lambda d(\pi_W(x), \pi_W(y)) + \lambda\).
Then there exists \(\{T_1, \dots , T_{\lceil N \rceil}\} \subseteq \H_W\) such that for each \(V \in \H_W\), either \(V \in \H_{T_i} \cup \{T_i\}\) for some \(i \in \{1, \dots, \lceil N \rceil\}\), or \(d_V(\pi_V(x), \pi_V(y)) < E\). Also, \(d_{W} (\pi_W(x), \rho^{T_i}_W) \leq N\) for each i.
	\begin{proof}
	We will show that if \(E\) is chosen large enough, the proposition holds. We will give a more precise characterization of large during
	the proof. For now, just assume it is much larger than \(\delta, K\) and \(\Theta\), where \(K\) is the quasiconvexity constant of the family \(\H\), 
	and \(\Theta\) is as in Proposition \ref{bounded projections}.
	Since \(W\) is uniformly quasi-isometrically embedded in \(\Gamma\),
	 and since \(\H_W\) is uniformly a factor system, to simplify notation we can assume
	that \(W=\Gamma\). In fact, for each \(U \subseteq W\) and \(x\in \Gamma\), we have that \(p_U(x)\) and \( p_U(p_W(x))\) 
	(and hence \(\pi_U(x)\) and \(\pi_U(\pi_W(x))\)) coarsely coincide (see 
	Proposition \ref{second B. inequality}). Thus we can substitute \(x\) and \(y\) with \(\pi_W(x)\) and \(\pi_W(y)\).
	
	Let \(\gamma\) be a \(\tau_1\)--quasi-geodesic of \(\widehat{\Gamma}\) between \(x\) and \(y\), such that for every de-electrification 
	\(\widehat{\gamma}\), we have \(\widehat{\gamma} \subseteq N_{\Delta}([x,y])\). Proposition \ref{Bound Hausdorff distance} guarantees that
	such a quasi-geodesic exists and that \(\tau_1, \Delta\) can be chosen uniformly.
	Suppose that there is \(F\) with \(d_{\widehat{F}}(\pi_F(x),\pi_F(y)) \geq E\). 
	We claim that if \(\widehat{\gamma}\) does not have a \(F\)--piece, then there is \(U \in \H\) such that \(\widehat{\gamma}\) has
	a \(U\)--piece and \(F\subseteq U\). Note that the number of \(\H\)--pieces of \(\widehat{\gamma}\) is at most 
	\(L_{\widehat{\Gamma}}(\gamma)\leq \tau_1 d_{\widehat{\Gamma}}(x,y) + \tau_1\). Then the spaces \(T_i\) of the statement are precisely
	the elements \(T\) of \(\H\) such that \(\gamma\) has a \(T\) component. Since \(\gamma\) is  uniformly a quasi-geodesic, 
	up to increase \(\lambda \), we get that the second claim translates as 
	\(d_{\widehat{\Gamma}}(x, T_i) \leq L_{\widehat{\Gamma}}(\gamma)\). Since the spaces \(T_i\) have, by construction, non empty intersection 
	with \(\gamma_i\), the second claim follows.
	
	Since \(d_{\widehat{F}}(\pi_F(x),\pi_F(y)) \geq E\), we get that \(\mathrm{diam}_{\widehat{F}}(\rho^\Gamma_F([x,y]) \geq E\).
	If \(E\) is large enough, a quadrilateral argument gives that there are points \(a',b'\) of \([x,y]\) such that:
	\begin{rules}
	\item \(d_\Gamma (a', F) \leq 2\delta + K\) and \(d_\Gamma (b', F) \leq 2\delta + K\), where \(K\) is the quasiconvexity constant of \(F\).
	\item \(d_\Gamma (p_F(a'), p_F(x)) \leq 4\delta + 2K\), and \(d_\Gamma (p_F(b'), p_F(y)) \leq 4\delta + 2K\). 
	\end{rules}
	Since \(\widetilde{\gamma} \subseteq N_\Delta ([x,y])\), we can find points \(a,b\) of \(\widetilde{\gamma}\) such that:
	\begin{rules}
	\item \(d_\Gamma (a, F) \leq 2\delta + K + \Delta\) and \(d_\Gamma (b, F) \leq 2\delta + K + \Delta\). 
	\item \(d_\Gamma (p_F(a), p_F(x)) \leq 2(4\delta + 2K + \Delta)\) and \(d_\Gamma (p_F(b), p_F(y)) \leq 2(2K + 4\delta + \Delta)\).
	\end{rules}
	
	Since \(F\) is uniformly quasi-isometrically embedded in \(\Gamma\), and since distances in \(F\) are larger than distances in \(\widehat{F}\),
	the last set of inequalities gives that we can find a uniform \(\rho\) such that \(d_{\widehat{F}} (\pi_F(a), \pi_F(x))\leq \rho\), 
	and similarly \(d_{\widehat{F}}(\pi_F(b), \pi_F(y)) \leq \rho\). 
	Since, by assumption, \(d_{\widehat{F}}(\pi_F(x), \pi_F(y)) \geq E\), we obtain that \(d_{\widehat{F}}(\pi_F(a), \pi_F(b) ) \leq E - 2 \rho\).
	 We would like to assume that \(a\) and \(b\) are points of \(\gamma\). Since this is not true
	in general, replace \(\gamma\) with its \(a,b\)--interruption \(\widehat{\gamma}\).
	Note that \(\widehat{\gamma}\) is a \((1, 4K + 2\delta)\)--quasi geodesic of \(CX\).
	We want now to estimate \(L(\widehat{\gamma}_{|[a,b]})_{\widehat{\Gamma}}\). 
	Since \(F\) is coned-off in \(\widehat{\Gamma}\), we have that 
	\[ d_{\widehat{\Gamma}}(a,b) \leq d_{\widehat{\Gamma}}(a, F) + d_{\widehat{\Gamma}}(b,F) + 1 \leq 2(2\delta + K + \Delta) +1.\]
	Thus,  \(L(\widehat{\gamma}_{|[a,b]})_{\widehat{\Gamma}} \leq  2(3\delta + 2K + \Delta) +1 = B\). In particular, \(\widehat{\gamma}_{|[a,b]}\)
	 has at most \(B\) pieces.
	Choosing \(E > B\Theta + 2 \rho\), we get that there is at least one
	\(U\)--piece \(\eta\) of \(\widehat{\gamma}_{|[a,b]}\), for some \(U \in \H\), such that 
	\(\mathrm{diam}_{\widehat{F}} (\rho^\Gamma_F(\eta)) > \Theta\). 
	But then by proposition \ref{bounded projections}, we have that \(F \subseteq U\), which proves the claim.
	\end{proof}
\end{prop}

\begin{prop}[Uniqueness]\label{Uniqueness}
For each \(\theta\) there exists \(T_\theta\) such that if \(x,y \in \Gamma\) and \(d_\Gamma (x,y) \geq T_\theta\), then there exists \(V \in \H\) such
 that \(d_{\widehat{V}}(\pi_V(x),\pi_V(y)) \geq \theta\).
	\begin{proof}
	
	We stick to the convention that \([x,y]\) indicates a geodesic in \(\Gamma\) between \(x\) and \(y\). 
	We will proceed by induction on the complexity of \(\H\). 
	If the complexity is \(0\) or \(1\), the result is respectively trivial or follows from Lemma \ref{Pigeonhole for cone-offs}.
	Suppose the result holds for complexity \(n-1\). 	
	Applying Lemma \ref{partial pigeonhole},
	we have that there exists \(T_C\) such that if \(d_{\Gamma}(x,y) \geq T_C\), then there is \(V \in \H\), a quasi-geodesic \(\widetilde{\gamma}\)
	of \(\Gamma\) and points \(s, t\) in \(V\), such that:
	\begin{rules}
	\item \(s,t \in \widetilde{\gamma}\), 
	\item there is a quasi-geodesic \(\sigma\) of \(\widehat{V}\) connecting \(s\) and \(t\),
	\item \(L_{\widehat{V}}(\sigma) \geq C\), in particular we can estimate \(d_{\widehat{V}} (s,t)\), 
	\item \(\widetilde{\gamma}_{|_{[s,t]}}=\widetilde{\sigma}\) is a de-electrification of \(\sigma\).
	\end{rules}
	Since \(\widetilde{\gamma}\) is a quasi-geodesic of \(\Gamma\), there is \(E\) such that \(\widetilde{\gamma} \subseteq N_E^\Gamma ([x,y])\). 
	
	Let \(s' \in \pi_V (x)\), \(t' \in \pi_V(y)\) realizing the distance \(d_{\widehat{V}}(\pi_V(x), \pi_V(y))\).
	Let \(q_s\) (resp \(q_t\)) be witnessing the closest point projection in \(\Gamma\) of 
	\(s\) (resp \(t\)) on \([x,y]\). Since \(d_{\Gamma} (q_s, s) \leq E\) and \(d_\Gamma (q_t, t) \leq E\),
	triangular inequality gives that \[d_\Gamma (x,s) + d_\Gamma (s,t) + d_\Gamma (t,y) - 4 E \leq d_\Gamma (x,y).\]
	Moreover, by the choice of \(s', t'\), we have that \(d_\Gamma (x,s') \leq d_\Gamma (x,s)\) and \(d_\Gamma (t',y) \leq d_\Gamma (t,y)\).
	Thus, we get that the left-hand-side of the above inequality is bounded below by \(d_\Gamma (x,s') + d_\Gamma (s,t) + d_\Gamma (t',y) - 4 E\). 
	Moreover,  by triangular inequality, the right-hand-side is bounded above by \(d_\Gamma (x,s') + d_\Gamma (s',t') + d_\Gamma (t',y)\).
	Thus we get	\[ d_\Gamma (s,t) -4E \leq d_\Gamma (s',t') .\] 
	In particular, up to increasing \(d_\Gamma (x,y)\), we can arbitrarily increase \(d_V(s',t')\). Since \(\H_V\) is a factor system for \(V\), 
	of complexity \(n-1\), we can apply the induction hypothesis on \(s', t'\). Thus we  will find a space \(U \subseteq V\) such that 
	\(d_{\widehat{U}}(\pi_U(s'), \pi_U (t')) = d_{\widehat{U}}(\pi_U(\pi_V(x)), \pi_U (\pi_V(y))) \leq \theta\). 
	Since, by Proposition \ref{second B. inequality}, the projection \(\pi_U(\pi_V(x))\) and \(\pi_U(x)\) coarsely coincide, for \(U \subseteq V\), 
	up to further increasing \(d_\Gamma (x,y)\), we get the claim.
	We remark that, in order to get uniform bound on the constants, it is crucial for the complexity of \(\H\) to be finite.

	\end{proof}
\end{prop}

\paragraph*{Verifying the axioms} \ \\
We will now verify that given a \(\delta\)--hyperbolic graph \(\Gamma\) and a factor system \(\H\) of \(\Gamma\), the set \(\H \cup \{\Gamma\}\) provides and hierarchically hyperbolic structure on \(\Gamma\). In particular, we have:
\begin{thm}\label{thm:main result for factor system}
Let \(\Gamma\) be a \(\delta\)--hyperbolic graph and \(\H\) a factor system for \(\Gamma\). Then there is a hierarchically hyperbolic space structure on \(\Gamma\) with indexing set \(\H \cup \{\Gamma\}\).
\begin{proof}
We claim that the set \(\{\widehat{W} \mid  W \in \H\cup\{ \Gamma\} \}\) satisfies the required conditions. Clearly the set is indexed by \(\H \cup \{\Gamma\}\). First note that the spaces \(\widehat{W}\) for \(W \in \H \cup \{\Gamma\}\) are uniformly \(\delta'\)--hyperbolic metric spaces (Corollary \ref{delta' hyperbolicity}). Then we have:
\begin{enumerate}
	\item There is a set of projections \(\{\pi_W\colon \Gamma \rightarrow 2^{\widehat{W}}\}\) (definition \ref{Def rho}) that are uniformly coarsely Lipschitz. 
	In fact, by definition, \(\pi_W = i_W  \circ p_W\) which is the composition of a coarsely Lipschitz map (Lemma \ref{projection q.Lips})
	with a distance-non-increasing map (Definition \ref{Def rho}).
	\item The set \(\H \cup \{\Gamma\}\) is naturally equipped with the partial order \(\subseteq \) induced by inclusion with maximal element \(\Gamma\).
	The sets and projections \(\rho^W_V\) are defined in Definition \ref{Def rho}. Since, for \(V \not\subseteq W\),  
	\(\rho^W_V \in \H_V\), we get that \(\rho^W_V\)
	is coned-off in \(V\), and thus	\(\mathrm{diam}_{V}(\rho^W_V) =2\).
	\item There are no orthogonality relations.
	\item \begin{rules}
			\item If \(W \not\subseteq V \not\subseteq W\), then there exists a uniform \(\kappa_0\) such that, for each \(x \in \Gamma\), 
			\[\mathrm{min}\{d_{\widehat{V}}(\pi_V(x), \rho^W_V), d_{\widehat{W}}(\pi_W(x), \rho^V_W) \leq \kappa_0.\]
			This is guaranteed by Lemma \ref{first B. inequality} and the fact that the maps \(i_W\) are distance-non-increasing.
			\item If \(V \subseteq W\), then for all \(x \in \Gamma\):
			\[\mathrm{min}\{d_{\widehat{W}}(\pi_W(x), \rho^V_W), \mathrm{diam}_{\widehat{V}}(\pi_V(x) \cup \rho^W_V(\pi_V(x))\}\leq \kappa_0.\]
			This is guaranteed by Lemma \ref{second B. inequality} and again the distance-non-increasingness.
			\item Since if \(U \subseteq W\), then \(\widehat{U} \subseteq \widehat{W}\), we get that for each \(V\), \(\rho^U_V \subseteq \rho^W_V\).
			\end{rules}
	\item Finite complexity is clear by the definition of factor system.
	\item The large link lemma is proved in  Proposition \ref{LLL}.
	\item The bounded geodesic image property is proved in Proposition \ref{bounded geodesic image}.
	\item The partial realization is trivial since there are no orthogonality relations, hence every family of pairwise orthogonal elements
	consists of a single element. 
	\item The uniqueness property is proved in Proposition \ref{Uniqueness}.
\end{enumerate}
\end{proof}
\end{thm}

\section{How to obtain a factor system}\label{sec: Obtaining a factor system}

One of the main drawbacks of the definition of factor system of Section \ref{sec:hierarchically hyperbolic spaces} is that it is not coarse. 
 The goal of this section is to address this issue, namely, to provide weaker conditions that a metric space needs to satisfy in order to be quasi-isometric to a graph equipped with a factor system.

\begin{Def}[Coarse inclusion]\label{def:coarse_inclusion}
	Let \(X\) be a metric space and \(V, W\) be two subspaces. We say that \(V\) is \emph{coarsely contained} in \(W\), and denote it by \(V \precsim W\),
	 if there exists \(R\) such that \(V \subseteq N_R(W)\).
	 We say that \(V\) is \emph{properly} coarsely contained in \(W\), and denote it by \(V \precnsim W\), if \(V \precsim W\) and for each \(R\), 
	 \(W \not\subseteq N_R(V)\). The relation \(\precsim\) will be called \emph{coarse inclusion}.
\end{Def}

\begin{Def}[Weak Factor System]\label{def: Weak factor system}
Let  \(X\) be a quasi-geodesic Gromov-hyperbolic space. A \emph{weak factor system} for \(X\) is a family \(\F\) of \(K\)--quasiconvex subspaces such that there exist constants \(c,\xi', D', B',q \) such that the following holds:
\begin{enumerate}
	\item Every chain of proper coarse inclusions \(H_n \precnsim \cdots \precnsim H_1\) of elements of \(\F\) has length at most \(c\).
	\item Given \(V, W \in \F\), then either \(\mathrm{diam}_V(p_V(W)) < \xi'\), or there exists \(U \in \F\) such that
		\(d_{\mathrm{Haus}}(U, p_V(W)) \leq B'\).
	\item For each \(V\in \F\), \(v \in V\) and constant \(\theta\) there is a \(q\)--quasi-geodesic segment \(\gamma\) with endpoints on \(V\) such that \(v \in N_{D'}(\gamma)\) and the distance between \(v\) and the endpoints of \(\gamma\) is at least \(\theta\).
\end{enumerate}
We say that \(\F\) is a \emph{geodesic} factor system for \(X\) if the segment \(\gamma\) in condition \(3\) can be chosen to be geodesic, namely if the following stronger condition is satisfied:
\begin{enumerate}
\item[3'.] For each \(V\in \F\), \(v \in V\) and constant \(\theta\) there is a  geodesic segment \(\gamma\) with endpoints on \(V\) such that \(v \in N_{D'}(\gamma)\) and the distance between \(v\) and the endpoints of \(\gamma\) is at least \(\theta\).
\end{enumerate}

\end{Def} 

\begin{rmk}[Up to quasi-isometry, all weak factor systems are geodesic and all spaces are graphs.]\label{rmk:wfs_holds_for_q_geod}
Let \(X\) be a \((C, \epsilon)\)--quasi-geodesic metric space and \(\F\) be a weak factor system for \(\X\). The approximation lemma (Lemma \ref{lem: approximation Lemma}) provides a quasi-isometry \(\omega \colon X \rightarrow \Omega(X)\) to a connected graph \(\Omega(X)\) such that the images \(\omega(F)\) of the elements of \(\F\) are connected graphs and constitute a geodesic weak factor system for \(\Omega(X)\). Indeed, note that condition (3) is quasi-isometric invariant. Since \(\Omega(X)\) is a hyperbolic geodesic space, the Morse Lemma (Theorem \ref{lem:Morse Lemma}) and condition (3) imply condition (3').
If \(X\) is equipped with a cobounded \(G\)--action, for some group \(G\), such that the action preserves the family \(\F\), the approximation lemma provides a cobounded \(G\)--action on \(\Omega(X)\) that preserves the images \(\omega (F)\), for \(F \in \F\). 
Finally, it is clear that the partial order on \(\F\) given by coarse inclusion is preserved under quasi-isometries. By all the above, we can effectively substitute \(X\) with an approximation graph without any loss of information.
\end{rmk}
\begin{rmk}
The condition (3) of the definition of weak factor system is more restrictive than necessary. Indeed, Example \ref{example:counterex wfs} shows a space \(X\) and a factor system for \(X\) that does not satisfy condition (3).
However, condition (3) can be easily verified for a large class of natural examples, such as infinite quasiconvex subgroups of hyperbolic groups. 
\end{rmk}

\begin{example}\label{example:counterex wfs}
Let \(m \in \NN \cup \infty\), and for each \(0 \leq i \leq m\) let \(I_i\) be a copy of the ray \([0, \infty)\) indexed by \(i\). 
Let \(X\) be the space obtained identifying together the point \(0\) of each of the rays \(I_i\). It is clear that we can regard \(X\) as a graph, in particular as a tree. For each \(n>0\), let \(F_n\) be the union \(I_0 \cup I_n\). 
Then \(\{F_n \mid n>0\} \cup I_0\) is a factor system for \(X\), but the space \(I_0\) does not satisfy condition (3) of the definition of weak factor system.
\end{example}

We consider the following equivalence relation on \(\F\): we say that \(V\) is in relation with \(W\) if \(V \precsim W \precsim V\), that is, if 
\(V\) and \(W\) have bounded Hausdorff distance. 
Let \(\H(\F)\) be the set of equivalence classes of \(\F\). It is easily checked that if \(V \precsim W\), then for each \(V'\in [V]\) and \(W' \in [W]\), one has \(V' \precsim W'\). Thus, the relation \(\precsim\) descends to a partial order on the set \(\H\) that we will denote by \(\sqsubseteq\).
The goal of this section is to prove the following:

\begin{thm}\label{thm: Main result for weak factor systems}
	Let \(X\) be a Gromov hyperbolic quasi-geodesic space space and \(\F\) a  weak factor system for \(X\). 
	Then \((X,\{X\} \cup \H(\F))\) is a hierarchically hyperbolic structure for \(X\).
\end{thm}

The proof of Theorem \ref{thm: Main result for weak factor systems} will follow as a corollary of Proposition \ref{prop: H factor system}, Remark \ref{rmk:wfs_holds_for_q_geod}  and the fact that hierarchically hyperbolic structures are invariant under quasi-isometry. 
The rest of this section is devolved in proving Proposition \ref{prop: H factor system}.

\begin{convention}
	For the rest of the section, we will use Remark \ref{rmk:wfs_holds_for_q_geod} to assume that a geodesic \(\alpha\)--hyperbolic space \(X\) with a geodesic weak factor system \(\F\) is fixed. In particular,	since no confusion is possible, we will simply denote \(\H = \H(\F)\).
\end{convention}

Our goal is to replace a weak factor system with a factor system, in order to be able to apply the results from Section \ref{sec:hierarchically hyperbolic spaces}. Roughly speaking, we want to associate a single subspace to an equivalence class \([W] \in \H\), in such a way that the requirements of factor system holds.

\begin{lemma}[Main consequence of condition (3')]\label{condition (3)}
Given a geodesic \(\delta\)--hyperbolic space \(X\) and a geodesic weak factor system \(\F\) for \(X\), there exists \(\zeta = \zeta (\delta, \H)\) such that for each \(V, W \in \F\), if  \(V \precsim W\), then \(V \subseteq N_\zeta (W)\).
\begin{proof}
	Let \(\zeta = 2\delta +D' +K\) and \(v \in V\). Then there exists \(v' \in V\) with \(d(v, v') \leq D'\) and a geodesic segment \(\gamma\) with
	endpoints in \(V\) such that \(\gamma\) has length \(2R + 4\delta +1\) and \(v'\) is the midpoint of \(\gamma\). 
	Let \(a, b\) be the endpoints of \(\gamma\) and \(a', b'\) points in \(W\) satisfying 
	\(d(a, a') \leq R\), and \(d(b,b') \leq R\). By hyperbolicity, \(v' \subseteq N_{2\delta} ([a, a'] \cup [a', b'] \cup [b',b])\). 
	We claim that  \(v' \subseteq N_{2\delta} ([a',b'])\). Then, by \(K\)--quasiconvexity of \(W\), the first property follows.
	So suppose that the claim does not hold, and assume that \(d(v', [a,a']) < 2\delta\). Let \(m \in [a,a']\) be a point realizing the distance.
	Since \(d(v', a) = R +2\delta +1\) and \(d(m,a') \leq d(a,a') \leq R\), we get a contradiction.
\end{proof}
\end{lemma}

For a class \([V] \in \H\), define the set \([\precsim V]=\{ U \in \F \mid  U \precsim V\}\). It is easily checked that the definition does not depend on representatives. We want to associate a subgraph to \([\precsim V]\).

\begin{Def}
Let \(\Gamma\) be a graph and \(Q\) be a subgraph of \(\Gamma\). For \(r > 0\) let \(P_r(Q)\) to be the set of \(\Gamma\)--geodesic segments  of length at most \(r\) connecting two points in \(Q\).
Define \[\mathrm{Approx}_r(Q) = \bigcup_{\gamma \in P_r(Q)} \gamma.\]
\end{Def}

\begin{Def}
Given \([V] \in \H(\F)\), we set 
\[\preg_V = \mathrm{Approx}_{\zeta} \left(\cup [\precsim V]\right),\]
where \(\cup [\precsim V] = \cup_{U \in [\precsim V]} U\).
\end{Def}
If we are in the presence of a group action, it is easily seen that  that the construction of  \(\preg_V\) is equivariant.
Indeed, for each \(g\in G\), if \(gV = V'\), then  \([gV] = [gV']\), and \([\precsim gV] = [\precsim gV']\). Thus,  \(g\preg_V = \preg_{V'}\).

As a consequence of Lemma \ref{condition (3)}, we have that for each \(V \in [V]\), \[d_{\mathrm{Haus}} (V, \preg_V) \leq \zeta.\]

\begin{lemma}\label{lem: CV are QI embedded}
The spaces \(\preg_W\) equipped with the induced path metric are uniformly quasi-isometrically embedded in \(X\).
\begin{proof}
This is because each \(W \in [W]\), is a connected graph, uniformly quasi-isometrically embedded in \(X\). Since every point of the graph \(\preg_W\) can be connected to \(W\) with a segment of uniformly bounded length, the claim follows. 
\end{proof}
\end{lemma}

\begin{corollary}\label{cor: prodregions quasiconvex}
The spaces \(\preg_W\) are uniformly quasiconvex and Gromov hyperbolic.
\end{corollary}

\begin{lemma}\label{lem:equivalence_with_preceq}
The following are equivalent:
\begin{rules}
\item \(\preg_V \precsim \preg_W\),
\item \([V] \sqsubseteq [W]\),
\item \(\preg_V \subseteq \preg_W\).
\end{rules}
\end{lemma}
\begin{proof}
Since \(\preg_V\) is the union of sets that have Hausdorff distance at most \(\zeta\), it is clear that the first condition implies that for each \(V' \in [V]\) and \(W' \in [W]\), we have \(V' \prec W'\), which is the definition of the second condition. 
The second implies the third because of the definition of the spaces \(\preg_V\).
The third trivially implies the first.
\end{proof}

\begin{lemma}\label{some conditions}
Given \(\preg_V\) and \(\preg_W\), if  \(\preg_V \precsim p_{\preg_V}(\preg_W)\), then \([V] \sqsubseteq [W]\), where \(p_{\preg_{W}}\) denotes the closest point projection on \(\preg_{W}\).
\begin{proof}
By Corollary \ref{cor: prodregions quasiconvex}, the closest point projection on \(\preg_W\) is well defined. 
Our goal is to show that \(p_{\preg_V}(\preg_W) \precsim \preg_W\). By transitivity of \(\precsim\) and by Lemma \ref{lem:equivalence_with_preceq}, this, in turn,  implies \([V]\sqsubseteq [W]\).
Let \(x \in p_{\preg_V}(\preg_W)\) be any point. By definition of \(\preg_V\), there is \(V' \in [V]\) and \(v \in V'\) such that \(d_X(x,v) \leq \zeta\).
By condition (3') of geodesic weak factor system, we can find an arbitrarily long geodesic segment \(\gamma\) with endpoints in \(V'\) and a point  \(v' \in \gamma\) such that \(d(v, v') \leq D'\), where \(D'\) is a uniform constant. 
Since the Hausdorff distance between \(V' and \preg_V\) is uniformly bounded, and since \(\preg_W \precsim p_{\preg_V}(\preg_W)\), we can assume that \(\gamma\) is a uniform quasi-geodesic with endpoints on \(p_{\preg_V}(\preg_W)\). 
Since \(\gamma\) was arbitrarily long, a quadrilateral argument allows us to uniformly bound the distance between \(x\) and \(\preg_W\).
Since all the constants are uniform in \(x\) and \(x\) is arbitrary, we get the claim.
\end{proof}
\end{lemma}

\begin{lemma}[Coarse commutativity of projections and quasi-isometries]\label{lemma: commuting q.iso-projection}
Let \(f \colon X \rightarrow Y\) be a \(C\)--quasi-isometry. 
Let \(H,J\) be \(K\)--quasiconvex subspaces  of \(X\). Then there exists \(M\) such that \(d_{Haus} \left(f\left(p_H(J)\right), p_{f(H)}\left(f(J)\right)\right) \leq M\).
	\begin{proof}
	Let \(\delta_X\) be the hyperbolicity constant of \(X\) and \(\delta_Y\) the one of \(Y\).
	We claim there exists \(M_1 = M_1(\delta_X, (\delta_Y), C)\) such that for any \(x \in J\), 
	\(d\left(f\left(p_H(x)\right), p_{f(H)}\left(f(x)\right)\right) \leq M_1\).
	In order to simplify notation we put \(y=f(p_H(x))\) and \(z = p_{f(H)}(f(x))\).
	Consider a geodesic triangle between \(f(x), y\) and \(z\). Since \([f(x),z]\) is a projection geodesic, it is easily 
	seen that there is a point \(m \in [f(x), y]\) that has distance at most \(2\delta_Y\) from \(f(H)\).
	However, since geodesics are uniformly near to quasi-geodesics, \(m\) is uniformly close to \(f\left(\left[x, f^{-1}(y)\right]\right)\). 
	Thus, there is a point \(m' \in f\left(\left[x, f^{-1}(y)\right]\right)\) such that it is possible to uniformly estimate the distances \(d(m', z)\) and
	\(d(m', H)\). Since \(f\left(\left[x, f^{-1}(y)\right]\right)\) is uniformly a quasi-geodesic, this proves the claim.
	\end{proof}
\end{lemma}

\begin{lemma}\label{Hausdorff distance for building f. sys.2}
There exists \(\xi, B\) such that if \(\mathrm{diam}(p_{\preg_V} (\preg_W) ) \geq \xi\), then there is \(U \in \F\) such that
\([U] \sqsubseteq [V]\) and \(d_{\mathrm{Haus}}(\preg_U, p_{\preg_V} (\preg_W)) \leq B\).
\begin{proof}

Since projections are quasi-Lipschitz and since, for each \(V' \in [V]\),  \(d_{\mathrm{Haus}}(\chi(V'), CV) \leq 1\), it is easily seen that is possible to uniformly bound the Hausdorff distance between \(p_{\preg_V}(\preg_W)\) and \(\bigcup_{V' \in [V], W' \in [W]} p_{\chi(V')}(\chi(W'))\).
Thus it suffices to show that there exists \(\beta\) such that for each \(V \in [V]\) and \(W \in [W]\), there exists \(U \in \F\) with 
\(d_{\mathrm{Haus}}(\chi(U), p_{\chi(V)}(\chi (W)) \leq \beta\). But this is an easy consequence of Lemma \ref{lemma: commuting q.iso-projection}.
\end{proof}
\end{lemma}

\begin{prop}\label{prop: H factor system}
    Let \(X\) be a connected graph and \(\F\) be a family of connected subgraphs forming a weak factor system for \(X\). Then the family \(\{\preg_V \mid  [V] \in \H(\F)\}\) is a factor system for \(X\). If, moreover, \(X\) is equipped with an isometric group action that preserves the family \(\F\), then the group action preserves the family \(\{\preg_V \mid  [V] \in \H(\F)\}\). 
	\begin{proof}
	We will show that all the items of Definition \ref{Factor System} are satisfied.
		\begin{enumerate}
		\item Lemma \ref{lem: CV are QI embedded} gives that the spaces \(\preg_V\) are quasi-isometrically embedded in \(\preg_X\), with constants that 
		can be chosen uniformly.
		\item The second condition follows from Corollary \ref{cor:omega(W)q.i.emb} and Lemma \ref{Hausdorff distance for building f. sys.2}. \item The third follows from Lemma \ref{some conditions}. 
		\item The fourth follows form the requirement of weak factor system and Lemma \ref{lem:equivalence_with_preceq}.
		\item The fifth follows from Corollary \ref{cor:omega(W)q.i.emb} and Lemma \ref{condition (3)}.
		\end{enumerate}
		The equivariance of the action follows from the definition of \(\preg_W\).
	\end{proof}
\end{prop}

As a corollary, we get a version of Remark \ref{sub factor system} for weak factor systems.

\begin{corollary}\label{cor:WeakSub-FactorSystem}
Let \(X\) be a Gromov hyperbolic space, \(\F\) be a weak factor system for \(X\) and let \(F \in \F\). 
Let \(\H(\F)_F = \{ [V] \in \F \mid  [V] \sqsubseteq [F]\}\). Then \((\preg_F, \{\preg_V \mid  [V] \in \H(F)_F\})\) is a hierarchically hyperbolic space.
\begin{proof}
This is because, by Remark \ref{sub factor system}, \(\{\preg_V \mid  [V] \in \H(F)_F\}-\preg_F \) is a factor system for \(\preg_F\). 
\end{proof}
\end{corollary}

\section{Groups with quasiconvex subgroups}\label{section:Groups with quasiconvex}

The goal of this section is to prove that given a hyperbolic group \(G\) and a family of infinite quasiconvex subgroups \(\F=\{F_1, \dots, F_n\}\) of \(G\), 
there is an HHS structure on \(G\) that contains the elements of the family \(\F\) in the indexing set.
In particular, we will show that we can extend the family \(\F\) to a weak factor system for \(G\). 
The key idea is to consider a "closure of \(\F\) under projection". Using the results of Section \ref{sec: Obtaining a factor system}, we can promote this to a factor system and obtain an HHS structure on the Cayley grapy of \(G\). Moreover, such a structure will be \(G\)--equivariant. In particular, this provides a, so called, hierarchically hyperbolic group structure on \(G\), which was introduced in \cite{HHSI}.

\subsection{Hierarchically hyperbolic group structure}\label{subsec HHG structure}

We recall the definition of hierarchically hyperbolic group.

\begin{Def}[Hieromorphism, \cite{HHSII}, Definition 1.19]
Let $(\X,\mathfrak S)$ and $(\X',\mathfrak S')$ be
hierarchically hyperbolic structures on the spaces $\X,\X'$
respectively.  A \emph{hieromorphism}, 
consists of a map $f\colon \X\rightarrow\X'$, an injective map $f^\diamondsuit \colon\mathfrak S\rightarrow\mathfrak
S'$ preserving nesting, transversality, and orthogonality, and maps
$f^* (U)\colon C U\rightarrow C(f^\diamondsuit (U))$, for each \(U \in \S\),
which are uniformly quasi-isometric embeddings.
The three maps should preserve the structure of the hierarchically hyperbolic space, that is, they coarsely commute with the maps \(\pi_U\) and \(\rho^U_V\), for \(U, V\) in either \(\S\) or \(S'\), associated to the hierarchical structures.
\end{Def}

\begin{Def}[Automorphism, hierarchically hyperbolic group, \cite{HHSII}, Definition 1.20]\label{defn:hhs_automorphism}
An \emph{automorphism} of the hierarchically hyperbolic space 
$(X,\mathfrak S)$ is a hieromorphism 
$f\colon (\X,\mathfrak S)\rightarrow(\X,\mathfrak S)$ such   
that $f^\diamondsuit$ is bijective and each $f^*(U)$ is an 
isometry.

The finitely generated group $\MCG$ is \emph{hierarchically hyperbolic}
if there exists a hierarchically hyperbolic space $(\X,\mathfrak
S)$ on which \(\MCG\) acts by automorphisms of hierarchically hyperbolic spaces,  so that the uniform quasi-action of $G$ on $\X$ is metrically proper and cobounded
and $\mathfrak S$ contains finitely many $\MCG$--orbits.  Note that if
$\MCG$ is hierarchically hyperbolic by virtue of its action on the
hierarchically hyperbolic space $(\X,\mathfrak S)$, then
$(\MCG,\mathfrak S)$ is a hierarchically hyperbolic structure with
respect to any word-metric on $\MCG$; for any $U\in\mathfrak S$ the 
projection is the
composition of the projection $\X\rightarrow C
U$ with a $\MCG$--equivariant quasi-isometry
$\MCG\rightarrow\X$.  In this case, $(\MCG,\mathfrak S)$ (with the
implicit hyperbolic spaces and projections) is a \emph{hierarchically
hyperbolic group structure}.
\end{Def}

In what follows, we will construct a weak factor system for the Cayley graph of a hyperbolic group \(G\). The weak factor system will be \(G\) equivariant, indeed, it will consists of cosets of subgroups of \(G\). Thus, the HHS structure on the Cayley graph of \(G\) will induce a HHG structure on \(G\).

\subsection{Constructing a factor system}

Fix a generating set for \(G\). All the following results are quasi-isometric invariant and hence holds for all finite sets of generators.
We start by recalling some well known properties of quasiconvex subgroups of hyperbolic groups. All the following facts and lemmas are proven in \cite{BridsonHaefliger}.

\begin{lemma}
If \(G\) is a hyperbolic group and \(H,J\) are \(K\)--quasiconvex subgroups, then \(H\cap J\) is quasiconvex, with quasiconvexity constant depending on \(G\) and \(K\).
\end{lemma}

\begin{lemma}
	If \(G\) is a hyperbolic group and \(H\) is a \(K\)--quasiconvex subgroup, then, for each \(c \in G\),  
	\(cHc^{-1}\) is quasiconvex, with quasiconvexity constant depending on \(G\), \(K\) and \(c\).
\end{lemma}

\begin{prop}\label{prop: existence infinite order elements}
An infinite hyperbolic group contains an infinite order element.
\end{prop}

\begin{lemma}\label{lem: infinite order elements are quasi-geodesics}
	If \(G\) is a hyperbolic group and \(g\) is an infinite-order element of \(G\), then \((g^n)_{n\in \ZZ}\) is a quasi-geodesic.
\end{lemma}

\begin{rmk}\label{rmk:Koenigs Lemma}
As a consequence of Proposition \ref{prop: existence infinite order elements} and Lemma \ref{lem: infinite order elements are quasi-geodesics} we have that given a hyperbolic group \(G\), and \(H_1, H_2\)  \(E\)--quasi-isometrically embedded subgroups of \(\G\), there exists \(R = R(\delta, E)\) such that if \(d_{\mathrm{Hauss}}(H_1, H_2) < \infty\), then \(d_{\mathrm{Hauss}}(H_1, H_2) \leq R\).
\end{rmk}

\begin{lemma}\label{lem: distance lateral conjugate}
Let \(G\) be a group and \(H, J\) be two subgroups. Then for any \(g \in G\), we have that \(H \cap gJg^{-1} \subseteq N_{2|g|} (p_{H}(gJ))\).
\begin{proof}
Let \(gyg^{-1} \in H \cap gJg^{-1}\). Then \(d(gyg^{-1}, gy) \leq |g|\). Thus, \(d(gy, H) \leq |g|\). Thus there is a point \(x\in p_{H}(gJ)\) with \(d(x, gy) \leq |g|\). By triangular inequality, \(d(gyg^{-1}, x) \leq 2|g|\).
\end{proof}
\end{lemma}

We want now to establish some relations between conjugates and cosets.

\begin{lemma}\label{inclusion projection conjugation} 
Let \(G\) be a \(\delta\)--hyperbolic group with respect to a fixed generating set. Let \(H, J\) be \(K\)--quasiconvex subgroups.
Then there exists \(D=D(\delta, K)\) such that the following holds. For each \(a, b \in G\) for which \(\mathrm{diam}\left(p_{aH}(bJ)\right) \geq 2(8\delta +2K) +1\), then 
\[d_{\mathrm{Haus}}\left( p_{aH} (bJ),  a\left( H\cap a^{-1}bJb^{-1} a\right)\right)\leq D.\]
\begin{proof}
	Since left multiplication by \(a^{-1}\) is an isometry, the above is equivalent to showing that for each \(c \in G\) such that \(\mathrm{diam}(p_H(cJ)) \geq 2(8\delta +2K) +1\), we have:
	\[d_{\mathrm{Haus}}\left(p_{H} (cJ), N_D\left( H\cap cJc^{-1}\right)\right) \leq D.\]
	 By a quadrilateral argument, if \(d(H, cJ) > 2\delta + 2K\), we get that \(\mathrm{diam}\left( p_{H}(cJ)\leq 8\delta + 2K\right)\).
	 Thus \(d(H, cJ) \leq 2\delta + 2K\) and hence \(|c| \leq 2\delta + 2K\). By Lemma \ref{lem: distance lateral conjugate}, we get that 
	 \(H \cap cHc^{-1} \subseteq N_{2|c|} (p_{H}(cJ))\).
	Now we want to show that there is a uniform \(D_1\) such that \(p_H(cJ) \subseteq N_{D_1}(H \cap cJc^{-1})\). This will prove the Lemma.
	Let \(a \in p_H(cJ)\). By the assumptions on the diameter, there is \(a' \in p_H(cJ)\) with \(d(a, a') > 8\delta +2K\). Let \(C= 4\delta + K\). 
	A quadrilateral argument shows that there is a point \(x \in H\) with \(d(a, x) \leq C\) and \(d(x, cJ) \leq 2\delta + 2K\). That is: 
	\[p_H(cJ) \subseteq N_C (P_H (N_{2\delta + 2K}(H) \cap cJ)).\]
	We want to show that there is \(R\) such that \(N_{2\delta + 2K}(H) \cap cJ \subseteq N_R(H\cap cJc^{-1})\). Since projections are quasi-Lipschitz
	(see Lemma \ref{nbhd inclusion of proj}), this will imply the claim.
	
	Let \(B = B_{2\delta + 2K} (\1) \) be the ball of radius \(2\delta\) around the identity. 
	Note that \(N_{\delta + 2K}(H)= \bigcup_{g \in B} Hg\). For each \(g \in B\) such that \(Hg \cap cJ \neq \emptyset\), choose once and for all 
	an element \(y_g \in Hg \cap cJ\). Let \(R = \mathrm{max} \{|y_g| \mid   g\in B\}\). This is well defined since \(B\) is finite.
	Consider \(x \in  N_{\delta + 2K}(H) \cap cJ\). Then there exists \(g \in B\), \(h \in H\) and \(j \in J\) such that 
	\(x= hg = cj\). Similarly there are \(h' \in H\) and \(j' \in J\) such that \(y_g = h'g = cj'\). 
	Then it is easily seen that \(xy_g^{-1} \in H\cap cJc^{-1}\), and thus 
	\[x \in (H \cap cJc^{-1}) y_g = N_{|y_g|}(H\cap cJc^{-1}).\]
	Maximizing over \(B\) yields \(N_{\delta + 2K} (H) \cap cJ \subseteq N_R (H \cap cJc^{-1})\).
\end{proof}
\end{lemma}

We recall the following Theorem.

\begin{thm}[\cite{brady2000finite}]
Let \(G\) be a group that is \(\delta\)--hyperbolic with respect to a fixed finite set of generators. 
Then there exists \(R=R(\delta)\) such that every finite subgroup can be conjugated to lie in the \(R\)--ball around the identity.
\end{thm}

\begin{corollary}\label{corollary Brady}
Given a hyperbolic group \(G\), there is \(\Delta\) such that every finite subgroup of \(G\) has at most \(\Delta\) elements.
\end{corollary}

Given two quasiconvex subgroups \(H\) and \(J\), there are only finitely many cosets \(cJ\) such that \(8\delta + 2 K < \mathrm{diam}(p_H(cJ))< \infty\), and this extends to any finite family. In particular, the next Lemma is trivially satisfied whenever it is applied to a finite family, which will turn out to be the case in our situation. 
However, we record it because the result is slightly more general and the proof does not rely on the above fact.

\begin{lemma}\label{property xi}
Let \(G\) be a \(\delta\)--hyperbolic group, and let \(H, J\) be \(K\)--quasiconvex subgroups of \(G\). Then 
there exists  \(\xi= \xi (\delta, K, G)\) such that for each \(c \in G\),  if \(\mathrm{diam}(p_H(cJ)) \geq \xi\), then \(\mathrm{diam}(p_H(cJ))= \infty\).
	\begin{proof}
	Let \(B\) be the ball of radius \(2\delta + 2K\) around the identity.
	Consider the elements of \(cJ\) that are at distance at most \(2\delta +2K\) from \(H\), that is the set \(\bigcup_{g \in B} cJg \cap H\).
	By hyperbolicity, increasing \(\xi\) we can assume that this set is arbitrarily large. Note that for every \(g \in B\) 
	and every pair of elements \(x,y\in cJg \cap H\), one has that \(xy^{-1}  \in cJc^{-1} \cap H\). 
	Since the ball \(B\) contains only finitely many elements, 
	by the pigeonhole principle we can assume that there is \(g \in B\) such that \(|cJg \cap H| > \Delta +1\), where \(\Delta\) is the constant
	provided by Corollary \ref{corollary Brady}. 
	In particular, we can find  \(x,y \in cJg \cap H\) such that \(xy^{-1}\) has infinite order.
	Thus \(|H \cap cJ c^{-1}|= \infty\) and then \(\mathrm{diam}(p_H(cJ))= \infty\).
	\end{proof}
\end{lemma}

We want to establish how subgroups relates with coarse inclusion.
\begin{lemma}\label{lem: coarse inclusion => infinite intersection}
	Let \(H,J\) be infinite quasiconvex subgroups of a \(\delta\)--hyperbolic group \(G\) and let \(a,b\) be elements of \(G\).
	Then if \(aH \precsim bJ\), then \(|aHa^{-1} \cap bJb^{-1} | = \infty\).
	\begin{proof}
	Let \(h\) be an infinite order element of \(H\). By Lemma \ref{lem: infinite order elements are quasi-geodesics}, \((h^n)_n\)
	is a quasi-geodesic of \(H\), and thus \((ah^{n})_n\) is a quasi-geodesic of \(aH\). By hypothesis, \((ah^n)_n\) is contained in a uniform
	neighborhood of \(bJ\). In particular, there is \(C = C(G, H)\) such that for each \( n \in \ZZ\), there is an element
	\(g\in B_C(\1)\) such that  \(ah^ng \in bJ\).
	Since there are only finitely many such \(g\), we get that there exist different numbers \(n, m \) and \(g_0 \in B_C(\1)\) such that
	\(ah^ng_0 \in bJ\) and \(ah^mg_0 \in bJ\). Thus \(ah^ng_0g_0^{-1}h^{-m}a^{-1} \in aHa^{-1}\cap bJb^{-1}\)
	is an element of the intersection that has infinite order. 	
	\end{proof}
\end{lemma}

\begin{Def}[Proximal pair]
Let \(\F=\{F_1, \dots, F_n\}\) be a finite family of subgroups of a group \(G\). 
Let \(F_i, F_j\) be elements of \(\F\), and let \(g\in G\).
We say that \(\left(F_i, gF_jg^{-1}\right)\) form a \emph{proximal pair} if \(\left| F_i \cap gF_jg^{-1}\right|= \infty\).
We define \(\overline{\mathrm{Prox}(\F)}\) to be the set of intersections of proximal pairs, namely:
\[\overline{\mathrm{Prox}(\F)} =\{F_i \cap gF_jg^{-1} \mid  \left| F_i \cap gF_jg^{-1}\right|= \infty\}\]

\end{Def}

Let \((F_i, gF_jg^{-1})\) be a proximal pair. It is clear that for each \(f \in F_i -\1\), we have that \(\left(F_i, fgF_j (fg)^{-1}\right)\) is also a proximal pair, thus the set \(\overline{\mathrm{Prox}(\F)}\) is, in general, infinite. 
However, it contains only finitely many conjugacy classes.

\begin{lemma}\label{finiteness of proximal elements}
Let \(G\) be a hyperbolic group and \(\F=\{F_1, \dots, F_n \}\) be a finite family of quasiconvex subgroups of \(G\).
Then \(\overline{\mathrm{Prox}(\F)}\) contains finitely many conjugacy classes.
\begin{proof}
Fix a finite set of generators on \(G\), and let \(\delta\) be the hyperbolicity constant.

We recall that for any two \(F_i, F_j\), we have \(F_i \cap gF_jg^{-1} \subseteq N_{2|g|} (p_{F_i}(gF_j))\).
Then we will show that, up to left multiplication by an element of \(F_i\), there are only finitely many choices for \(gF_j\) such that \(\mathrm{diam}(p_{F_i}(gF_j)) = \infty\), which implies the result. 

So let \(gyg^{-1} \in F_i \cap gF_jg^{-1}\). Then \(d(gyg^{-1}, gy) \leq |g|\). Thus, \(d(gy, F_i) \leq |g|\). Thus there is a point \(x\in p_{F_i}(F_j)\) with \(d(x, gy) \leq |g|\). By triangular inequality, \(d(gyg^{-1}, x) \leq 2|g|\), which proves the first claim. Let \(K\) be such that both \(F_i\) and \(F_j\) are \(K\)--quasiconvex. By a quadrilateral argument, if \(d(F_i, gF_j) > 2\delta + 2K\), we get that \(\mathrm{diam}\left( p_{F_i}(gF_j)\right)\leq 8\delta + 2K\). Thus, up to left multiplication by an element of \(F_i\), there are only finitely many \(gF_j\) such that \(\mathrm{diam}p_{F_i}(gF_j)\) is infinite, and thus finitely many conjugacy classes.

\end{proof}
\end{lemma}

This motivates the following definition:

\begin{Def}
For each conjugacy class \([H]\) of elements of \(\overline{\mathrm{Prox}(\F)}\), choose once and for a representative whose quasiconvexity constant is minimal among the elements of the class. 
Then we define \(\mathrm{Prox}(\F)\) to be the set of such representatives.
\end{Def}
Note that by Lemma \ref{finiteness of proximal elements}, if \(\F\) is a finite family, so is \(\mathrm{Prox}(\F)\).

In what follows, we will describe an inductive "closure process" for the family \(\F\). 

\begin{Def}\label{Definition families}
Given a finite family of infinite quasiconvex subgroups \(\F=\{F_1, \dots, F_n\}\), we will inductively describe a sequence of families as follows:
\begin{itemize}
\item[b)] Set \(\F^{(0)}=\F\).
\item[i)] Given \(\F^{(i)}\), define \(\F^{(i+1)}\) as \( \mathrm{Prox}(\F^{(i)})\).
\end{itemize}
\end{Def}
\begin{rmk}
Since \(\F\) consisted of only infinite elements, it is easy to see that for each \(i\),  \(\F^{(i)} \subseteq \F^{(i+1)}\).
Moreover, by Lemma \ref{finiteness of proximal elements}, we have that if \(\F^{(i)} \) is finite, so is \(\F^{(i+1)}\). In particular, for each finite set of generators of \(G\), there exists \(K = K(i)\) such that all the elements of the family
\(\F^{(i)}\) are \(K\)--quasiconvex. 
\end{rmk}

The motivating property of the above definition is the following.

\begin{lemma}\label{hausdorff distance for families}
	Let \(\F^{(i)}\) be constructed as in \ref{Definition families}. 
	There exists \(D = D (\F^{(i)})\) such that the following holds. For every \(a, b \in G\) and \(H, J \in \F^{(i)}\) 
	satisfying \(|p_{aH} (bJ)| = \infty\), 
	there exists \(c \in G\) and \(E \in \F^{(i+1)}\) such that \[d_{\mathrm{Haus}}(p_{aH} (bJ), cE) \leq D.\]
	\begin{proof}
	As previously remarked, there is \(K\) such that all elements of \(\F^{(i)}\) are \(K\)--quasiconvex.
	In particular, we can apply Lemma \ref{inclusion projection conjugation} to get that there is a uniform \(D\) such that
	\[d_{\mathrm{Haus}}\left(p_{aH} (bJ), a\left( H\cap a^{-1}bJb^{-1} a\right)\right) \leq D.\]
	In particular, \(|H \cap a^{-1}bJb^{-1}a| = \infty\), thus \((H, a^{-1}bJb^{-1}a)\) is a proximal pair for the family \(\F^{(i)}\).
	In particular, there exists a representative \(E= H \cap gJg^{-1}  \in \mathrm{Prox}(\F^{(i)})= \F^{(i+1)}\) and an element \(h \in G\) 
	such that \[H \cap a^{-1}bJb^{-1}a = h E h^{-1}.\]
	
	Thus, we have a uniform estimate of the Hausdorff between \( p_{aH} (bJ)\) and a conjugate of an element of \(\F^{(i+1)}\).
	Since \(d_{\mathrm{Haus}}(hE, hEh^{-1}) \leq |h|\), by Remark \ref{rmk:Koenigs Lemma} we get that there exists a uniform \(R\) such
	that \(d_{\mathrm{Haus}}(hE, hEh^{-1}) \leq R\). Setting \(c=ah\), we get the claim.
	\end{proof}	
\end{lemma}

\begin{thm}[\cite{GMRS1998widths}]\label{GMRS}
	Let \(G\) be a hyperbolic group, and \(\{F_1, \dots, \F_n\}\) be a finite family of quasiconvex subgroups. 
	Then there exists \(c\) such that for each collection 
	\(\{g_{\alpha_i}F_{\alpha_i}g_{\alpha_i}^{-1}\}_{i=1}^c\) of \(c\) distinct conjugates, 
	the intersection \[\bigcap_{i=1}^{c} g_{\alpha_i}F_{\alpha_i}g_{\alpha_i}^{-1}\] is finite.
\end{thm}

\begin{corollary}\label{finite depth of families}
Given a hyperbolic group \(G\) and a family \(\{F_1, \dots, F_n\}\) of quasiconvex subgroups, there is \(M \in \NN\) such that for each 
\(n >M\), \(\F^{(n)}= \F^{(M)}\).
\end{corollary}

\begin{thm}\label{thm:wfs for groups}
Let \(G\) be a \(\delta\)--hyperbolic space, and let \(\{F_1, \dots, F_n\}\) be a family of quasiconvex subgroups. Let \(M\) be the constant of Corollary \ref{finite depth of families}, and let \(\F_{\mathrm{cos}}\) be set of all left-cosets of the family \(\F^{(M)}\). Then \(\F_{\mathrm{cos}}\) is a weak-factor system for \(G\) with respect to each finite set of generators of \(G\).
Moreover, the natural action of \(G\) on itself extend to an action on \(\F_{\mathrm{cos}}\) with finitely many orbits. 
\begin{proof}
	As remarked, for each choice of generators of \(G\), the elements of the family \(\F^{(M)}\) are uniformly quasiconvex.
	In particular, this is preserved by left-multiplication. Thus \(\F_{\mathrm{cos}}\) is a family of uniformly quasiconvex subspaces of \(G\).
	It is clear that the \(G\) action preserves \(\F_{\mathrm{cos}}\). Moreover, since \(\F^{(M)}\) is finite, there are only finitely many \(G\)--orbits.
	\begin{enumerate}
	\item By Lemma \ref{lem: coarse inclusion => infinite intersection} we have that \(|a_1F_1a_1^{-1} \cap a_2 F_2 a_2^{-1}| = \infty\). 
		Since \(a_1 F_1 a_1^{-1} \subseteq N_{|a_1|}(a_1 F_1)\), and since \(\precsim\) is a transitive relation, we get that 
		\( \left(a_1F_1a_1^{-1} \cap a_2 F_2 a_2^{-1}\right) \precnsim a_3F_3\). 
		Proceeding in this way, we get that 
		\[\left| a_1F_1a_1^{-1} \cap \cdots \cap a_nF_na_n^{-1}\right|= \infty.\]
		Since we require the coarse inclusions to be proper, all the conjugates are distinct. 
		Hence, by Theorem \ref{GMRS}, we get that the length of every such chain is uniformly bounded.
	\item Lemma \ref{property xi} and Lemma \ref{hausdorff distance for families} give that there are \(\xi\) and \(B\) depending on the
		choice of generators such that given \(H,J \in \F_{\mathrm{cos}}\),  either \(\mathrm{diam}(p_H(J)) < \xi\), or there exists 
		\(L \in \F_{\mathrm{cos}}\) such that \(d_{\mathrm{Haus}}(L, p_H(J)) \leq B\).
	\item For each coset \(aH \in \F_{\mathrm{cos}}\) and element \(x \in aH\), if \(h\) is a generator of \(H\) that is not a torsion element, 
		then \(xa^n\) is a quasi-geodesic that contains \(x\). Since the elements of \(\F^{(M)}\) are finite, one can uniformly estimate the distance 		
		between \(x\) and an arbitrarily long geodesic segments with endpoints in \(aH\).
	\end{enumerate}

\end{proof}
\end{thm}

As a corollary we get the main result of this section.

\begin{corollary}\label{cor: HHG structure on G}
Let \(G\) be a hyperbolic group and let \(\F= \{F_1, \dots, F_N\}\) be a finite family of infinite quasiconvex subgroups. 
Let \(\sim\) be the equivalence relation between subset of \(G\) given by having finite distance in \(\mathrm{Cay}(G)\)  (note that does not depend on the choice of generators).
Then there exists a finite family of quasiconvex subgroups \(\F^{(M)}\) that contains \(\F\) such that if \(\F_{\mathrm{cos}}\) is the set of cosets of \(\F^{(M)}\), then \((G, \F_{\mathrm{cos}} /_\sim )\) is a hierarchically hyperbolic group structure on \(G\). 
\end{corollary}
\begin{proof}
By Theorem \ref{thm:wfs for groups}, \(\F_{\mathrm{cos}}\) is a weak factor system for the Cayley graph of \(G\) on which \(G\) acts equivariantly. Proposition \ref{prop: H factor system}, yields an HHS structure on the Cayley graph of \(G\) on which \(G\) acts equivariantly,  with indexing set \(\F_{\mathrm{cos}}/_\sim\). Thus \((G, \F_{\mathrm{cos}} /_\sim )\) is a hierarchically hyperbolic group.

\end{proof}

The following is a useful consequence of Corollary \ref{cor: HHG structure on G}.

\begin{corollary}
Let \(G\) be a hyperbolic group, and let \(\F= \{F_1, \dots , F_N\}\) be a finite family of quasiconvex subgroups. 
Let \((G, \S)\) be the HHG structure on \(G\) provided by Corollary \ref{cor: HHG structure on G}. Then for each \(F \in \F\), we have that \((F, \S_{[F]})\) is an HHG structure on \(F\) Moreover the inclusion map \(i \colon F \hookrightarrow G\) induces an injective hieromorphism between \(F\) and \(G\), such that  for each \([H] \in \S_{[F]}\), we have \(i^\diamondsuit ([H]) = [H]\).
\begin{proof}
By Corollary \ref{cor:WeakSub-FactorSystem} we have that \((CF, \S_{[F]})\) is an HHS and that the inclusion map \(i\) satisfy the requirements.
Since \(\mathrm{Cay}(F)\) is contained in \(CF\) and it is quasi-isometric to it, and since \(\mathrm{(F)}\) is preserved under the \(F\) -action, it follows that \((F, \S_{[F]})\) is an HHG.
\end{proof}
\end{corollary}

\section{Boundaries}\label{sec:boundaries}

Let \(\X\) be a hyperbolic space and let \(\F\) be a factor system for \(\X\). 
We will exploit the HHS structure on \(\X\) induced by \(\F\) to obtain a description of the boundary of \(\X\) that relates it to the boundaries of the various \(CU\) for \(U \in \F \cup \{\X\}\). We recall that for each \(U \in \F \cup \{\X\}\), the space \(CU\) is obtained from  \(U\) by coning off all elements \(V\) of \(\F\) that are strictly contained in \(U\).
 More precisely, we will describe a topology on 
\[\bigcup_{F \in \F \cup \X} \partial_{\infty} CU\] that makes it homeomerphic to the boundary of \(\X\) (see Theorem \ref{thm:main_result_bdry}). 
It is worth to remark that both the topology and the spaces \(CU\) can be defined without the hierarchical machinery. 

A special case of this result it the case of a space \(\X\) which is hyperbolic relatively to a set of uniformly hyperbolic spaces \(\F\), which was treated by Hamenst\"adt in \cite{Hamenstadt_Hyp_rel_hyp_graphs}. Indeed, in this case the space \(\X\) has to be hyperbolic (Theorem 2.4 of \cite{Hamenstadt_Hyp_rel_hyp_graphs}) and, due to a result of Dru\c{t}u and Sapir (\cite{DrutuSapir}), the Hausdorff distance between any two elements of \(\F\) which have infinite diameter must be infinite. Thus, up to discarding those elements of \(\F\) that are bounded, we have that \(\F\) is a factor system for \(\X\). Since the Gromov boundary of a bounded set is empty, our result recovers the previous result in this case. 

Moreover, in the case when \(\X\) is hyperbolic relatively to \(\F\), the explicit description of the Gromov boundary allows us to obtain an explicit description of the Bowditch boundary (under the mild hypothesis of \(\X\) being a proper metric space). Indeed, the Bowditch boundary is obtained from the Gromov boundary collapsing all the boundaries of the various \(CU\), for \(U \in \F\), to a point (see Theorem \ref{thm:Bowditch Boundary}). In the case when \(\X\) is a hyperbolic group and \(\F\) a family of peripheral subgoups, this recovers a well known result, proved by Tran in \cite{TranRelationsBetween} and, independently, by Manning in \cite{ManningBowditch_Boundary}. The former points out that a proof of this result can also be obtained from \cite{GerasimovFloyd, GerasimovPotyagailo} or \cite{MatsudaOguniYamagata}.

\subsection{Fixing notations: The Gromov boundary of a hyperbolic space}

It is a well known fact that the Gromov boundary of a hyperbolic space has several different characterizations. We will briefly recall the definitions and conventions used. For the proofs of the statements and a more precise exposition, we refer to \cite[Chapter III.H]{BridsonHaefliger}, or to the survey \cite{KapovichBenakliBoundary}.

Firstly, however, we will recall the definition of hierarchy path for an HHS.
\begin{Def}[Coarse map, unparametrized quasi-geodesic]
Let \(X\) be a metric space. A \emph{coarse map} \(f \) from \(Y\) to \(X\) is a map \(f \colon Y \rightarrow 2^{X}\) such that the image of each point has uniformly bounded diameter. A coarse \(f\colon [0,l] \rightarrow X\) is an \emph{unparametrized \((D,D)\)--quasi-geodesic} if there exists an increasing function \(g\colon [0, L] \rightarrow [0,l]\) such that \(f\circ g\) is a \((D,D)\)--quasi-isometric embedding, and for each \(x,y\in[0,L]\) with \(|x-y|\leq 1\), we have \(\mathrm{diam}(f(g(\{x,y\})) \leq D\).
\end{Def}

\begin{Def}[Hierarchy path, \cite{HHSII}, Definition 4.2]
Let \((\X, \S)\) be a hierarchically hyperbolic space. A path \(\gamma\) of \(\X\) is a \emph{D-hierarchy path} if 
\begin{enumerate}
\item \(\gamma\) is a \((D,D)\)--quasi-geodesic of \(\X\);
\item for each \(U \in \S\) the projection \(\pi_U(\gamma)\) is an unparametrized  \((D, D)\)--quasi-geodesic.
\end{enumerate}
\end{Def}

\begin{thm}[Existence of hierarchy paths, \cite{HHSII}, Theorem 4.4]\label{thm: Existence of hierarchy paths}
Let \(\X\) be a hierarchically hyperbolic space. 
Then there exists \(D_0\) so that any \(x,y \in \X\) are joined by a \(D_0\)--hierarchy path.
\end{thm}

Let \(\X\) be a \(\delta\)--hyperbolic HHS and let \(\gamma_1, \gamma_2\) be two quasi-geodesic rays of \(\X\). We say that \(\gamma_1\) and \(\gamma_2\) are equivalent if their Hausdorff distance is finite. Let \(D\) be such that any two points of \(\X\) can be joined by a \(D\)--hierarchy path.
Since \(\X\) is hyperbolic, there is a constant \(H_D\) such that given two \(D\)--quasi-geodesic segments that share the same endpoints, their Hausdorff distance is at most \(H_D\). 
If \(\gamma, \eta\) are quasi geodesic rays that represent the same point in the Gromov boundary, then there is \(n= n(d(\gamma(0), \eta(0))\) such that the Hausdorff distance between \(\gamma\) and \(\eta\) outside the ball of radius \(n\) around \(\gamma(0)\) is at most \(H_D\). 
Let \(k = D+H_D\), and let \(H_k\) be such that if the Hausdorff distance between 2 \(k\)--hierarchy ray is finite, then it is at most \(H_k\).
We can identify the Gromov boundary with the set of equivalence classes of \(k\)--hierarchy rays. The reason why we choose \(k\) to be larger than \(D\) is that in this way we are allowed to "perturb hierarchy rays". 

Given a point \(x \in \X\) and \(p \in \partial_\infty \X\), there is always a \(D\)--hierarchy ray \(\gamma \in p\) such that \(\gamma(0) = x\). In that case we say that \(\gamma\) connects \(x\) and \(p\).

Let \(\overline{\X}_\infty\) be the union \(\partial_\infty \X \cup \X\). There is a natural topology on \(\overline{\X}_\infty\) such that \(\X\) is embedded in \(\overline{\X}_\infty\), and the latter is compact in case \(\X\) is proper. We will briefly recall how the topology of \(\overline{\X}_\infty\) is defined via prescribing a neighborhood base for each point. If \(x \in \X\), then we consider the standard base obtained by the metric of \(\X\). For a point at infinity, we recall the following lemma, that can be taken as a definition.

\begin{lemma}[Neighborhoods at infinity \cite{BridsonHaefliger} Chapter III.H: Theorem 1.7, Lemma 3.6]\label{lem:nbhd_point_at_infnity}
Let \(X\) be a \(\delta\)--hyperbolic space and \(x_0\) be a point of \(X\). Let \(r > 2 H_k\). Let \(\eta\) be a \(k\)--quasi-geodesic ray representing a point \(p \in \partial_\infty X\) and for each positive integer \(n\) let \(V_n(\eta)\) be the set of points \(q \in \overline{X}_\infty\) such that for all \(k\)--quasi-geodesic rays \(\gamma\) connecting \(x_0\) and \(q\) there is a point \(x\) on \(\eta\) such that \(d(x, \eta(0) ) \geq n\) and \(\gamma \cap B_r(x) \neq \emptyset\). Then \(\{V_n(\eta) \mid  n \in\NN\}\) is a fundamental system of neighborhoods for \(p\) in \(\partial_\infty X \cup X\). 
\end{lemma}

\subsection{The HHS-boundary of a hyperbolic space} 

We recall the definition of HHS-boundary of a space introduced in \cite{HHSBoundaries} and some important results.

\begin{Def}[Support set, boundary point, HHS-boundary; \cite{HHSBoundaries}]\label{def:HHS-boundary}
Let \((\X, \S)\) be a hierarchically hyperbolic spaces. A \emph{support set} \(\overline{U} \subseteq \S\) is a set with \(U_i \perp U_j\) for all distinct \(U_i, U_j\in \overline{U}\). A \emph{boundary point} with support \(\overline{U}\) is a formal sum
\[p= \sum_{U \in \overline{U}} a_{U} p_{U},\] where \(p_U \in \partial_\infty CU\) and \(a_{U} > 0\), with the requirement that \(\sum_{\overline{U}} a_U=1\).
The \emph{HHS-boundary} \(\partial \X\) of \((\X,\S)\) is the set of all boundary points.
\end{Def}

It is a well-known fact that, in a hierarchically hyperbolic space \((\X, \S)\), each family of pairwise orthogonal elements has uniformly bounded size. See for instance \cite[Lemma 1.4]{HHSBoundaries}. In particular, the boundary points consist of uniformly finite sums.

\begin{thm}[\cite{HHSBoundaries}]\label{thm:Boundaries}
There is a topology on \(\partial \X\) such that the following holds:
\begin{itemize}
\item For each \(U \in \S\), the inclusion \( \partial_\infty CU \hookrightarrow \partial \X\) is an embedding.
\item The boundary \(\partial\X\) is closed in \(\X\).
\item The space \(\overline{\X} = \partial \X \cup \X\) is 
	\begin{rules}
	\item Hausdorff,
	\item separable, in case \(\X\) is separable,
	\item compact, in case \(\X\) is proper. 
	\end{rules}
\end{itemize}
\end{thm}

There is an explicit description for the topology of the HHS-boundary of a general hierarchically hyperbolic space, and we refer to \cite{HHSBoundaries} for the precise definition. 

In the case of hyperbolic HHS, however, the HHS-boundary and its topology admit a significantly simpler description. This is because, under the mild assumption that  for all \(U \in \S\), \(CU\) has infinite diameter, there are no pairwise orthogonal elements (see \cite[Lemma 4.1]{HHSBoundaries}, or \cite[Subsection 5.1]{HHSII}). 
In this case, Definition \ref{def:HHS-boundary} translates as:
\[\partial\X = \bigcup_{U \in \S} \partial_\infty CU.\]
We will describe the topology of \(\partial \X\) via prescribing neighborhood basis at each point.

\begin{Def}[Boundary projections, \cite{HHSBoundaries}]\label{def:boundary_projections}
Let \((\X, \S)\) be a hyperbolic HHS and let \(U \in \S\). Let \(x\) be a point of \(\overline{\X}\). We define the projection \(\partial \pi_U \colon \overline{\X} \rightarrow \partial_\infty CU \cup CU\) as:
\begin{enumerate}
\item The projection \(\pi_U\) if \(x \in \X\).
\item The identity map if \(x \in \partial_\infty CU\).
\item The set \(\rho^V_U\) if \(x \in \partial_\infty CV\) and either \(V\) is properly nested in \(U\) or \(V \pitchfork U\).
\item The coarse map \(\rho^{\partial V}_U \colon \partial CV \rightarrow 2^{CU}\) defined as follows in all the other cases. Consider a \(D\)--hierarchy ray \(\gamma\) connecting \(\rho^U_V\) with \(q \in \partial CV\). 
Let \(E\) be the constant of the bounded geodesic image property.  We set \(\rho^{\partial V}_U (q)= \rho^V_U\left(\gamma - \left(N_{E+1} (\rho^U_V)\right)\right)\), which, by the bounded geodesic image property, has uniformly bounded diameter. 
\end{enumerate}
\end{Def}

\begin{Def}[Neighborhood basis for the topology, \cite{HHSBoundaries}]\label{def:topology_HHS_bdry}
Let \(p\) be a point in \(\partial_\infty CU\), and let \(O_p\) be an open neighborhood for \(p\) in the cone-topology for \(CU \cup \partial_\infty CU\). We define \(N_{O_p}(p) \subseteq \partial \X\) to be the set:

\[N_{O_p} (p) = \{ q\in \partial_\infty \X \mid  {\partial}\pi_U(q) \cap O_p \neq \emptyset\}.\]

\end{Def}

We declare the sets \(N_{O_p} (p)\) to be a neighborhood basis at \(p\), and this topology coincides (indeed, it is just a special case of the definition) with the topology of the HHS-boundary.

However, in this case much more it is true:

\begin{thm}[\cite{HHSBoundaries}, Theorem 4.3]\label{thm:HHS_Gromov_boundary}
Let \((\X, \S)\) be a hyperbolic HHS and let \(\overline{\X}_\infty\) be the union \(\partial_\infty \X \cup \X\). 
Then the identity map \(\X \rightarrow \X\) extends uniquely to a homeomorphism \(\overline{\X}_\infty \rightarrow \overline{\X}\).
\end{thm}


\subsection{Explicit description when the HHS structure comes from a factor system}

If \((\X, \S)\) is the hierarchically hyperbolic space structure obtained applying Theorem \ref{thm:main result for factor system}, we have a very explicit description of the projections \(\pi_U\) and the spaces \(\rho^U_V\). This allows us to give a more explicit description of the HHS boundary, and thus of the Gromov boundary.

Unraveling the construction of Theorem \ref{thm:main result for factor system}, we have that if \(\F\) is a factor system for \(\X\), then the elements of the set \(\S\) are in bijection with \(\F \cup \{\X\}\). In particular, each index \(U \in \S\) is naturally associated to a subspace \(F_U \in \F \cup \{\X\}\) of \(\X\) and the associated \(\delta\)--hyperbolic space \(CU\) is an appropriate cone-off of \(F_U\).
We also recall that the maps \(\pi_U\) and \(\rho^V_U\) are defined as closest point projection in \(\X\) on the subspace \(F_U\). Indeed, this defines a map with image in \(2^{CU}\) since \(CU\) is the cone-off of \(F_U\).

Since \(\partial \X = \bigcup_{U \in \S} \partial_\infty CU\), each point \(p \in \partial \X\) corresponds to an equivalence class of hierarchy rays in some \(CU\) and viceversa. 
Given a point \(q \in \partial_\infty CV\), and \(U \in\S\), we observe that the projection \(p_{F_U}(q)\) is coarsely well defined, in the sense that for each pair of representatives \(\gamma, \eta\) of \(q\), there is \(N\) such that,  for each \(n > N\), we have \(d_{\mathrm{Haus}}\left(p_{F_U}(\gamma[n, \infty)), p_{F_U}\big(\eta([n, \infty))\big)\right)\leq c\) for some uniform \(c\). 

It is easy to see that for each point \(p \in \partial_\infty CU\), the Hausdorff distance between \(p_{F_V}(p)\) and \(\partial \pi_V (p)\) is uniformly bounded, where \(\partial \pi_V (x)\) is as in Definition \ref{def:boundary_projections}.

We claim that substituting the projections \(\partial\pi_V\) with the shortest distance projections \(p_{F_V}\) in Definition \ref{def:topology_HHS_bdry} does not change the resulting topology. 
This will allow to give a non trivial decomposition of the Gromov boundary as the union of the boundaries of the various \(CU\), with \(U \in \S\). 

For brevity, given \(q \in \partial_\infty CV\), \(U \in \S\) and a set \(O \in CU\), we say that \(p_{F_U}(q) \cap O \neq \emptyset\) if there exists a representative \(\gamma \in q\) and \(N \in \NN\) such that for each \(n > N\), the projection \(p_{F_U} (\gamma([n, \infty))\) intersects \(O\). 
\begin{Def}[Alternative neighbourhood basis for the topology]\label{def:topology_bdry2}
Let \(p\) be a point in \(\partial_\infty CU\), and let \(O_p\) be an open neighborhood for \(p\) in the cone-topology for \(CU \cup \partial_\infty CU\). We define \(M_{O_p}(p) \subseteq \partial \X\) to be the set:

\[M_{O_p} (p) = \{  q\in \partial_\infty \X \mid   p_{F_U}(q) \cap O_p \neq \emptyset\}.\]
\end{Def}

The main result of this section is the following. 
\begin{thm}\label{thm:main_result_bdry}
Let \(X\) be a hyperbolic space and let \(\F\) be a factor system for \(X\). For each element \(U \in \F \cup \{X\}\) let \(CU\) be the cone-off of \(U\) with respect to all the elements \(V \in \F\) that are strictly contained in \(U\). Then the Gromov boundary of \(X\) decomposes as 
\[\partial_\infty X = \bigcup_{U \in \F \cup \{X\}} \partial_\infty CU,\]
with the topology described in Definition \ref{def:topology_bdry2}.
\end{thm} 

The proof of Theorem \ref{thm:main_result_bdry} is an easy consequence of Theorem \ref{thm:HHS_Gromov_boundary} and the following proposition.

\begin{prop}\label{prop:the_two_topologies_coincide}
The topologies for \(\overline{\X}\) of Definition \ref{def:topology_HHS_bdry} and Definition  \ref{def:topology_bdry2} agree.
\end{prop}

Before proving Proposition \ref{prop:the_two_topologies_coincide}, we present one useful application. 
\begin{corollary}\label{cor:bdry_dec_for_gps}
Let \(G\) be a hyperbolic group and let \(\F\) be a finite family of quasiconvex subgroups of \(G\). Let \((G, \S)\) be the HHG structure on \(G\) induced by \(\F\) (Corollary \ref{cor: HHG structure on G}). Then 
\[\partial_\infty G = \bigcup_{U \in \S} \partial_{\infty}CU.\]
\end{corollary}

\begin{proof}[Proof of Proposition \ref{prop:the_two_topologies_coincide}.]

We need to show that every neighborhood of one of the two basis contains a neighborhood of the other.

\textit{For all \(p\) and \(O_p\), there exist \(O_p'\) such that \(M_{O_p'} (p) \subseteq N_{O_p} (p)\).}

Set \(O_p' = O_p\) and let \(q \in M_{O_p}\). 
Suppose that \(p \in \partial_\infty CU\). If \(q\in \partial_\infty CU\), then the conclusion follows.
So suppose that \(q \in \partial_\infty CV\) for some \(V \neq U\). 
If \( V \pitchfork U\), then, since \(q \in M_{O_p}(p)\), we have that there is \(N\) and \(\gamma \in q\) such that for each \(n \geq N\), we have \(p_{F_U} (\gamma[n, \infty))\cap O_p\neq \emptyset\). In particular, \((p_{F_U}(F_V)= \rho^V_U) \cap O_p \neq \emptyset\), thus \(q \in N_{O_p}\).

So suppose that \(U \sqsubseteq V\), that is \(F_U \subseteq F_V\). Since \(q \in M_{O_p}(p)\), there is \(\gamma \in q\) and \(N\) such that for each \(n >N\) one has \(p_{F_U}(\gamma([n, \infty))) \cap O_p\neq \emptyset\). Let \(\eta\in q\) be a \(D\)--hierarchy ray connecting \(\rho^U_V\) with \(q\). There is \(n\) such that outside a ball of radius \(n\) around \(\eta(0)\), the Hausdorff distance between \(\gamma\) and \(\eta\) is at most \(H_D\). Thus there is a \((D+H_D)\)--hierarchy ray \(\eta'\in q\) such that \(\eta'(0) = \eta(0)\) and \(\eta'\) and \(\gamma\) coincide outside a ball around \(\eta'(0)\). In particular, this implies that \(\partial\pi_U(q) \cap O_p \neq \emptyset\), and hence that \(q\) is an element of \(\subseteq N_{O_p} (p)\).

\textit{For all \(p\) and \(O_p\), there exist \(O_p'\) such that \(N_{O_p'} (p) \subseteq M_{O_p} (p)\).}

We recall that there are constants \(\xi\) and \(E\) such that for each pair of transverse \(U\) and \(V\), one has \(\mathrm{diam}(\rho^U_V) \leq \xi\), and for each pair of nested \(U \sqsubseteq V\) and \(k\)--hierarchy path \(\gamma\) of \(CV\), one has that if \(\gamma \cap N_E(\rho^U_V) = \emptyset\), then \(\mathrm{diam} (\rho^V_U(\gamma) \leq E)\) (See Definition \ref{defn:space_with_distance_formula}).
Let \(s\) be much larger than \(\xi\) and \(E\), for instance \(s = 10 (2\xi + 2E + 20\delta)\). 

Note that \(O_p \cap CU\) is an open set in \(CU\), that we will still denote \(O_p\). 

Let \(\widehat{L}\) be the \(s\)--neighborhood in \(CU\) of the complementary of \(O_p\), that is  \(\widehat{L}= \bigcup_{x \in CU - O_p} B_s(x)\). Note that \(\widehat{L}\) is open in \(CU\).
We claim that \(L= \widehat{L} \cup \left(O_p \cap \partial_\infty CU\right)\) is an open neighborhood of \(p\) in \(\overline{CU}_\infty\). 
In order to show this, we will prove that \(L\) contains a neighborhood of each of its points. Let \(x\) in \(L\) and suppose that \(x \in CU\). 
Then, since \(\widehat{L}\) is open, there is an open neighborhood of \(x\) contained in \(\widehat{L}\) and thus in \(L\).
So let \(q\) be a point in \(L \cap \partial_\infty CU\). 
Since \(O_p\) is open, there is \(m\) and a representative \(\eta\) of \(q\) such that \(V_m(\eta)\subseteq O_p\), where \(V_m(\eta)\) is defined as in \ref{lem:nbhd_point_at_infnity}. 

We claim that there exists \(n\) large enough such that \(V_n(\eta) \subseteq L\), which implies the claim.
In what follows, we will often use the following fact. Let \(\alpha\) and \(\beta\) be \(k\)--quasi-geodesic starting at a point \(\beta(0)\). Suppose that there is a point \(b\) of \(\beta\) at distance more than \(n\) from \(\beta(0)\), such that \(d(b,\alpha) \leq s\). Then, inside a ball of radius approximately \(n-s\) around \(\beta(0)\), the Hausdorff distance between \(\alpha\) and \(\beta\) is at most \(H_k\).

Suppose that the claim does not hold, that is, suppose that there is \(y \in V_n(\eta)\) that does not belong to \(L\). 
Since \(L\) and \(O_p\) coincide on the boundary, we must have that \( y \in CU\). 
This means that there is a point \(z\in CU\) such that \(d(y,z) < s\) and \(z \not \in O_p\). In particular, \(z \not \in V_m(\eta)\).

Let \(\alpha\) be a  \(k\)--hierarchy path joining \(\eta(0)\) and \(y\), and \(\beta\) a \(k\)--hierarchy path joining \(\eta(0)\) and \(z\). 
Since \(y\) and \(z\) are \(s\)--near, \(\alpha\) and \(\beta\) are \(H_k\) near approximately inside a ball of radius \(d(\eta(0), y) - s\geq n-s\), since \(d(\eta(0), y) \geq n\). But, since \(y \in V_n (\eta)\), inside a ball of radius \(n- r\), \(\alpha\) is \(H_k\) near to \(\eta\). 
Choosing \(n\) large enough such that both \(n -s\) and \(n - r\) are much larger than \(m + r\), we get that in a ball of radius \(m +r \), \(\beta\) and \(\eta\) are \(2 H_k\) near. 
Thus \(z \in V_m(\eta)\), which is a contradiction.  

Thus \(L\) is an open set.
We claim that \(N_{L}(p) \subseteq M_{O_p}\). 
As before, let \(q \in N_{L}\) and suppose that \(q \in \partial_\infty CV\). If \(V = U\), the conclusion follows since \(L \subseteq O_p\). 
If \(V \pitchfork U\), then \(\rho^V_U \cap L \neq \emptyset\). Since \(L \subseteq O_p\) and \(\mathrm{diam}(\rho^V_U) \leq s\),  we have that \(\rho^V_U \subseteq O_p\). In particular, \(\pi_{F_U}(q) \subseteq O_p\) and thus \(q \in M_{O_p}\). 

Finally, suppose that \(U \sqsubseteq V\). The fact that \(q\) is an element of \(N_{L}\) implies that there is \(\gamma\in q\) with endpoint on \(\rho^U_V\) such that \(p_{F_U}\left(\gamma([E, \infty))\right) \cap L \neq \emptyset\). Since \(\mathrm{diam}\left(p_{F_U} \left(\gamma([E, \infty))\right)\right) \leq E\), we get that \(p_{F_U} \left(\gamma([E, \infty))\right)\subseteq O_p\). In particular, \(q \in M_{O_p}\).
\end{proof}

\subsection{Bowditch Boundary}

\begin{Def}
Let \(F\) be a connected graph. The \emph{combinatorial horoball} associated to \(F\) is the graph \(\Gamma (F)\) with vertices \(V(F) \times \NN\) and the following edges: 
\begin{rules}
\item For each \(v \in V(F)\) and \(n \in \NN\), there is an edge between \((v,n)\) and \((v, n+1)\).
\item For each pair of vertices \(v,w\in V(F)\) such that \(d_F(v,w) \leq 2^n\), there is an edge between \((v,n)\) and \((w,n)\).
\end{rules}
\end{Def}

It is easily seen that vertical rays in \(\Gamma(F)\) are  infinite geodesic rays, and it is not hard to see that they are the only ones. 
In particular, for each \(F\) the Gromov boundary of \(\Gamma(F)\) consists of a single point that we denote \(\ast_F\).

\begin{Def}
Let \(X\) be a hyperbolic geodesic space, and let \(\F\) be a family of subspaces of \(X\).
For each \(F \in \F\), let \(\Omega(F)\) be a connected approximation graph for \(F\), with constants chosen uniformly for all elements of \(\F\).
The \emph{Bowditch space} \(\Gamma(X, \F)\) (or simply \(\Gamma(X)\)) is defined as the space obtained from \(X\) attaching to it the combinatorial horoballs \(\Gamma(\Omega(F))\) under the identification \((x,0)\sim x\). 
\end{Def}

\begin{Def}
A space \(X\) is said to be \emph{hyperbolic relative to the family} \(\F\) if the Bowditch space \(\Gamma(X, \F)\) is Gromov hyperbolic.
\end{Def}

\begin{convention}
From now, let \(X\) be a proper geodesic space hyperbolic relative to a family \(\F\), where all the elements of \(\F\) have infinite diameter, and suppose that, in addition, \(X\) is hyperbolic.
\end{convention}
The request that the elements of \(\F\) to have infinite diameter is because, for an \(F\) of finite diameter, we would have  \(\partial F = \emptyset\), but \(\partial \Gamma(F) = \{\ast_F\}\). There are several ways to fix this, but they all boil down to considering only the elements of \(\F\) that have infinite diameter.

It is an easy consequence of \cite{SistoOnMetricRelative} that the family \(\F\) forms a factor system for \(X\) and, for any two different \(F_1, F_2 \in \F\), we have \(F_1 \not\precsim F_2 \not\precsim F_1\), where \(\precsim\) denote the coarse inclusion (Definition \ref{def:coarse_inclusion}).
It is not hard to see that this implies that \(\Gamma(\F) = \{\Gamma(F) \mid  F \in \F\}\) is a factor system for \(\Gamma(X)\). 
Note, moreover, that the cone-off of \(X\) with respect to \(\F\) is quasi-isometric to the cone-off of \(\Gamma(X)\) with respect to \(\Gamma(\F)\). 
Thus, we will identify the latter with \(CX\).

In particular, applying Theorem \ref{thm:main_result_bdry} we obtain the following decomposition of the Gromov boundary of \(\Gamma (X)\): 
\[\partial \Gamma (X) = \partial CX \cup_{F \in \F} \partial \Gamma (F) = \partial CX \cup_{F \in \F} \ast_F,\]
where the topology is as in Definition \ref{def:topology_HHS_bdry}. 
Let \(\partial_\F X\) be the quotient of \(\partial X = \partial CX \cup_{F \in F} \partial CF\) obtained collapsing each \(\partial CF\) to a point, equipped with the quotient topology. We claim that 
\[\partial_\F X= \partial \Gamma(X),\]
which amounts to saying that the Bowditch boundary of \(X\) can be easily described as the quotient of \(\partial X\) by a suitable set of subspaces. 

It is clear by the above description that there is a bijection \(\phi \colon \partial_\F X \rightarrow \partial \Gamma (X)\). 
We want now to show that \(\phi\) is a homeomorphism. Since \(X\) is proper, so  is \(\Gamma(X)\). In particular, \(\overline{X}\), \(X \cup \partial_\F (X)\) and \(\overline{\Gamma(X)}\) are compact, and so are \(\partial_\F X\) and \(\partial \Gamma(X)\). 
Thus, it suffices to show that \(\phi\) is continuous.

\begin{thm}\label{thm:Bowditch Boundary}
Let \(X\) be a proper hyperbolic space, which is hyperbolic relative to a family \(\F\), where all the elements of \(\F\) have infinite diameter. Let  \(CX\) be the cone-off of \(X\) with respect to \(\F\). Then \(\partial X = \partial CX \cup_{F \in \F} \partial F\) and the Bowditch boundary \(\partial\Gamma(X, \F)\) is obtained from \(\partial X\) collapsing each \(\partial F\) to a point.
\end{thm}
\begin{proof}
This is an immediate consequence of the Lemmas \ref{lem:nbhd_for_bijection_I} and \ref{lem:nbhd_for_bijection_II}, which describe the behavior of the fundamental neighborhoods under the map \(\phi\).
\end{proof}

\begin{notation}
We fix notations as follows: let \(\psi \colon \partial X\rightarrow \partial_\F X\) be the quotient map. Given a point \(p \in \partial CX\) and an open  set \(O_p \subseteq \overline{CX}\), that contains \(p\), we denote by \(N_{O_p}^X\) the neighborhood of \(p\) defined by \(O_p\) in \(\partial X\), and by \(N^{\Gamma(X)}_{O_p}\) the neighborhood in \(\partial \Gamma(X)\). 
\end{notation}

\begin{lemma}\label{lem:nbhd_for_bijection_I}
Let \(p \in \partial CX\). Then for each open set  \(O_p \in \overline{CX}\) that contains \(p\), one has \( \phi \circ \psi \left( N_{O_p}^X\right) = N_{O_p}^{\Gamma(X)}\) and \(\psi^{-1} \circ \phi^{-1}  \left(N_{O_p}^{\Gamma (X)} \right) = N_{O_p}^X\).
\begin{proof}
It is clear that \(\phi \circ \psi\) is a bijection on \(\partial CX\). Thus it is easy to see that a point \(q \in \partial CX\) belongs to \(N_{O_p}^X\) if and only if \(\phi(\psi(x))\) belongs to \(N^{\Gamma(X)}_{O_p}\). 
So, consider  the point \(\ast_F \in \partial \Gamma(F)\) for some \(F \in \F\). We have that \(\ast_F\) belongs to \(N_{O_p}^{\Gamma(X)}\) if \(\rho^F_X\) intersects \(O_p\), where \(\rho^F_X\) is defined to be the projection of \(\Gamma(F)\) on \(CX\). But this projection coincides with the projection of \(CF\) on \(CX\). 
Thus \(\ast_F\) belongs to \(N_{O_p}^{\Gamma(X)}\) if and only if \(\partial CF= \psi^{-1}\circ \phi^{-1} (\ast_F)\) belongs to \(N_{O_p}^X\). 
\end{proof}
\end{lemma}

\begin{lemma}\label{lem:nbhd_for_bijection_II}
Let \(F \in \F\). For each \(O_F \in \overline{\Gamma(F)}\) there is \(O_F' \in \overline{CF}\) such that \( N^X_{O_F'} \subseteq \psi^{-1}\circ\phi^{-1} \left(N_{O_F}^{\Gamma(X)} \right) \) and \(\partial CF \subseteq N^X_{O_F'}\).
\begin{proof}
Let \(O_F\) be open in \(\overline{\Gamma(F)}\). Then, there is a representative \(\eta\) of \(\ast_F\) and a number \(n\) such that \(V_n(\eta) \subseteq O_F\), where \(V_n(\eta)\) is defined as in Lemma \ref{lem:nbhd_point_at_infnity}.
Let \(\eta(0)\) be the starting point of \(\eta\). Without loss of generality we can assume that \(\eta(0)\) is a point of \(CF\). 

Let \(V_n(\eta)^c\) be the complement of \(V_n(\eta)\) in \(\Gamma(F)\). It is easy to see that \(V_n(\eta)^c\) is contained in a closed ball of finite radius around \(\eta(0)\). Let \(K\) be such a ball. Thus we have \(K^c \subseteq O_F\). 

The image of \(K\) in \(CF\) has also finite diameter. Thus we have that \(K^c\) is an open of \(\overline{CF}\). 
Since \(K\) has finite diameter, it is easy to see that \(\partial  CF \subseteq N_{K^c}^X\). Moreover, since by construction \(K^c \subseteq O_F\), we have that for each point \(q \in \partial CX\), if the projection of \(q\) on \(CF\) intersects \(K^c\), then it intersects \(O_F\), which shows \(N^X_{O_F'} \subseteq \psi^{-1}\circ\phi^{-1} \left(N_{O_F}^{\Gamma(X)} \right)\).
\end{proof}
\end{lemma}

\end{document}